\pdfoutput=1
\documentclass[11pt,twoside]{article}

\usepackage{amsmath,bm}
\usepackage{amsthm} 
\usepackage[colorlinks=true,
linkcolor=blue,
urlcolor=blue,
citecolor=blue]{hyperref}
\usepackage{tikz}
\usetikzlibrary{arrows,shapes.callouts, positioning, arrows.meta}
\tikzset{
  box/.style={
    rectangle, rounded corners,
    draw=black, thick,
    align=center,
    minimum height=1cm,
    minimum width=4cm
  },
  arrow/.style={
    ->, thick
  }
}
\usepackage{bookmark}
\usepackage{amssymb}
\usepackage{amsfonts}
\usepackage{titlesec}
\usepackage{verbatim}
\usepackage{fullpage}
\usepackage{color,xcolor}
\usepackage{mathrsfs}
\usepackage[normalem]{ulem}
\usepackage{longtable} 
\usepackage{array}

\usepackage[capitalize]{cleveref}

\usepackage[round,sort]{natbib}
\usepackage{dsfont}
\usepackage{graphicx}
\usepackage{caption}
\usepackage{subcaption}
\usepackage{float}
\usepackage{mathtools}
\usepackage{etoolbox}
\usepackage{thmtools}
\usepackage{enumitem}
\usepackage[textsize=tiny]{todonotes}
\usepackage{booktabs}
\usepackage{caption}
\usepackage{multirow}
\usepackage{esvect}
\usepackage[utf8]{inputenc}

\newtheorem{theorem}{Theorem}
\newtheorem{lemma}{Lemma}
\newtheorem{definition}{Definition}

\newtheorem{corollary}{Corollary}
\newtheorem{assumption}{Assumption}
\declaretheoremstyle[headfont=\bf,bodyfont=\normalfont]{ex}
\declaretheorem[style=ex]{example}
\declaretheoremstyle[bodyfont=\normalfont]{rm}
\declaretheorem[style=rm]{remark}

\DeclareMathOperator*{\argmin}{arg\,min}

\newcommand {\hs} {\text{HS}}
\newcommand {\localcomp} {{\mathcal{R}_n}}
\newcommand {\emlocalcomp} {{\widehat{\mathcal{R}}_n}}

\def\bK{\mathbf{K}}
\def\bI{\mathbf{I}}
\def\bX{\mathbf{X}}
\def\bZ{\mathbf{Z}}
\def\bW{\mathbf{W}}
\def\bY{\mathbf{Y}}  
 
\def\be{\bm{e}}

\def\NN{\mathbb{N}}
\def\OO{\mathbb{O}}
\def\RR{\mathbb{R}}
\def\PP{\mathbb{P}}
\def\EE{\mathbb{E}}
\def\SS{\mathbb{S}}

\def\cK{\mathcal{K}}
\def\cO{\mathcal{O}}

\def\cH{\mathcal H}
\def\cD{\mathcal D}
\def\cE{\mathcal E}
\def\cR{\mathcal R}
\def\cZ{\mathcal Z}

\def\cF{\mathcal F}
\def\cT{\mathcal T}
\def\cL{\mathcal L}
\def\cB{\mathcal B}

\def\cX{\mathcal{X}}

\def\rI{\mathrm{I}}
\def\rII{\mathrm{II}}

\def\wh{\widehat}
\def\wt{\widetilde}

\def\T{\top}
\def\F{\text{F}}
\def\i{\infty}
\def\tr{{\rm Tr}}
\def\op{{\rm op}}

\def\Var{{\rm Var}}
\def\d{~{\rm d}}
\def\diag{{\rm diag}}
\def\sw{\Sigma_W}
\def\sz{\Sigma_Z}

\def\fh{f_{\cH}}

\def\Cov{{\rm Cov}}

\def\l{\left}
\def\r{\right}
\def\b{\big}
\def\B{\Big}

\newcommand{\normmm}{{\vert\kern-0.25ex\vert\kern-0.25ex\vert}}
\newcommand{\bignormmm}{{\big\vert\kern-0.25ex\big\vert\kern-0.25ex\big\vert}}
\newcommand{\Bignormmm}{{\Big\vert\kern-0.25ex\Big\vert\kern-0.25ex\Big\vert}}

\usepackage{threeparttable} 
\long\def\comment#1{}

\newcommand{\ydsout}[1]{}

\usepackage{stackrel}

\newcommand{\rank}{\ensuremath{{\rm rank}}}

\newcommand{\RKHS}{\ensuremath{\mathscr{F}}}

\newcommand{\StateSpgoodh}[1]{\ensuremath{}}

\newcommand{\StateSpbadh}[1]{\ensuremath{}}

\newcommand{\Data}{\ensuremath{\mathcal{D}}}

\usepackage{upgreek} \newcommand{\distr}{\ensuremath{\upmu}}

\newcommand{\distrdata}{\ensuremath{\bar{\distr}}}

\newcommand{\sdistrdata}{\ensuremath{\sdistr_{\Data}}}

\newcommand{\sdistr}{\ensuremath{\xi}}

\DeclarePairedDelimiterX{\anglep}[1]{(}{)}{#1}
\makeatletter
\newcommand{\@spanstar}[1]{{\rm span}\anglep*{#1}}
\newcommand{\@spannostar}[2][]{{\rm span}\anglep[#1]{#2}}
\newcommand{\Span}{\@ifstar\@spanstar\@spannostar}
\makeatother

\DeclarePairedDelimiterX{\dfun}[2]{(}{)}{#1 \;\delimsize\|\; #2}
\makeatletter
\newcommand{\@trunstar}[2]{\chi^2\dfun*{#1}{#2}}
\newcommand{\@trunnostar}[3][]{\chi^2\dfun[#1]{#2}{#3}}
\newcommand{\chisq}{\@ifstar\@trunstar\@trunnostar}
\makeatother

\DeclarePairedDelimiterX{\inprod}[2]{\langle}{\rangle}{#1, \, #2}

\DeclarePairedDelimiterX{\kulldiv}[2]{(}{)}{#1\;\delimsize\|\;#2}
\makeatletter
\newcommand{\@kullstar}[2]{D_{\text{KL}}\kulldiv*{#1}{#2}}
\newcommand{\@kullnostar}[3][]{D_{\text{KL}}\kulldiv[#1]{#2}{#3}}
\newcommand{\kull}{\@ifstar\@kullstar\@kullnostar}
\makeatother

\makeatletter
\newcommand{\@hilinstar}[2]{\inprod*{#1}{#2}_{\RKHS}}
\newcommand{\@hilinnostar}[3][]{\inprod[#1]{#2}{#3}_{\RKHS}}
\newcommand{\hilin}{\@ifstar\@hilinstar\@hilinnostar}
\makeatother

\makeatletter
\newcommand{\@mudatainstar}[2]{\inprod*{#1}{#2}_{\distrdata}}
\newcommand{\@mudatainnostar}[3][]{\inprod[#1]{#2}{#3}_{\distrdata}}
\newcommand{\mudatain}{\@ifstar\@mudatainstar\@mudatainnostar}
\makeatother

\makeatletter
\newcommand{\@mudatahinstar}[3]{\inprod*{#2}{#3}_{\distrdata}}
\newcommand{\@mudatahinnostar}[4][]{\inprod[#1]{#3}{#4}_{\distrdata}}
\newcommand{\mudatahin}{\@ifstar\@mudatahinstar\@mudatahinnostar}
\makeatother

\DeclarePairedDelimiterX{\defabs}[1]{|}{|}{#1}
\makeatletter
\newcommand{\@absstar}[1]{\defabs*{#1}}
\newcommand{\@absnostar}[2][]{\defabs[#1]{#2}}
\newcommand{\abs}{\@ifstar\@absstar\@absnostar}
\makeatother

\DeclarePairedDelimiterX{\norm}[1]{\|}{\|}{#1}
\makeatletter
\newcommand{\@normstar}[1]{\norm*{#1}_{\RKHS}}
\newcommand{\@normnostar}[2][]{\norm[#1]{#2}_{\RKHS}}
\newcommand{\hilnorm}{\@ifstar\@normstar\@normnostar}
\makeatother

\DeclareFontFamily{U}{matha}{\hyphenchar\font45}
\DeclareFontShape{U}{matha}{m}{n}{
	<-6> matha5 <6-7> matha6 <7-8> matha7
	<8-9> matha8 <9-10> matha9
	<10-12> matha10 <12-> matha12
}{}
\DeclareSymbolFont{matha}{U}{matha}{m}{n}

\DeclareFontFamily{U}{mathx}{\hyphenchar\font45}
\DeclareFontShape{U}{mathx}{m}{n}{
	<-6> mathx5 <6-7> mathx6 <7-8> mathx7
	<8-9> mathx8 <9-10> mathx9
	<10-12> mathx10 <12-> mathx12
}{}
\DeclareSymbolFont{mathx}{U}{mathx}{m}{n}

\DeclareMathDelimiter{\vvvert} {0}{matha}{"7E}{mathx}{"17}%

\DeclarePairedDelimiterX{\opnorm}[1]{\vvvert}{\vvvert}{#1}

\makeatletter
\newcommand{\@hilopnormstar}[1]{\opnorm*{#1}_{\RKHS}}
\newcommand{\@hilopnormnostar}[2][]{\opnorm[#1]{#2}_{\RKHS}}
\newcommand{\hilopnorm}{\@ifstar\@hilopnormstar\@hilopnormnostar}
\makeatother

\makeatletter
\newcommand{\@muopnormstar}[1]{\opnorm*{#1}_{\distr}}
\newcommand{\@muopnormnostar}[2][]{\opnorm[#1]{#2}_{\distr}}
\newcommand{\muopnorm}{\@ifstar\@muopnormstar\@muopnormnostar}
\makeatother

\makeatletter
\newcommand{\@mudataopnormstar}[1]{\opnorm*{#1}_{\distrdata}}
\newcommand{\@mudataopnormnostar}[2][]{\opnorm[#1]{#2}_{\distrdata}}
\newcommand{\mudataopnorm}{\@ifstar\@mudataopnormstar\@mudataopnormnostar}
\makeatother

\makeatletter
\newcommand{\@supnormstar}[1]{\norm*{#1}_{\infty}}
\newcommand{\@supnormnostar}[2][]{\norm[#1]{#2}_{\infty}}
\newcommand{\supnorm}{\@ifstar\@supnormstar\@supnormnostar}
\makeatother

\makeatletter
\newcommand{\@munormstar}[1]{\norm*{#1}_{\distr}}
\newcommand{\@munormnostar}[2][]{\norm[#1]{#2}_{\distr}}
\newcommand{\munorm}{\@ifstar\@munormstar\@munormnostar}
\makeatother

\makeatletter
\newcommand{\@mudatanormstar}[1]{\norm*{#1}_{\distrdata}}
\newcommand{\@mudatanormnostar}[2][]{\norm[#1]{#2}_{\distrdata}}
\newcommand{\mudatanorm}{\@ifstar\@mudatanormstar\@mudatanormnostar}
\makeatother

\makeatletter
\newcommand{\@xidatanormstar}[1]{\norm*{#1}_{\sdistrdata}}
\newcommand{\@xidatanormnostar}[2][]{\norm[#1]{#2}_{\sdistrdata}}
\newcommand{\xidatanorm}{\@ifstar\@xidatanormstar\@xidatanormnostar}
\makeatother

\makeatletter
\newcommand{\@distrnormstar}[2]{\norm*{#1}_{#2}}
\newcommand{\@distrnormnostar}[3][]{\norm[#1]{#2}_{#3}}
\newcommand{\distrnorm}{\@ifstar\@distrnormstar\@distrnormnostar}
\makeatother

\makeatletter
\newcommand{\@psinormstar}[2]{\norm*{#2}_{\psi_{#1}}}
\newcommand{\@psinormnostar}[3][]{\norm[#1]{#3}_{\psi_{#2}}}
\newcommand{\psinorm}{\@ifstar\@psinormstar\@psinormnostar}
\makeatother

\newcommand{\featureh}[1]{\ensuremath{\phi}}

\newcommand{\Lip}{\ensuremath{L}}

\newcommand{\Lipf}[1]{\ensuremath{\Lip_{f}}}

\newenvironment{carlist}
{\begin{list}{$\bullet$}
		{\setlength{\topsep}{0.1in} \setlength{\partopsep}{0in}
			\setlength{\parsep}{0.1in} \setlength{\itemsep}{\parskip}
			\setlength{\leftmargin}{0.15in} \setlength{\rightmargin}{0.08in}
			\setlength{\listparindent}{0in} \setlength{\labelwidth}{0.08in}
			\setlength{\labelsep}{0.1in} \setlength{\itemindent}{0in}}}
	{\end{list}}

\newcommand{\bcar}{\begin{carlist}}
	\newcommand{\ecar}{\end{carlist}}

 \title{Kernel Ridge Regression with Predicted Feature Inputs and Applications to Factor-Based Nonparametric Regression}

\author{Xin Bing\thanks{Department of 
Statistical Sciences, University of Toronto, Toronto, Canada. E-mail: \texttt{xin.bing@utoronto.ca}} 
\hspace{1cm}Xin He\thanks{School of Statistics and Data Science, Shanghai University of Finance and Economics, Shanghai, China. E-mail: \texttt{he.xin17@mail.shufe.edu.cn}}
\hspace{1cm}Chao Wang\thanks{School of Statistics and Data Science, Shanghai University of Finance and Economics, Shanghai, China. E-mail: \texttt{wang.chao@163.sufe.edu.cn}}.}

\begin{document}

\maketitle


\begin{abstract}
Kernel methods, particularly kernel ridge regression (KRR), are time-proven, powerful nonparametric regression techniques known for their rich capacity, analytical simplicity, and computational tractability. The analysis of their predictive performance has received continuous attention for more than two decades. However, in many modern regression problems where the feature inputs used in KRR cannot be directly observed and must instead be inferred from other measurements, the theoretical foundations of KRR remain largely unexplored.
In this paper, we introduce a novel approach for analyzing KRR with predicted feature inputs. Our framework is not only essential for handling predicted feature inputs—enabling us to derive risk bounds without imposing any assumptions on the error of the predicted features—but also strengthens existing analyses in the classical setting by allowing arbitrary model misspecification, requiring weaker conditions under the squared loss, particularly allowing both an unbounded response and an unbounded function class, and being flexible enough to accommodate other convex loss functions. We apply our general theory to factor-based nonparametric regression models and establish the minimax optimality of KRR when the feature inputs are predicted using principal component analysis.
Our theoretical findings are further corroborated by simulation studies and real-data analyses using pretrained LLM embeddings for the downstream prediction task.
\end{abstract}




\section{Introduction}\label{sec:intro}

Regression is one of the most fundamental problems in statistics and machine learning where the goal is to construct a prediction function $f:\cZ \to \RR$ based on $n$ independent pairs $(Y_i, Z_i) \in (\RR, \cZ)$ for $i\in [n]:=\{1,\ldots,n\}$ such that for a new pair $(Y, Z)$, the prediction $f(Z)$ is close to $Y$. Among existing regression approaches, Kernel Ridge Regression (KRR) is a cornerstone of nonparametric regression due to its flexibility, analytical simplicity, and computational efficiency \citep{steinwart2008support,cucker2002mathematical}, making it widely applicable across diverse fields, including finance, bioinformatics, signal processing, image recognition, natural language processing, and climate modeling. In this paper, we study the prediction of KRR when the regression features $Z_i$'s are not directly observed and must be predicted from other measurements.

Let $K: {\cZ \times \cZ}\to \RR$  be a pre-specified kernel function. In the classical setting where  $(Y_i,Z_i)$ for $i\in [n]$ are observed, KRR searches for the best predictor over the function class $\cH_K$, the Reproducing Kernel Hilbert Space (RKHS) induced by $K$, by minimizing the penalized least squares loss:
\begin{align}\label{def_KRR}
\wt f =\argmin_{f\in \cH_K}  \left\{ \frac{1}{n} \sum_{i=1}^n\bigl(Y_i- f(Z_i)\bigl)^2
+\lambda\|f\|_{K}^2 \right\},
\end{align}
where $\lambda>0$ is some regularization parameter and $\|\cdot\|_K$ is the endowed norm of $\cH_K$. 
It is well known that the RKHS induced by a universal kernel, such as the Gaussian or Laplacian kernel, is a large function space, as any continuous function can be approximated arbitrarily well by an intermediate function in the RKHS under the infinity norm \citep{micchelli2006universal}.  More importantly, the representer theorem \citep{kimeldorf1971some} ensures that $\wt f$ in \eqref{def_KRR} admits a closed-form solution and can be computed effectively.  Theoretical properties of $\wt f$ have been constantly studied over the past two decades, see, for instance, \cite{caponnetto2007optimal, smale2007learning,steinwart2008support, fischer2020sobolev,rudi2017generalization}, to just name a few.  

However, in many modern regression problems across machine learning, natural language processing, healthcare, and finance, the feature inputs $Z_1,\ldots, Z_n$ in \eqref{def_KRR} are not directly observed and must be predicted from other measurable features. Consequently, regression such as KRR is performed by regressing the response onto the predicted feature inputs. Below, we present prominent examples, including well-known feature extraction and representation learning techniques in large language models and deep learning. 

\begin{example}[Dimension reduction via PCA]
    KRR, or more broadly, nonparametric regression approaches, are particularly appealing due to their flexibility in approximating any form of the regression function when the feature dimension is small and the sample size is large. However, in many modern applications, the number of features is often comparable to, or even greatly exceeds, the sample size. Classical non-parametric approaches including KRR are affected by the curse of dimensionality. To address the high-dimensionality issue, one common approach is to first find a low-dimensional representation of the observed high-dimensional features and then regress the response onto the obtained low-dimensional representation. 
    Principal Component Analysis (PCA) is the most commonly used method to construct a small number of linear combinations of the original features, known as the  Principal Components (PCs) \citep{Hotelling1957}. The prediction is then made by regressing the response onto the obtained PCs. When this regression step is performed using ordinary least squares, the approach is known as Principal Component Regression (PCR) \citep{stock2002forecasting,bing2021prediction}, which is fully justified under the following factor-based linear regression model:
    \begin{align}\label{model_Y_linear}
        Y = Z^\T \beta + \epsilon,\\\label{model_X}
        X = AZ + W.
    \end{align}
    Here both $X\in \RR^p$ and $Y\in \RR$ are observable while $Z\in \RR^r$, with $r \ll p$, are some unobservable, latent factors. The $p\times r$ loading matrix  $A$ is deterministic but unknown, the vector $\beta$ contains the linear coefficients of the factor $Z$, and both $\epsilon \in \RR$ and $W \in \RR^p$ represent additive errors. The leading PCs in PCR can be regarded as a linear prediction of the latent factor $Z$. While the linear factor model in \eqref{model_X} is reasonable and widely adopted in the literature \citep{Anderson2003,ChamberlainRothschild1983,LawleyMaxwell,Bai-factor-model-03}, PCR often yields unsatisfactory predictive performance due to its reliance on linear regression in the low-dimensional space. As a running example of our general theory, we adopt KRR to capture more complex nonlinear relationships between $Y$ and  $Z$ in place of \eqref{model_Y_linear}, while still using \eqref{model_X} to model the dependence between $X$ and $Z$. 
\end{example}

\begin{example}[Feature extraction via autoencoder] 
Unlike PCA, an autoencoder is a powerful alternative in machine learning and data science for constructing a nonlinear low-dimensional representation of high-dimensional features \citep{Goodfellow-et-al-2016}. Specifically, an autoencoder learns an encoding function $E$
that maps the original features $X\in \cX\subseteq \RR^p$ to its latent representation in $\cZ\subseteq \RR^r$
while simultaneously learning a decoding function $D: \cZ \to \cX$  to reconstruct the original features from the encoded representation \citep{bank2023machine}. 
Both the encoder and decoder are typically parameterized by using deep neural networks as $E_\theta$ and $D_\vartheta$, respectively.  
Given  a metric $ d: \cX \times \cX: \to \RR^+$, for instance,  $d(x,x')= \|x-x'\|_2^2$ for any $x,x'\in \cX$, the encoder and decoder are learned by solving the optimization problem:
\[
(\wh \theta , \wh \vartheta )   : =  \argmin _{ \theta , \vartheta} \frac{1}{n}
\sum_{i=1}^n d\b((D _ \vartheta \circ E_\theta)(X_i), X_i  \b) . 
\] 
The trained autoencoder constructs a non-linear low-dimensional representation of the original features, given by $E_{\wh \theta}(X)$, which are then utilized in various downstream supervised learning tasks \citep{Berahmand2024Autoencoders,Chen2023Machine}, with applications found in finance, healthcare   and climate science.
\end{example}

\begin{example}[Representation learning in large language models]  
    Recent advancements in natural language processing have demonstrated the effectiveness of large language models (LLMs), such as Transformer-based architectures, in learning meaningful representations of text data \citep{vaswani2017attention, devlin2019bert}. These models convert text input into numerical vector representations, commonly referred to as word embeddings or contextualized embeddings, which capture semantic and syntactic properties of words and phrases.
    Formally, a large language model learns an embedding function  $E: \cX \to \cZ$, where $\cX$ represents the input text space (e.g., tokenized sequences) and $\cZ$ is the latent representation space. This includes the traditional static word embeddings (e.g., Word2Vec or GloVe) where the same token is mapped to a unique representation. Recently, modern transformer-based models produce contextualized embeddings  based on the surrounding words. Given a sequence of tokens  $X = (x_1, x_2, \dots, x_T)$, a trained transformer model computes the embeddings as: 
    \[
    E_{\wh \theta}(X) = (E_{\wh\theta}(x_1), E_{\wh\theta}(x_2), \ldots, E_{\wh\theta}(x_T)),
    \]
    where  $E_{\wh \theta}$ denotes the parametrized embedding function, {\em pretrained} on large-scale corpora using unsupervised learning objectives, such as masked language modeling \citep{devlin2019bert}. The embeddings serve as powerful feature representations for various downstream supervised learning tasks.
    In regression applications, the embeddings of the input text together with their corresponding responses are used  to train a regression function $\wh f: \cZ \to \RR$ with the final predictor  given by:
    $
       \wh f \circ  E_{\wh \theta} : \cX \to \RR,
    $
    where $\circ$ denotes the composition of two functions.
    This approach is widely adopted in natural language processing  tasks such as sentiment analysis with continuous sentiment scores, automated readability assessment, and financial text-based forecasting, where pretrained embeddings are used to significantly improve prediction performance \citep{raffel2020exploring}. Our real data analysis in \cref{sec_real_data} provides two concrete examples of this approach.
\end{example}

Motivated by the aforementioned examples, in this paper we study  the KRR in \eqref{def_KRR} when the feature inputs  $Z_i$'s need to be predicted from other measurable features. Concretely, consider $n$ i.i.d. training data $\cD = \{(Y_i, X_i)\}_{i=1}^n \in (\RR, \cX)^n$  with $\cX \subseteq \RR^p$ and the response $Y_i$ generated according to the following non-parametric regression model
\begin{equation}\label{model}
   Y=f^*(Z)+\epsilon,
\end{equation}
where $Z\in \cZ\subseteq \RR^r$  is some random latent factor with some $r\in \NN$, $f^*:\cZ\to \RR$ is the true regression function, and $\epsilon$ is some regression error with zero mean and finite second moment $\sigma^2 < \i$.  Let $\wh g: \cX \to \cZ$ be any generic measurable predictor of $Z$, constructed independently of $\cD$. The KRR predictor $\wh f$ is obtained as in \eqref{def_KRR}  with  $Z_i$ replaced by its  predicted value   $\wh g(X_i)$ for $i\in [n]:=\{1,\ldots, n\}$.  

The final predictor of KRR using $\wh g$-predicted inputs is 
$
\wh f\circ \wh g: \cX \to \RR,
$
with  its  excess risk defined as 
\begin{align} \label{def_excess_risk}
     \cE(\wh f\circ\wh g)  ~ := ~  \EE\bigl[Y- (\wh f \circ\wh g)(X)\bigr]^2- \sigma^2 ~ =~    \EE\bigl[f^*(Z) - (\wh f \circ\wh g)(X)\bigr]^2,
\end{align} 
where $(Y,X)$ is a new random pair following the same model as $\cD$. Our main goal is to provide a general treatment for analyzing $\cE(\wh f\circ\wh g)$ that applies to any generic predictor  $\wh g$. 
As detailed below, existing analyses of KRR in the classical setting cannot be used to study excess risk when the feature inputs are subject to prediction errors.
The framework we develop is indispensable for handling predicted features and also strengthens classical results: it allows arbitrary model misspecification, requires weaker conditions under the squared loss, and flexibly accommodates other convex loss functions in \eqref{def_KRR}. We start by reviewing existing approaches for analyzing KRR in the classical setting. 

\subsection{Existing approaches for analyzing KRR}\label{sec_intro_related}
In the classical setting where the features $Z_i$'s in \eqref{def_KRR} are directly observed, existing techniques for bounding the excess risk of $\wt f$ under model \eqref{model} mainly fall into two streams.

The first proof strategy relies on the integral operator approach (see, e.g., \cite{smale2007learning, caponnetto2007optimal, blanchard2018optimal, fischer2020sobolev,rudi2017generalization}, and references therein). Denote by  $\rho$
  the probability measure of the feature $Z\in \cZ$ and by $
    \cL^2(\rho)= \{f: \int_\cZ f^2(z)\d\rho (z)<\infty \}
 $  its induced space of square-integrable functions, equipped with the inner product  $\langle \cdot, \cdot \rangle_\rho$ and norm $\|\cdot\|_{\rho}$. For any $z \in \cZ$, by writing $K_z:=K(z, \cdot)\in \cH_K$, the integral operator is $L_K: \cL^2(\rho)\to \cL^2(\rho)$, defined as
 \begin{align*}
    L_K: f \mapsto \int_{\mathcal{Z}} K_z f(z)\d\rho(z).
\end{align*} 
 Let  $\{\mu_j\}_{j\ge 1}$ be the eigenvalues of $L_K$ and $\{\phi_j\}_{j\ge 1}$ be the corresponding eigenfunctions, assuming their existence.
 When the regression function $f^*$ belongs to $\cH_K$ and the eigenvalues $\mu_j$'s  satisfy the following polynomial decay condition
 \begin{equation}\label{cond_poly}
   c j^{-1/\alpha} \le  \mu_j \le  C j^{-1/\alpha},\qquad \text{for all $j\ge 1$ and some $\alpha\in (0,1)$}, 
 \end{equation}
\cite{caponnetto2007optimal} established the minimax optimal rate of $\|\wt f - f^*\|_\rho^2$, which depends on the decaying rate $\alpha$. While the integral operator approach allows for sharp rate analysis of KRR, it oftentimes has limitations in handling model misspecification, and the required conditions on $\mu_j$'s could be difficult to verify.  Without assuming \eqref{cond_poly}, \cite{smale2007learning} derives upper bounds of $\|\wt f - f^*\|_\rho^2$ but their analysis still requires $f^*\in \cH_K$ along with some additional smoothness condition on $f^*$.
Allowing $f^*\notin \cH_K$ is addressed by \cite{rudi2017generalization} in the context of learning with random features where the authors resort to a so-called effective dimension condition $\sum_j \mu_j/(\mu_j + \delta) \le C \delta^{-\alpha}$, for all $\delta >0$ and some $\alpha \in [0,1]$, which is related to but weaker than  \eqref{cond_poly}.
Alternatively, \cite{fischer2020sobolev,pmlr-v202-zhang23x} handle $f^*\notin \cH_K$ by relying on the smoothness condition $f^* = L_K^r g^* $ for some $g^*\in \cL^2(\rho)$ and some $r\in (0,1/2]$ as well as the upper bound requirement in condition \eqref{cond_poly}.
 
  The other proof strategy for analyzing KRR adopts the empirical risk minimization perspective and employs probabilistic tools from empirical process theory to establish excess risk bounds. See, e.g., \cite{bartlett2005local, mendelson2010regularization,steinwart2009optimal}; more recently, \cite{ma2022optimally, duan2024optimal}, and references therein. Inherited from the classical learning theory,  this proof strategy has the potential to deal with model misspecification and can be easily applied to KRR in \eqref{def_KRR} in which the squared loss is replaced by other convex loss functions \citep{Eberts}. However, the existing works typically assume boundedness conditions on the response variable and the function class, or restrictive conditions on the eigen-decomposition of  $L_K$. For instance, the analyses in \cite{mendelson2010regularization,ma2022optimally,duan2024optimal} require 
  the eigenfunctions $\{\phi_j\}_{j\ge 1}$ of $L_K$ to be uniformly bounded, a useful condition for obtaining improved $\|\cdot\|_\infty$ bound but is not always satisfied even for the most popular kernel functions \citep{zhou2002covering}. On the other hand, the authors of \cite{bartlett2005local}  derive prediction risk bounds for general empirical risk minimization based on {\em local Rademacher complexity}, and applied them to analyze $\wt f_{bd}$, which is computed from \eqref{def_KRR} with $\lambda = 0$ but constrained to the unit ball $\cB_{\cH} := \{f: \|f\|_K \le 1\}$. A key step in their proof is  connecting the local Rademacher complexity of $\cB_{\cH}$ with the {\em kernel complexity function}
\begin{equation}\label{kernel_complexity}
          R(\delta) = \biggl({1\over n}\sum_{j=1}^\i \min\{\delta,\mu_j\}\biggr)^{1/2},\qquad \forall~  \delta \ge 0,
\end{equation}
   using the result proved in \cite{mendelson2002geometric}.
   The kernel complexity function depends not only on  $\cH_K$ but also on the probability measure $\rho$ of $Z$ through the integral operator $L_K$, and is  commonly used to characterize the complexity of $\cH_K$.  Indeed, it is closely related with the effective dimension \citep{caponnetto2007optimal,rudi2017generalization}, covering numbers \citep{cucker2002mathematical} and related capacity measures \citep{steinwart2008support}. Let  $\delta_n>0$ be the fixed point of $R(\delta)$ such that $R(\delta_n) = \delta_n$. 
   The risk bound of $\wt f_{bd}$ in \cite{bartlett2005local} reads
  \begin{equation}\label{rate_classic}
      \PP\left\{\|\wt f_{bd} - f^*\|_\rho^2 
      ~ \lesssim~  \delta_n + {t \over n}\right\} \ge 1-e^{-t},\qquad \forall~  t\ge 0.
  \end{equation}
  However, the analysis of \cite{bartlett2005local}  requires the response variable $Y$ and $\cH_K$ to be bounded, as well as $f^*\in \cH_K$. The case $f^*\notin \cH_K$ is considered in \cite{steinwart2009optimal} where the authors derive an oracle inequality for the prediction risk of a truncated version of the KRR predictor $\wt f$. But their analysis requires $Y$ to be bounded, condition \eqref{cond_poly} and the approximation error to satisfy $\inf_{f\in \cH_K}\{\|f^*-f\|_\rho^2 + \lambda\|f\|_K^2\} \le C\lambda^{\beta}$ for some $\beta \in (0,1]$ and $\lambda>0$. The last requirement turns out be related to the aforementioned smoothness condition $f^* = L_K^r g^*$ with $r\in (0,1/2]$.

  Summarizing, even in the classical setting where the feature inputs are observed, existing analyses of KRR require either that the true regression function belongs to the specified RKHS, that both the RKHS and the response $Y$ are bounded, or that certain conditions on the eigenvalues or eigenfunctions of the RKHS hold. As explained in the next section, our new framework of analyzing KRR removes such restrictions under the squared loss in \eqref{def_KRR}, while remaining applicable to the analysis of general convex loss functions.

  More substantially, when the feature $Z_i$'s in \eqref{def_KRR} must be predicted by $\wh g(X_i)$, neither of the two proof strategies above applies. This is because both the integral operator $L_K$ and the kernel complexity function $R(\delta)$ depend on the probability measure $\rho$ of  $Z$. When regressing $Y$ on the predicted feature $\wh g(X)$, each strategy requires a different integral operator $L_{x,K}$ and hence a different kernel complexity function, both depending on the probability measure $\rho_x$ of $\wh g(X)$. Let $\{\mu_{x,j}\}_{j\ge 1}$ be the eigenvalues of $L_{x,K}$, and define the corresponding kernel complexity function as
  \[
     R_x(\delta) = \biggl({1\over n}\sum_{j=1}^\i \min\{\delta, \mu_{x,j}\}\biggr)^{1/2},\qquad \forall~ \delta \ge 0.
  \]
  for $\delta \ge 0$.
  Establishing a direct relationship between either $R_x(\delta)$ and $R(\delta)$, or between $\{\mu_j\}_{j\ge 1}$ and $\{\mu_{x,j}\}_{j\ge 1}$, is generally intractable without imposing strong assumptions on the relationship between $\rho_x$ and $\rho$. Consequently, handling predicted feature inputs requires a new proof strategy.

\subsection{Our contributions}\label{sec_intro_contri}

We summarize our main contributions in this section.

\paragraph*{Non-asymptotic excess risk bounds of KRR with predicted feature inputs.}  Our first contribution is to derive new non-asymptotic upper bounds on the excess risk of the predictor $\wh f\circ \wh g$. In particular, our risk bounds are valid for any generic predictor $\wh g$ {\em without} imposing any assumptions on the proximity of $\wh g(X)$  to $Z$. Our results also allow for arbitrary model misspecification and require neither the response nor the RKHS to be bounded. 
To allow for model misspecification, we assume in \cref{ass_f_H} of \cref{sec_theory_KRR_decomp} only the existence of $\fh$ as the $\cL^2(\rho)$ projection of $f^*$ onto $\cH_K$. Using $\fh$,  the excess risk of $\wh f\circ \wh g$ is decomposed in \eqref{bd_irre_error} of \cref{sec_theory_KRR_decomp} as 
\begin{align} \label{eq_intro_1}
     \cE(\wh f\circ\wh g) ~ \lesssim ~ \EE  [  \ell_{\wh f\circ \wh g}(Y,X)  ] + \EE \Delta_{\wh g} + \|\fh-f^*\|_\rho^2
\end{align}
where $\|\fh-f^*\|_\rho^2$ represents the approximation error due to model misspecification,  
$
    \Delta_{\wh g}: = \| K_Z - K_{\wh g(X)}\|_K^2
$
reflects the error of predicting $Z$ by $\wh g(X)$ through the kernel function $K$, and 
$
      \ell_{\wh f\circ \wh g}(y,x)  := [y - (\wh f\circ \wh g)(x) ]^2 - [y - (\fh\circ \wh g)(x) ]^2
$
denotes the squared loss of $\wh f\circ \wh g$ relative to $\fh\circ \wh g$. Analyzing the first term $\EE  [  \ell_{\wh f\circ \wh g}(Y,X)  ]$ in \eqref{eq_intro_1} turns out to be the main technical challenge, and its bound also depends on both $\EE \Delta_{\wh g}$ and $\|\fh -f^*\|_\rho^2$.
In \cref{thm_risk} of \cref{sec_theory_KRR_risk}, we state our main result:
    for any $\eta \in (0,1)$ and a suitable choice of $\lambda$ in \eqref{def_KRR}, with probability $1-\eta$, one has
    \begin{align}\label{rate_intro}
        \cE(\wh f\circ \wh g) 
          ~ \lesssim ~  \delta_n\log(1/\eta)  +\EE \Delta_{\wh g}  +  \|\fh-f^*\|_\rho^2+ \frac{\log(1/\eta)}{n}.  
    \end{align} 
\cref{rate_intro} holds for any $f^*\in \cL^2(\rho)$ and any predictor $\wh g$ that is constructed independently of $\cD$.
In particular, it does not impose any restrictions on $\EE \Delta_{\wh g}$, the prediction error for $Z$, nor on the approximation error $\|\fh - f^*\|_\rho^2$. Regardless of their magnitudes, both terms appear {\em additively} in the risk bound, a particular surprising result for $\EE \Delta_{\wh g}$. This is a new result to the best of our knowledge. Even in the classical setting, the rate in \eqref{rate_intro} with $\EE\Delta_{\wh g} = 0$ generalizes the existing result in \eqref{rate_classic} by allowing arbitrary model misspecification, as well as unboundedness of both the response and the RKHS. Moreover, even when the model is correctly specified, our analysis achieves the same optimal rate without imposing any regularity or decaying conditions on $\mu_j$'s, such as those in \eqref{cond_poly}, in contrast to existing integral operator approaches, for instance, \cite{fischer2020sobolev,steinwart2009optimal,caponnetto2007optimal,rudi2017generalization}.

In \cref{sec_theory_latent} we derive more transparent expression of $\EE \Delta_{\wh g}$: for any kernel  $K$ that satisfies the Lipschitz property in \cref{ass_Lip_K},
$\EE \Delta_{\wh g}$ gets simplified to the $\cL^2$ prediction error of $\wh g(X)$, given by $\EE\|\wh g(X)- Z\|^2_2$. The term $\delta_n$ in \eqref{rate_intro} is shown in \cref{thm_lb_excess_risk} of \cref{sec_theory_lb} to be indispensable in the minimax sense. To discuss its order, we derive in \cref{sec_theory_delta} its slow rate in \cref{cor_slow_rates} and its fast rates under three common classes of kernel functions: linear kernels in \cref{cor_K_linear}, kernels with polynomially decaying eigenvalues in \cref{cor_K_poly}, and kernels with exponentially decaying eigenvalues in \cref{cor_K_exp}. Our new risk bounds generalize existing results in all cases by characterizing the dependence on both the approximation error and the error in predicting the feature inputs.

In \cref{sec_general_loss}, we further extend our analysis to other convex loss functions in \eqref{def_KRR} that satisfy some classical smoothness and strong convexity conditions.

\paragraph*{A new framework for analyzing KRR with predicted feature inputs.} Along with our new risk bounds comes a novel approach to analyzing the KRR with predicted inputs. The framework we develop offers two key contributions: (1) it improves existing analyses in the classical setting by accommodating arbitrary model misspecification and allowing both the response variable and the RKHS to be unbounded; (2) it is essential for handling predicted feature inputs, enabling us to derive risk bounds {\em without} imposing any assumptions on the proximity of $\wh g(X)$ to $Z$.

To accomplish the first point, we develop a new approach that combines the integral operator method with the local Rademacher complexity based empirical process theory from \cite{bartlett2005local}, along with a careful reduction argument. This hybrid analysis is necessary because the integral operator approach cannot easily handle model misspecification, while the empirical process theory requires the boundedness of both the response variable and the RKHS. Specifically, as seen in the decomposition \eqref{eq_intro_1}, the key is to bound from above $\EE[\ell_{\wh f\circ \wh g}(Y,X)]$. Since  $\sup_{f\in \cH_K} \|f\|_K$ could be unbounded, we first apply a reduction argument, showing that bounding $\EE[\ell_{\wh f\circ \wh g}(Y,X)]$ can be reduced to controlling the empirical process
\begin{equation}\label{emp_process}
    \sup_{f\in \uline{\cF_b}}\Bigl\{\EE[\ell_{f\circ \wh g}(Y,X)] - 2\EE_n[\ell_{f\circ \wh g}(Y,X)]\Bigr\}
\end{equation}
where $\uline{\cF_b}$ is a {\em bounded} subset of $\cH_K$, given by \cref{def_local_ball} of \cref{sec_theory_KRR_proof}. We use $\EE_n$ in this paper to denote the expectation with respect to the empirical measure of $n$ i.i.d. samples.
Since $Y$ could also be unbounded, we need to bound the cross-term $\EE_n [\epsilon (f\circ \wh g)(X) - \epsilon (\fh \circ \wh g)(X)]$ uniformly over $f \in \uline{\cF_b}$. This is  accomplished by employing both the empirical process theory based on local Rademacher complexity from \cite{bartlett2005local} and the integral operator approach, as detailed in \cref{sec_cross_term}; see also \cref{lem_epsilon}. Our analysis only requires moment conditions on the regression error $\epsilon$. Finally, we again apply the local Rademacher complexity technique of \cite{bartlett2005local} to control the remaining bounded components of the empirical process \eqref{emp_process}. Notably, our hybrid approach extends to learning algorithms in \eqref{def_KRR} that use convex losses beyond the squared loss, as illustrated in \cref{sec_general_loss}.

Another fundamental difficulty in our analysis arises from handling predicted feature inputs. As noted in the end of \cref{sec_intro_related}, the primary challenge lies in establishing a precise connection between the local Rademacher complexity and the kernel complexity function $R(\delta)$ in the presence of errors in predicting the features. The existing analysis in \cite{mendelson2002geometric} only establishes a connection to $R_x(\delta)$, which depends on the probability measure $\rho_x$ of $\wh g(X)$. And it is challenging to relate $R_x(\delta)$ with $R(\delta)$ without imposing restrictive assumptions on the closeness of  $\rho_x$ and $\rho$. To circumvent this issue, rather than working with the population-level local Rademacher complexity and $R_x(\delta)$—both of which rely on $\rho_x$—we instead focus on their {\em empirical} counterparts. The first step is to relate the population-level local Rademacher complexity to its empirical counterpart where we borrow existing concentration inequalities in \cite{boucheron2000sharp,boucheron2003concentration,bartlett2005local}. By adapting the argument in \citet[Lemma 6.6]{bartlett2005local}, we further link the empirical local Rademacher complexity to $\wh R_x (\delta)$, the empirical counterpart of $R_x(\delta)$. The next step is to connect $\wh R_x (\delta)$ to $\wh  R(\delta)$, with the latter being the empirical counterpart of $R(\delta)$. This is the most challenging step, as $\wh R_x (\delta)$ and $\wh  R(\delta)$  depend, respectively, on the empirical integral operators associated with $\wh g(X_i)$ and $Z_i$ for $i\in [n]$. Our proof reveals that this requires to control the spectral difference between two $n\times n$ kernel matrices with entries equal to $K(\wh g(X_i), \wh g(X_j))$ and $K(Z_i, Z_j)$ for $i,j\in [n]$. As accomplished in \cref{lem_bd_wh_R,lem_bd_bar_Delta} of \cref{sec_bd_wh_R_delta}, this is highly non-trivial as we do not impose any requirements on the proximity of $\wh g(X_i)$ to $Z_i$. Finally, we establish sharp concentration inequalities between $\wh  R(\delta)$ and its population-level counterpart, $R(\delta)$,  to close the loop of relating the local Rademacher complexity to $R(\delta)$.  Completing the loop requires a series of technical lemmas, collected in  \cref{supp_lem_4,lem_bd_em_local_rade_gb,lem_bd_wh_R,lem_bd_bar_Delta,lem_bd_wt_R} of \cref{app_sec_radem_emp,sec_pf_bd_emlocom,sec_bd_wh_R_delta,sec_bd_wt_R}. Developing this proof strategy is part of our contribution, and we provide a more detailed discussion in \cref{sec_main_relate_complexity}. For illustration, Figure \ref{Illustration_proof} below depicts our roadmap of relating the local Rademacher complexity to $R(\delta)$, represented by the solid arrow.

 \begin{figure}[ht]
    \centering
    \centerline{\includegraphics[width=0.85\textwidth]{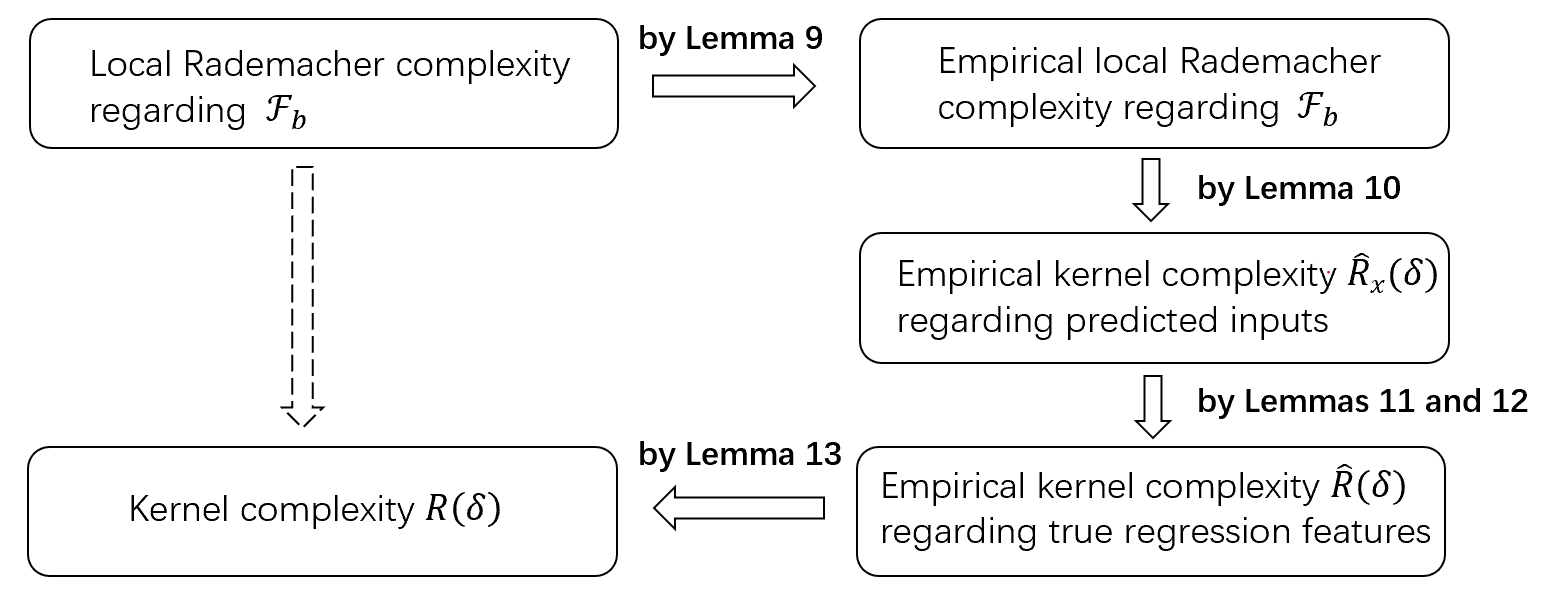}}
    \caption{Proof strategy of relating local Rademacher complexity to kernel complexity $R(\delta)$.}
    \label{Illustration_proof}
\end{figure}

\paragraph*{Application to factor-based non-parametric regression models.} 
Our third contribution is to provide a complete analysis of using KRR with feature inputs predicted by PCA under the factor-based nonparametric regression models \eqref{model_X} and \eqref{model}.  As an application of our general theory, together with the existing results on high-dimensional factor models, we state in \cref{thm_PCA} of \cref{sec_theory_PCA} an explicit upper bound of the excess risk of the KRR with inputs predicted by PCA. 
In \cref{thm_lb_excess_risk} of \cref{sec_theory_PCA}, we further establish a matching minimax lower bound of 
 the excess risk 
 under models \eqref{model_X} and \eqref{model}, thereby demonstrating the minimax optimality of KRR using the leading PCs. 

Prediction consistency of Principal Component Regression (PCR) has been established in \cite{stock2002forecasting} under the factor-based linear regression models \eqref{model_Y_linear} and \eqref{model_X}.  Explicit rates of the excess risk are later given in \cite{bing2021prediction}, where a general linear predictor of $X_i$, including PCA as a particular instance, is used to predict $Z_i$. In a more recent work \citep{fan2024factor}, the authors investigate the predictive performance of fitting neural networks between the response and the retained PCs, proving minimax optimal risk bounds under the factor model \eqref{model_X} when the regression function belongs to a hierarchical composition of functions in the  H\"{o}lder class.
Under model \eqref{model_X} and \eqref{model} with $f^*$ belonging to some RKHS, our results in \cref{sec_PCA} provide the first minimax optimal risk bound for using KRR with the leading PCs. Moreover, our theory extends beyond the factor model \eqref{model_X}, as it applies to any predictor $\wh g$ of $Z$ and accommodates arbitrary dependence between $X$ and $Z$.\\

This paper is organized as follows.  The algorithm of KRR with predicted inputs is stated in \cref{sec_method_krr}. Theoretical results of analyzing KRR with predicted inputs are stated in \cref{sec_theory_KRR}. The risk decomposition is given in \cref{sec_theory_KRR_decomp} while the risk bounds along with the main assumptions are stated in \cref{sec_theory_KRR_risk}. Risk bounds for specific kernels are given in \cref{sec_theory_cor} while in \cref{sec_theory_KRR_proof} we explain the proof sketch and highlight the main technical difficulties. In \cref{sec_PCA} we apply our general result to factor-based nonparametric regression models. In \cref{sec_general_loss} we extend our analysis and derive risk bounds for general convex loss functions.  Simulation studies and real data analysis that corroborate our theoretical findings are stated in \cref{app_sec_sim,sec_real_data}, respectively. All proofs are deferred to the Appendix. \\

 \noindent {\bf Notation.} 
For any $a,b\in \RR$, we write $a \wedge b= \min\{a,b\}$ and $a \vee b= \max\{a,b\}$. 
For any integer $d$, we let $[d] = \{1,\ldots,d\}$.  
For any symmetric, semi-positive definite matrix $Q \in  \RR^{d\times d}$,
 we use $\lambda_1(Q)\ge \lambda_1(Q)\ge\ldots \ge  \lambda_d (Q)$ to denote its eigenvalues. For any matrix $A$, we use $\|A\|_\op$ and $\|A\|_{\F}$ to denote its operator norm and Frobenius norm, respectively. 
We use $\bI_d$  to denote the $d\times d$ identity matrix. For $d_1\ge d_2$, we write $\OO_{d_1\times d_2}$ for  matrices with orthonormal columns. 
The inner product and the endowed norm in the Euclidean space are denoted as  $\langle\cdot,\cdot\rangle_2$ and $\|\cdot\|_2$, respectively. For any function $f: \cZ \to \RR$, $\|f\|_\infty =\sup_{z \in \cZ}|f(z)|$.  For any two sequences $a_n$ and $b_n$, 
 we write $a_n \lesssim b_n$ if there exists some constant $C$ such that $a_n\le C b_n$ for all $n$. We write $a_n\asymp b_n$ if $a_n\lesssim b_n$ and $b_n\lesssim a_n$.  In this paper, we use $c,c', C$ and $C'$ to denote constants whose values may vary line by line unless otherwise stated.

\section{Kernel ridge regression with predicted feature inputs}\label{sec_method_krr}

In this section, we state the algorithm for Kernel Ridge Regression (KRR) with predicted feature inputs. Recall that $\cD = \{(Y_i,X_i)\}_{i=1}^n \in (\RR, \cX)^n$ contains the training data.  
Let $\wh g:\cX \to \cZ$ be any measurable function used to predict the features $Z_i\in \cZ$. For instance, $\wh g$ could be obtained from one of the feature extraction or representation learning approaches mentioned in the Introduction. The predicted feature inputs are $\wh g(X_1),\ldots, \wh g(X_n)$.  

Let  $K:{\cZ \times \cZ}\to \RR$ be some pre-specified kernel function and denote by $\cH_K$ the Reproducing Kernel Hilbert Space (RKHS) induced by $K$. Write the equipped inner product in $\cH_K$ as $\langle\cdot,\cdot\rangle_K$ and the endowed norm as $\|\cdot\|_K$. We propose to regress the response $\bY = (Y_1, \ldots, Y_n)^\T$ onto  $\wh g(X_1),\ldots, \wh g(X_n)$ by solving the following optimization problem:
\begin{align}\label{def_f_hat}
\wh f =\argmin_{f\in \cH_K}  \left\{ \frac{1}{n} \sum_{i=1}^n\bigl(Y_i- (f\circ \wh g)(X_i)\bigl)^2
+\lambda\|f\|_{K}^2 \right\}.
\end{align}
By letting {$K_{v}:=K(v, \cdot)\in \cH_K$} for any $v \in \cZ$, the reproducing property asserts that
$    
    \langle f, K_{v}\rangle_K=f(v)
$
for all $f \in \cH_K$. As a result, the representer theorem \citep{kimeldorf1971some} states that  the  minimization in \eqref{def_f_hat} admits a closed-form solution: 
\begin{align}\label{def_f_hat_closedform}
   \wh f =\frac{1}{\sqrt{n}}\sum_{i=1}^n \wh \alpha_i K_{\wh g(X_i)},
\end{align}
where the representer coefficients  $\wh \alpha=(\wh \alpha_1,\ldots, \wh \alpha_n)^\T$ are obtained by
\begin{align}\label{coeff}
    \wh \alpha &=\argmin_{\alpha \in \RR^n}  \left \{ 
    {1\over n}\bigl\|\bY - \sqrt{n} ~ \bK_x \alpha \bigr\|_2^2
    +\lambda~ \alpha^\T \bK_x \alpha \right \} = {1\over \sqrt n}\bigl(\bK_x +\lambda \bI_n\bigr)^{-1}\bY.
\end{align}
Here   $\bK_x$ is the $n\times n$ kernel matrix with entries equal to $n^{-1} K  (\wh g(X_i),\wh g(X_j))$ for $i, j\in [n]$. 
Finally, for a new data point $X \in \cX$, we predict its corresponding response $Y$ by 
\begin{align}\label{def_f_hat_pred}
    (\wh f \circ \wh g) (X)=\frac{1}{\sqrt{n}}\sum_{i=1}^n \wh \alpha_i K\left(\wh g(X_i), \wh g(X)\right).
\end{align}


\begin{remark}[Independence between $\wh g$ and $\cD$]\label{rem_g_hat}
    Our theory in \cref{sec_theory_KRR} requires that $\wh g$ be constructed independently of the training data $\cD$. This independence simplifies our analysis, especially given the technical complexity of the current framework. Moreover, such independence can lead to better results in many statistical problems, including smaller prediction or estimation errors and weaker conditions -- a phenomenon has been observed in various problems, such as estimating the optimal instrument in sparse high-dimensional instrumental variable models \citep{belloni2012sparse}, inferring a low-dimensional parameter in the presence of high-dimensional nuisance parameters \citep{DDML}, and performing discriminant analysis using a few retained principal components of high-dimensional features \citep{bing2023optimal}. See, also, our simulation studies in \cref{app_sec_sim}.
    On the other hand, in many applications such as the examples mentioned in the Introduction, auxiliary data without labels or responses are readily available or at least easy to obtain. When this is not the case, one can still use a procedure called $k$-fold cross-fitting \citep{DDML}. The 2-fold version of this method first splits the data into equal parts, computes $\wh g$ on one subset, and then performs the regression step using the other subset, before switching their roles. The final prediction is then made by the average of the two resulting predictors.
\end{remark}

\section{Theory of KRR with predicted feature inputs}\label{sec_theory_KRR}

 In this section we  present our theoretical guarantees of the KRR estimator $\wh f$ in \eqref{def_f_hat} that is based on a generic predictor $\wh g:\cX\to \cZ$,   
constructed independently of $\cD$. 
The quantity of our interest is the excess risk of the prediction function  
$    \wh f \circ \wh g :~  \cX \to \RR $, given by  \eqref{def_excess_risk}. 
Since the excess risk is  random depending on $\cD$ and $\wh g$, our goal is to establish its rate of convergence in probability.   We start by decomposing it into three interpretable terms.

 \subsection{Decomposition of the excess risk}\label{sec_theory_KRR_decomp}

For any deterministic $f\in \cL^2(\rho)$, we prove in \cref{app_sec_proof_risk_decomp} that
\begin{align}\label{eq_risk_decomp}
    \cE \bigl ( \wh f \circ \wh g \bigr )  =  \EE\bigl[ (Y  - (\wh f\circ \wh g)(X) )^2 -  (Y  - ( f\circ \wh g)(X) )^2  \bigr]    
    + \EE \left [ f^*(Z)- (f\circ \wh g)(X)   \right ]^2.
\end{align}
To specify the choice of  $f$, we make the following blanket assumption. Recall that $\|\cdot\|_\rho$ is the induced norm of $\cL^2(\rho)$.
\begin{assumption}\label{ass_f_H}
    There exists some $f_\cH\in \cH_K$ such that 
    $
         \|f_\cH -  f^*\|_\rho^2  =    \inf_{f\in \cH_K } \|f - f^*\|_\rho^2. 
    $
\end{assumption}
\noindent \cref{ass_f_H} is weaker than assuming $f^* \in \cH_K$.   If the latter holds, we immediately have $f_\mathcal{H} = f^*$. Since $\cH_K$ is convex, existence of $\fh$ also ensures its uniqueness \citep{cucker2002mathematical}.   As discussed in \cite{rudi2017generalization}, existence of $\fh$   can be  ensured if $\cH_K$ in \cref{ass_f_H} is replaced by $\{f\in \cH_K : \|f\|_K\le R\}$ for any finite radius $R$. Alternatively, we can choose $f_{\cH,\lambda} = \argmin_{f\in \cH_K}\{\|f-f^*\|_\rho^2 + \lambda\|f\|_K^2\}$ which, for $\lambda>0$, is always unique and exists \citep{cucker2002mathematical}. Notably, our results apply to  any function $f\in \cH_K$, in replace of $\fh$, including these examples.

By choosing $f = \fh$ in \eqref{eq_risk_decomp}, the second term on the right-hand side of \eqref{eq_risk_decomp} represents the irreducible error from two sources.   Indeed, 
we prove in \cref{app_sec_proof_risk_decomp} that for any $\theta \ge 1$, 
\begin{align}\label{bd_irre_error}
    \EE \left [  f^*(Z)- (\fh \circ \wh g)(X)   \right ]^2 
    &~ \le~  (1+\theta) \|\fh \|_K^2 ~ \EE\Delta_{\wh g} +  {1+\theta \over \theta}\|f_\cH - f^*\|_\rho^2.
\end{align} 
The quantity 
\begin{align}\label{def_Delta}
    \EE \Delta_{\wh g}:= \EE \| K_Z - K_{\wh g(X)}\|_K^2
\end{align}
is referred to as the {\em kernel-related latent error}, representing the error in using $\wh g(X)$ to predict the latent factor $Z$, composited with the kernel function $K$. The term $\|f_\cH - f^*\|_\rho^2$ represents the {\em approximation error} due to model misspecification when $f^* \notin \cH_K$. We provide detailed discussion on both terms after stating our main result in \cref{thm_risk} below. Back to the decomposition in  \eqref{eq_risk_decomp}, the first term of the right hand side, with  $ \fh$ in place of $f$, 
represents the prediction risk of $\wh f\circ \wh g$ relative to $\fh\circ \wh g$. Its analysis is the main challenge and is detailed in the next section.

  \subsection{Non-asymptotic upper bounds of the excess risk}\label{sec_theory_KRR_risk}

In this section we establish non-asymptotic upper bounds of the excess risk $\cE(\wh f\circ \wh g)$ in \eqref{def_excess_risk}. An important component of our analysis is the empirical process theory based on local Rademacher complexity developed in \cite{bartlett2005local}, building on a line of previous works such as \cite{koltchinskii2000rademacher,massart2000some,Lugosi}. In our context, local Rademacher complexity gets translated into the complexity of the RKHS. We begin with a brief review of RKHS basics, then discuss its complexity measures and the assumptions underlying our analysis.

We make the following assumption on the kernel function, which is satisfied by many commonly used kernels, including the Gaussian and Laplacian.
\begin{assumption}\label{ass_bd_K}
The kernel function $K$ is continuous,   symmetric and positive semi-definite.\footnote{We say $K$ is positive semi-definite if for all finite sets $\{z_1,\ldots,z_m\}\subset \cZ$ the $m\times m$ matrix  whose $(i,j)$ entry is $K(z_i, z_j)$ is positive semi-definite.} Moreover, there exists some positive constant $\kappa< \i$ such that $\sup_{z,{z}'\in \cZ} K(z,{z}')\le \kappa^2$.
\end{assumption}
\noindent A key element in our analysis is the integral operator  $L_K:  \mathcal{L}^2(\rho) \to \mathcal{L}^2(\rho)$, defined  as
\begin{align}\label{def_L_K}
    L_K f:=\int_{\mathcal{Z}} K_{z} f(z)\d\rho(z). 
\end{align}
We assume that the integral operator $L_K$ can be eigen-decomposed.
\begin{assumption}\label{ass_mercer}
There exist a sequence of eigenvalues $\{\mu_j\}_{j=1}^\infty$ arranged in non-increasing order and corresponding eigenfunctions $\{\phi_j\}_{j=1}^\infty$ that form  an
orthonormal basis of $\cL^2(\rho)$, such that
$
L_{K}   =   \sum_{j=1}^\i \mu_{j}  \langle\phi_{j}, \cdot \rangle_\rho~  \phi_{j}. 
$
\end{assumption}
\noindent  \cref{ass_mercer} is standard in the kernel‐methods literature. It is guaranteed, for example, by Mercer's theorem \citep{mercer1909xvi} when the kernel domain is compact and \cref{ass_bd_K} holds, though it can also hold in more general settings (see, e.g., \citet[Theorem 12.20]{wainwright2019high}). As a result of Assumption \ref{ass_mercer},  
the kernel function $K$ admits
\begin{align}\label{eq_eigen_decomp}
K(z,{z}')=\sum_{j=1}^\infty\mu_j\phi_j(z)\phi_j({z}'), \qquad  \forall ~ z,z'\in\cZ. 
\end{align}
In the case $\mu_j>0$ for all $j\ge 1$,  the RKHS  can be explicitly written as 
\[
\cH_K= \{f=\sum_{j=1}^\infty\gamma_j\phi_j:
\ \sum_{j=1}^\infty {\gamma_j^2 }/{\mu_j}
<\infty \}.
\]
For any $f_1=\sum_{j=1}^\infty\alpha_j\phi_j$
with $\alpha_j=\langle f_1 , \phi_j\rangle_\rho$  and $f_2=\sum_{j=1}^\infty\beta_j\phi_j$ with $\beta_j=\langle f_2 , \phi_j\rangle_\rho$, the equipped  inner product of $\cH_K$ is
\[
\langle f_1, f_2\rangle_{K}=\sum_{j=1}^\infty{\alpha_j\beta_j}/{\mu_j}.
\]
If for some $k\in \NN$, $\mu_j > 0$ for $j\le k$ and $\mu_j=0$ for all $j>k$, the RKHS reduces to a $k$-dimensional function space spanned by  $\{\phi_1,\ldots, \phi_k\}$. In the rest of this paper, we are only interested in non-degenerate kernels, that is,  $\mu_1>0$.

As mentioned in the Introduction,  complexity of the function class  $\cH_K$ is commonly measured by the  kernel complexity function  $R(\delta)$, given by \eqref{kernel_complexity}.
As detailed in \cref{sec_theory_KRR_proof}, $R(\delta)$ is related with the (local) Rademacher complexity \citep{mendelson2002geometric}. Its fixed point, the  positive solution to
$
   R(\delta)= \delta,
$
defines the {\em critical radius} $\delta_n$, a key  quantity in our excess risk bounds.
Both existence and uniqueness
of  
$\delta_n$  are guaranteed as $R(\delta)$ is a sub-root function of $\delta$  
 \citep[Lemma 3.2]{bartlett2005local}. Here we recall that
\begin{definition}
   A function $\psi: [0,\infty)\to  [0,\infty)$ is called sub-root if it is  non-decreasing, and if $\delta\mapsto \psi(\delta)/\sqrt{\delta}$ is non-increasing for $\delta >0$. 
\end{definition}
\noindent We review other properties of sub-root functions in \cref{app_sec_subroot}. 

Finally, we assume the regression error in model \eqref{model}  to be sub-Gaussian. This simplifies our analysis and can be relaxed to sub-exponential tails or bounded finite moments. 
\begin{assumption}\label{ass_reg_error}
There exists some constant $\gamma_\epsilon<\infty$  such that
{$\EE[\exp(t\epsilon)]\le \exp(t^2\gamma_\epsilon^2)$} for all $t\in \RR$.
\end{assumption} 

The following theorem is our main result which provides non-asymptotic upper bounds of the excess risk of $\wh f\circ \wh g$, given in \eqref{def_excess_risk}.

\begin{theorem}\label{thm_risk}
Grant model (\ref{model}) and
Assumptions  \ref{ass_f_H}--\ref{ass_reg_error}. For  any  $\eta\in (0,1)$, by choosing 
\begin{equation}\label{lb_lambda} 
\lambda  =  C\l(      \delta_n\log(1/\eta)  +\EE \Delta_{\wh g}  +  \|\fh-f^*\|_\rho^2 + \frac{\log(1/\eta)}{n} \r)
\end{equation}
in \eqref{def_f_hat},
with probability at least $1-\eta$, one has
\begin{align}\label{bd_pred_rate}
    \cE(\wh f\circ \wh g)  
      ~ \le ~  C'\l(  \delta_n\log(1/\eta)  +\EE \Delta_{\wh g}   +  \|\fh-f^*\|_\rho^2+ \frac{\log(1/\eta)}{n} \r) .
\end{align}  
Here both positive constants $C$ and $C'$ depend only on $\kappa$, $\|f^*\|_\i$, $\|\fh\|_K$ and $\gamma_\epsilon^2$.
\end{theorem}
    The explicit dependence of $C$ and $C'$ on $\kappa$, $\|f^*\|_\i$, $\|\fh\|_K$  and $\gamma_\epsilon^2$ can be found in the proof.
    Since the proof of \cref{thm_risk} is rather long, we offer its sketch and highlight the main technical difficulties in  \cref{sec_theory_KRR_proof}  while defer its full statement to \cref{app_sec_estimation_error_bound}.

From \cref{thm_risk}, the excess risk is determined by the critical radius $\delta_n$, the kernel-related latent error $\EE \Delta_{\wh g}$ and the approximation error $\|\fh-f^*\|_\rho^2$.  We discuss each term in detail shortly. It is worth emphasizing that \cref{thm_risk} does not impose any assumptions on the magnitudes of either $\EE \Delta_{\wh g}$ or $\|\fh - f^*\|_\rho^2$. Surprisingly, regardless of the magnitude of $\EE \Delta_{\wh g}$, it appears {\em additively} in the risk bound.  Finally, achieving  
the risk bound in \eqref{bd_pred_rate} requires to choose $\lambda$ as in \eqref{def_f_hat}, which in practice could be selected via cross-validation.

\begin{remark}[Comparison with the existing theory in the classical setting]
Although \cref{thm_risk} is derived for the case where the feature inputs are predicted, it is informative to compare it with the existing results  in the classical setting. When $Z_1,\ldots, Z_n$ are observed and used in \eqref{def_KRR}, setting $\EE\Delta_{\wh g} = 0$ with $\wh g(z) = z$ for $z\in \cZ$ simplifies the risk bound  \eqref{bd_pred_rate} to
\begin{equation}\label{rate_KRR_classical}
 \cE(\wh f\circ \wh g) ~ \lesssim ~       \delta_n\log(1/\eta)  + \frac{\log(1/\eta)}{n}  +  \|\fh-f^*\|_\rho^2
\end{equation}
which, to the best of our knowledge,  is also new in the literature of kernel learning methods. As mentioned in \cref{sec_intro_related}, most existing analyses of KRR cannot handle arbitrary model misspecification without requiring regularity and decay properties of the eigenvalues $\mu_j$'s. When the model is correctly specified, our risk bound in \eqref{rate_KRR_classical} matches that in \eqref{rate_classic}, derived by \cite{bartlett2005local}. However, their analysis requires both the response variable and the RKHS to be bounded, which we do not assume. Compared to the integral operator approach, our analysis  does not impose any restrictive conditions on the eigenvalues $\mu_j$'s, which are often required in the existing literature \citep{caponnetto2007optimal,lin2018distributed,steinwart2009optimal,fischer2020sobolev}.
For instance, \cite{caponnetto2007optimal,rudi2017generalization} require  $\mu_1\ge\lambda$ in their analysis,  a condition that is somewhat counterintuitive.
Such refinement of our result stems from a distinct proof strategy, contributing to a new analytical framework for kernel-based learning. We provide  in \cref{sec_theory_KRR_proof} further discussion on technical details.
\end{remark}


\begin{remark}[Discussion on the approximation error]\label{rem_approx_error}
If we choose a universal kernel, such as the Gaussian or Laplacian, it is known that the induced RKHS
is dense in  the space of bounded continuous functions ${\cal C}(\cal Z)$ under the infinity norm  \citep{micchelli2006universal}.  In this case $\|\fh - f^*\|_\rho = 0$ if $f^* \in {\cal C}(\cZ)$. Indeed,
for any $\vartheta >0$,  there must exist some $f_\vartheta \in {\cal H}_K$ such that $\|f_{\vartheta} - f^*\|_{\infty} \leq \vartheta$, which implies
\[
\|f_{\cH} - f^*\|_\rho \leq     \|f_{\vartheta} - f^*\|_\rho \leq \|f_{\vartheta} - f^*\|_{\infty}\leq \vartheta.
\]
The leftmost inequality uses the portrayal of $\fh$ in \cref{ass_f_H}.
\end{remark}

\subsubsection{Discussion on the kernel-related latent error}\label{sec_theory_latent}

Recall from \eqref{def_Delta} that the term $\EE \Delta_{\wh g} $  reflects the error of $\wh g(X)$ in predicting $Z$ through the kernel function $K$.
To obtain more transparent expression, 
we need the following Lipschitz property of the map $\cZ\to \cH_K$, that is, $z\mapsto K_z$ for all $z\in \cZ$. 
\begin{assumption}\label{ass_Lip_K}
There exists some constant $C_K>0$ such that   
\begin{align*}\bigl   \|K_z-K_{z'} \bigr  \|_K\le C_K \|z-z'\|_2, \qquad \forall \ z,z'\in \cZ. 
\end{align*}
\end{assumption}

\noindent\cref{ass_Lip_K} can be verified for various kernels. One sufficient condition is the following uniform boundedness condition of the mixing partial derivatives of the kernel function \citep{blanchard2011generalizing}: there exists some  constant $C>0$ such that 
\[
\sup_{z,z' \in \cZ} \left\| \frac{\partial^2K(z, z')}{\partial z \partial z'} \right\|_{\op} \le~  C.
\]
 Such a condition holds for several commonly used kernels, such as the linear kernel and the Gaussian kernel \citep{blanchard2019concentration,sun2022nystrom}.
Under \cref{ass_Lip_K}, the kernel-related latent error can be bounded by
\begin{equation}\label{def_eps_g_hat}
   \EE\Delta_{\wh g} ~ \le ~ C_K^2 ~ \EE \|\wh g(X)-Z\|_2^2,
\end{equation}
entailing consistent prediction of $Z$ itself by using $\wh g(X)$.  We remark that this can be relaxed when  $K$ is invariant to orthogonal transformations, that is, $K(z,z') = K(Qz,Qz')$ for any $z,z'\in \cZ$ and any $Q\in \OO_{r\times r}$. When this  is the case, we only need consistent prediction of $Z$ up to an orthogonal transformation, in the precise sense that \eqref{def_eps_g_hat} can be replaced by
\begin{equation}\label{rem_pred_Z}
    \EE\Delta_{\wh g} ~ \le ~ C_K^2  \inf_{Q\in \OO_{r\times r}} \EE \|\wh g(X) - QZ\|_2^2.
\end{equation}
This claim is proved in \cref{app_sec_proof_rem_pred_Z}.
Two common classes of kernels invariant under orthogonal transformations are
\begin{itemize}[itemsep = 0mm]
    \item[(1)] Inner product kernel: $K(z,z')= h(\langle z,z'\rangle_2)$, e.g. linear kernels, polynomial kernels;
    \item[(2)] Radial basis kernel: $K(z,z')= h(\|z-z'\|_2)$, e.g. exponential kernels, Gaussian kernels.
\end{itemize}
Here $h: \RR\to \RR$ is a kernel-specific function. 

\subsubsection{Discussion on the critical radius}\label{sec_theory_delta}
The critical radius $\delta_n$, defined as $R(\delta_n) = \delta_n$, is an important factor in the risk bound \eqref{bd_pred_rate}. Finding its exact expression is generally difficult. To quantify its rate, 
by the sub-root property of $R(\delta)$ and   \cref{ass_bd_K},  
we prove in \cref{slow_rate_delta} the {\em slow rate} 
$
        \delta_n    \le    {\kappa/ \sqrt n}.
$
Together with \cref{thm_risk}, we obtain the following slow rate of the excess risk.
\begin{corollary}[Slow rate of the excess risk]\label{cor_slow_rates}
Grant model \eqref{model} and  Assumptions \ref{ass_f_H}--\ref{ass_Lip_K}. 
For any $\eta\in(0,1)$ with $\log (1/\eta) \le \sqrt n$, 
by choosing $\lambda$ as in \eqref{lb_lambda} with $\delta_n$ replaced by $\kappa/\sqrt{n}$,
\begin{equation}\label{slow_rate_risk}
    \PP\left\{\cE(\wh f\circ \wh g)~ \le ~   C'\l(  \frac{ \log (1/\eta) }{\sqrt{n}}   
    + \EE \|\wh g(X)-Z\|_2^2  +  \|\fh-f^*\|_\rho^2\r)\right\} \ge 1-\eta
\end{equation}
Here the positive constant $C'$ depends  only on $\kappa$, $\|f^*\|_\i$, $\|\fh\|_K$, $\gamma_\epsilon^2$ and $C_K$.
\end{corollary}
The first  $n^{-1/2}$-type rate is well-known in the literature of kernel methods when $Z_1,\ldots, Z_n$ are observed and no assumptions on the smoothness of $f^*\in \cH_K$ or decaying properties of the eigenvalues $\mu_j$'s are made. See, for instance, Corollary 5 in \cite{smale2007learning}. Our new result also accounts for  model misspecification and the error of predicting feature inputs.

In many situations, however, {\em fast rates} of $\delta_n$  are attainable. \cite{yang2017randomized} introduces the {\em statistical dimension}, defined as the largest index $j$ for which the kernel eigenvalues $\mu_j$ exceeds $\delta_n$, that is,  
$
d(\delta_n)= \max\left\{j\ge 0: {\mu}_j\ge \delta_n \right\}
$
with $\mu_0 := \i$. 
By the expression of $R(\delta)$ in \eqref{kernel_complexity}, one has
\[
    \delta_n = R(\delta_n ) = \biggl( {d(\delta_n)\over n} ~  \delta_n   +  {1\over n}\sum_{j > d(\delta_n)}  \mu_j\biggr)^{1/2}.
\]
If $\sum_{j > d(\delta_n)}  \mu_j \lesssim d(\delta_n) \delta_n $ holds,  then
$\delta_n \asymp  {d(\delta_n) / n}$, hence the name statistical dimension for $d(\delta_n)$. Any kernel whose eigenvalues satisfy such condition is said to be {\em regular} \citep{yang2017randomized}. 
More importantly, for any regular kernel, \citet[Theorem 1]{yang2017randomized} shows that, in the fixed design setting ($Z_i$'s are treated as deterministic) with $\fh = f^*$, an empirical version of $\delta_n$ serves as  a fundamental lower bound for the {\em in-sample} prediction risk:
\[
    \inf_{f}\sup_{\|f^*\|_K \le 1} \EE_n [f(Z) - f^*(Z)]^2 ~ \gtrsim ~ \wh \delta_n 
\]
where the infimum is taken over all estimators based on $\{(Y_i, Z_i)\}_{i=1}^n$, and 
$\wh \delta_n$ is the positive solution to  $\wh  R(\delta) =\delta$ with $\wh  R(\delta)$ being the empirical kernel complexity function, given in \eqref{def_em_kercomp_true_input}. 
For random designs, we show in \cref{thm_lb_excess_risk} of \cref{sec_theory_lb} that $\wh \delta_n$ in the lower bound is replaced by $\delta_n$. 
In the next section we derive explicit expressions of $\delta_n$ for some specific kernels, with the resulting results summarized in \cref{table_cor}.
\begin{table}[ht]
    \centering
    \renewcommand{\arraystretch}{0.6}   
    \setlength{\tabcolsep}{12pt}        
    \caption{Expressions of $\delta_n$ for several specific kernels}
    \label{table_cor}
    \begin{tabular}{l|c c c}
    \toprule
     & Linear kernel & Polynomial kernels & Exponential kernels \\
    \midrule
    $\delta_n$ &
    $r / n $ &
    $n^{-\frac{2\alpha}{2\alpha+1}}$ &
    $\log(n) / n$ \\
    \bottomrule
    \end{tabular}
\end{table}
 
\subsection{Explicit excess risk bounds for specific kernels}\label{sec_theory_cor}

To establish explicit rates of $\delta_n$, we consider three common classes of kernels, including the linear kernels, kernels with polynomially decaying eigenvalues and kernels with exponentially decaying eigenvalues. We also present simplified excess risk bounds of Theorem \ref{thm_risk}  for each class. For ease of presentation, we assume $\fh = f^*$ and 
$\kappa = 1$ 
in the rest of this section.

\subsubsection{Linear kernels}\label{sec_linear_kernel}
Consider $\cZ \subseteq \RR^r$ for some $r\in \NN$. We start with the linear kernel 
$$
    K(z,{z}')=  z^\T{z}',\qquad \forall ~ z,z'\in \cZ.
$$
Denote by $\sz$ the second moment matrix of $Z$ and write its eigen-decomposition as $\sz= \sum_{j=1}^r \sigma_j v_j v_j^\T$ where the eigenvalues are arranged in non-increasing order.
For simplicity, we assume $\rank(\sz) = r$ whence the induced $\cH_K$ also has rank $r$. For each $j\in [r]$, define $\phi_j:\cZ \to \RR$ as 
$$
    \phi_j(z):=  \sigma_j^{-1/2}  v_j^\T z,\qquad \forall ~  z\in \cZ.
$$
It is easy to verify that 
\begin{align}\label{K_decomp_linear}
 K(z,z')=   z^\top z' =   z^\T \sum_{j=1}^r  v_j v_j^\T   z' = \sum_{j=1}^r  \sigma_j \phi_j(z)\phi_j(z'),\qquad \forall\ z,z'\in\cZ.
\end{align}  
The following corollary states the simplified excess risk bound of \cref{thm_risk} for the linear kernel. Its proof is deferred to \cref{app_sec_cor_K_linear}.  
The notation $\lesssim_{\eta}$ denotes the inequality holds up to $\log(1/\eta)$.
 
\begin{corollary}\label{cor_K_linear}
Consider the linear kernel. Grant model \eqref{model} with $f^*\in \cH_K$ and Assumptions \ref{ass_f_H}--\ref{ass_Lip_K}. For any $\eta\in(0,1)$, with probability at least $1-\eta$, one has
\begin{equation}\label{bd_risk_linear}
    \cE(\wh f\circ \wh g)~ \lesssim_{\eta}  ~  \frac{r}{n} + \EE \|\wh g(X)-Z\|_2^2.
\end{equation}
\end{corollary}
    The term $r/n$ is as expected for the prediction risk under linear regression models using $\{(Y_i,Z_i)\}_{i=1}^n$. Note that it only depends on the dimension of the latent factor $Z$ instead of the ambient dimension of $X$. 
     The second term, $\EE \|\wh g(X)-Z\|_2^2$, represents the error in  using $\wh g(X)$ to predict  $Z$.
    When $\wh g$ is chosen as PCA (see \cref{sec_theory_PCA} for details), the predictor $\wh f\circ \wh g$ coincides with the Principal Component Regression (PCR). Our risk bound aligns with existing works of  PCR \citep{stock2002forecasting,bing2021prediction}, although it holds for more general predictor $\wh g$.

\subsubsection{Kernels with polynomially decaying eigenvalues}\label{sec_cor_K_poly}

We consider the kernel class satisfying the  $\alpha$-polynomial decay condition in their eigenvalues, that is, there exists some $\alpha > 1/2$ and some absolute constant $C>0$ such that ${\mu}_j\le C j^{-2\alpha}$ for all $j\ge 1$. Several widely used kernels, including the Sobolev kernel and the Laplacian kernel, belong to this class. See, Chapters 12 and 13 in \cite{wainwright2019high}, for more examples. The quantity $\alpha$ controls the complexity of $\cH_K$, with a smaller $\alpha$ indicating a more complex $\cH_K$.

The following corollary states simplified risk bounds of \cref{thm_risk} for such kernels with polynomially decaying eigenvalues. Its proof can be found in \cref{app_sec_cor_K_poly}. 
\begin{corollary}\label{cor_K_poly}
     Consider the polynomial decay kernels. Grant model \eqref{model} with $f^*\in \cH_K$ and Assumptions \ref{ass_f_H}--\ref{ass_Lip_K}.
     For any $\eta\in(0,1)$, the following holds with probability at least $1-\eta$,
     \begin{equation}\label{bd_excess_risk_poly}
                \cE(\wh f\circ \wh g)~ \lesssim_{\eta} ~ n^{-\frac{2\alpha}{2\alpha+1}}   +  \EE \|\wh g(X)-Z\|_2^2.
        \end{equation} 
\end{corollary}
     The first term in the risk bound \eqref{bd_excess_risk_poly} corresponds to  the minimax optimal prediction risk established in \cite{caponnetto2007optimal}  
    when $n$ i.i.d. pairs of $(Y_i,Z_i)$ are observed. As $\alpha$ decreases and the corresponding $\cH_K$ becomes more complex, the first term decays more slowly, while the term $\EE \|\wh g(X)-Z\|_2^2$ becomes easier to absorb into the first term.

An important example of an RKHS induced by polynomial decay kernels is the well-known Sobolev RKHS. Specifically, consider the case where $\cZ\subset \RR^r$ is open and bounded,  and $\rho$ has a density with respect to the uniform distribution on $\cZ$ \citep{steinwart2009optimal}.
 For any integer $s >  r/2$, let $W_{s}(\cZ)$ denote the Sobolev space of smoothness $s$.  When $s= r =1$, the Sobolev space $W_{1}(\cZ)$ is an RKHS induced by $K(z,z')= \min\{z,z'\}$. For more general settings, we refer the reader to \cite{novak2018reproducing} for the explicit form of the kernel function that induces $W_{s}(\cZ)$.
As described by \citet[Equation 4 on page 119]{edmunds1996function}, the eigenvalues $\{\mu_j\}_{j\ge 1}$ corresponding to $W_s(\cZ)$ satisfy the polynomial decay condition with $\alpha= s/r>1/2$. 
As a result,  the excess risk bound in \cref{cor_K_poly} becomes
\begin{align*}
    \cE(\wh f\circ \wh g)~ \lesssim_{\eta}  ~ n^{-\frac{2s}{2s+r}}   +  \EE \|\wh g(X)-Z\|_2^2.
\end{align*}
The first term follows the classical nonparametric rate but depends only on $r$, the dimension of the latent factor $Z$, rather than the ambient dimension of $\cX$. It tends to zero as long as $r = o(\log n)$ for any fixed smoothness $s$. By contrast, when applying KRR directly to predict the response from $X_1, \ldots, X_n \in \cX\subseteq \RR^p$, the excess risk bound would be $n^{-\frac{2s}{2s+p}}$, which does not vanish once $\log n = \cO(p)$. When $r \ll p$, reducing to a lower dimension before applying KRR could lead to a significant advantage for prediction. See a concrete example in \cref{sec_PCA}.

\subsubsection{Kernels with exponentially  decaying eigenvalues}

We consider in this section RKHSs induced by kernels whose eigenvalues exhibit exponential decay, that is,  there exists some $\gamma > 0$  such that $\mu_j\le \exp(-\gamma j)$ for all $j\ge 1$. It is well known that the eigenvalues of the Gaussian kernel satisfy this exponential decay condition.  

The following corollary states the simplified risk bounds of \cref{thm_risk} for such kernels with exponentially decaying eigenvalues. Its proof can be found in \cref{app_sec_cor_K_exp}.

\begin{corollary}\label{cor_K_exp}
Consider the exponential decay kernels. Grant model \eqref{model} with $f^*\in \cH_K$ and Assumptions  \ref{ass_f_H}--\ref{ass_Lip_K}.
For any $\eta\in(0,1)$, the following holds with probability at least $1-\eta$,
\begin{equation}\label{bd_excess_risk_exp}
\cE(\wh f\circ \wh g)~ \lesssim_{\eta}  ~  \frac{\log n}{n}  +  \EE \|\wh g(X)-Z\|_2^2.
\end{equation}
\end{corollary}

The first term in the bound of \eqref{bd_excess_risk_exp} coincides with the minimax optimal rate in the classical setting for an exponentially decaying kernel (see, for instance, Example 2 in \cite{yang2017randomized} under the fixed design setting). It is interesting to note that absorbing the error from predicting feature inputs into the first term requires a higher accuracy of $\wh g(X)$ in predicting $Z$ for exponentially decaying kernels compared to polynomially decaying kernels.

\subsection{Proof techniques of \texorpdfstring{\cref{thm_risk}}{\texttwoinferior}}\label{sec_theory_KRR_proof}
In this section, we provide a proof sketch of \cref{thm_risk} and highlight its main challenges.  The proof consists of three main components, which are discussed separately in the following subsections.

\subsubsection{Bounding  the empirical process with a reduction argument} 

Recall that the {\em excess} loss function of $f\circ \wh g$ relative to $\fh\circ \wh g$ is: 
$$
     \ell_{f\circ \wh g}(y,x) := \bigl (y - ( f\circ \wh g)(x) \bigr )^2 - \bigl (y - (\fh\circ \wh g)(x) \bigr )^2,\qquad \forall~ x\in \cX,y\in \RR.
$$
The first component of our proof employs a delicate reduction argument, showing that proving \cref{thm_risk} can be reduced to bounding from above the empirical process 
    \begin{equation}\label{target_proof}
        \sup_{f\in \uline{\cF_b} } \Bigl\{\EE[\ell_{f\circ \wh g}(Y,X)] - 2\EE_n[\ell_{f\circ \wh g}(Y,X)]\Bigr\}
    \end{equation}
  where   $\EE_n$ denotes the expectation with respect to the empirical measure, and 
  $\uline{\cF_b}$ is the local RKHS-ball around $\fh$, given by
    \begin{equation}\label{def_local_ball}
        \uline{\cF_b} :=\left\{f\in\cF_b: 
      \EE[\ell_{f\circ \wh g}(Y,X)]\le  \lambda  \right\},\qquad  \cF_b :=\left\{f\in\cH_K: 
     \|f-f_{\cH}\|_K\le 3\|f_{\cH}\|_K\right\}
    \end{equation} 
    with $\lambda$ given in \eqref{lb_lambda}. 
    Since $\uline{\cF_b}\subseteq \cF_b$ and $\cF_b$
    is a bounded function class, which can be seen from 
    \begin{equation*}
         \|f\|_\i \le  4 \kappa  \|\fh \|_K,\qquad \forall~  f\in  \cF_b,
    \end{equation*}
    by the reproducing property, 
    this reduction argument allows us later to apply existing empirical process theory which is only applicable for bounded function classes. Moreover, the reduction argument ensures that our bound in \cref{thm_risk}  depends only on $\|\fh\|_K^2$, a finite quantity that can be much smaller than $\sup_{f\in \cH_K}\|f\|_K^2$, which could be infinite.   
            
     The second part of this component is to bound from above the empirical process \eqref{target_proof}.  Since the response variable $Y$ is not assumed to be bounded,  existing empirical process theory for bounded functions cannot be directly applied. Instead we rewrite the excess loss function as 
    \[
       \EE[\ell_{f\circ \wh g}(Y,X)]  = \EE\left[\bigl(f^*(Z)- (f\circ \wh g)(X)\bigr)^2-\bigl(f^*(Z)-  (f_\cH\circ \wh g)(X)\bigr)^2\right] =:\EE\left[h_{f\circ \wh g}(Z, X)\right]
    \]
    for the function $h_{f\circ \wh g}: \cZ \times \cX \to \RR$. Also note that its empirical version satisfies
    \begin{align*}
         \EE_n[\ell_{f\circ \wh g}(Y,X)] =  \EE_n\left[h_{f\circ \wh g}(Z, X)\right] +  2\EE_n\bigl[\epsilon(f_{\cH}\circ \wh g)(X)- \epsilon(f\circ \wh g)(X)\bigr].
    \end{align*}    
Our proof in this step 
relies on the local Rademacher complexity,  defined as:   for any $\delta\ge 0$, 
\begin{align}\label{def_wt_psi}
     \psi_x (\delta) :=  \localcomp  \l\{  f \circ \wh g -\fh \circ \wh g:  f\in \cF_b, ~      \EE\l[   (f \circ \wh g)(X)-(\fh \circ \wh g)(X ) \r]^2 \le \delta \r\} . 
\end{align}
The detailed definition of $\localcomp$ can be found in \cref{app_sec_Rade}. 
Let $\delta_x$ be the fixed point  of
$ \psi_x$ such that $\psi_x(\delta_x) = \delta_x$.   A key step is to establish the following result that holds uniformly over $f\in \cF_b$,
    \begin{align}\label{bd_h_proof} 
        \EE\left[h_{f\circ \wh g}(Z,X)\right] & ~ \lesssim ~ \EE_n\left[h_{f\circ \wh g}(Z,X)\right] +\delta_x +  \EE \Delta_{\wh g} +  \|f^*-\fh\|_\rho^2.
    \end{align}
    This is proved in \cref{lem_bd_h_f} where we apply the empirical process result based on the local Rademacher complexity in \cite{bartlett2005local}.  A notable difficulty in our case is to take into account both the approximation error and the error of predicting feature inputs. 
    The next step towards bounding the empirical process \eqref{target_proof} is to bound from above the following cross-term uniformly over $f\in \uline{\cF_b}$, 
    \[
         \EE_n\bigl[\epsilon(f_{\cH}\circ \wh g)(X)- \epsilon(f\circ \wh g)(X)\bigr].
    \]  
     This  also turns out to be quite challenging due to the presence of $\wh g(X_1),\ldots, \wh g(X_n)$, and we sketch its proof in \cref{sketch_cross_term} separately. Note that this difficulty could have been avoided if the response  $Y$ were assumed to be bounded.
    Combining \eqref{bd_h_proof} with the bound of the cross-term stated below yields upper bounds for the empirical process \eqref{target_proof}.

\subsubsection{Uniformly bounding the cross-term} \label{sketch_cross_term}
 
The proof of bounding the cross-term over the function space $\uline{\cF_b}$ consists of three steps. The first step reduces bounding the cross-term to bounding $\EE_n [ (f\circ \wh g) (X) - (\fh \circ \wh g) (X)  ]^2 $ through the {\em empirical} kernel complexity function $\wh R_x (\delta)$, defined as  
\begin{align}\label{def_em_kercomp}
    \wh R_x (\delta) =  \biggl(\frac{1}{n}\sum_{j=1}^n 
\min \l \{ \delta, \wh {\mu}_{x,j} \r\}\biggr)^{1/2}, \qquad \forall  \ \delta\ge 0. 
\end{align}
Here we use $\wh \mu_{x,1}\ge \cdots\ge\wh \mu_{x,n}$ to denote the eigenvalues of the kernel matrix $\bK_x$ with entries $n^{-1}K(\wh g(X_i),\wh g(X_j))$ for $i,j\in [n]$. Since $\wh R_x (\delta)$  depends only on the predicted inputs, its introduction is pivotal not only for 
uniformly bounding the cross-term but also, as detailed in the next section, for determining the rate of $\delta_x$, the
fixed point to the local Rademacher complexity in \eqref{def_wt_psi}. 
We show  in  \cref{lem_cross_term_1} of \cref{sec_cross_term} that for any $q>  0$, 
with probability $1-\eta$, the following holds uniformly over the function class $\{f\in \cF_b: \EE_n[(f\circ\wh g)(X)- (\fh\circ\wh g)(X)]^2\le q \}$,
\begin{align*}
    \EE_n [\epsilon(f\circ \wh g)(X) - \epsilon(\fh\circ \wh g)(X)] ~ \lesssim ~   \sqrt{\log(1/\eta)}  ~ 
    \wh R_x(q).
\end{align*}  
The next step connects $\EE_n [ (f\circ \wh g) (X) - (\fh \circ \wh g) (X)  ]^2$ 
to its population-level counterpart. By leveraging the empirical process result in \cite{bartlett2005local}, we show in  \cref{lem_cross_term_2}  that with probability $1-\eta$ and uniformly over $f\in \cF_b$,
\begin{align*}
    \EE_n [ (f\circ \wh g) (X) - (\fh \circ \wh g) (X)  ]^2 ~ \lesssim ~  \EE [ (f\circ \wh g) (X) - (\fh \circ \wh g) (X)  ]^2 
    +  \delta_x  { + \frac{\log{(1/\eta)}}{n}}.
\end{align*}
Finally, as $\uline{\cF_b}$ is defined via the excess loss function $\EE[\ell_{f\circ \wh g}(Y,X)]$, we further prove that  
\[
    \EE \l  [  (f\circ \wh g) (X) - (\fh \circ \wh g) (X)  \r]^2 ~  \lesssim ~ \EE[\ell_{f\circ \wh g}(Y,X)] +\EE\Delta _{\wh g}  +\|\fh -f^*\|_\rho^2
\]
so that combining the three displays above yields the final uniform bound for the cross-term, stated in the following lemma, with the detailed proof provided in  \cref{sec_cross_term}.  
\begin{lemma}
    Grant model \eqref{model} with Assumptions \ref{ass_f_H}--\ref{ass_reg_error}. Fix any $\eta\in (0,1)$. With probability at least $1-\eta$,  the following holds uniformly over $f\in \uline{\cF_b}$,
     \[
         \EE_n\bigl[\epsilon(f_{\cH}\circ \wh g)(X)- \epsilon(f\circ \wh g)(X)\bigr]~ \lesssim ~   \sqrt{\log(1/\eta)}  ~ \wh R_x\l(\lambda  + \delta_x   \r).
     \]
\end{lemma}
\noindent It thus remains to derive upper bounds for 
$\delta_x$ and $\wh R_x (\cdot)$, which are outlined  in the next section. 

\subsubsection{Relating the local Rademacher complexity to the kernel complexity}\label{sec_main_relate_complexity}

The two previous components necessitate studying the fixed point $\delta_x$ in \eqref{bd_h_proof}, which in turn requires relating the local Rademacher complexity $\psi_x (\delta)$ to the kernel complexity function $R(\delta)$ in \eqref{kernel_complexity}. This arises exclusively from using predicted feature inputs and is highly non-trivial without imposing any assumptions on the error of predicting features.

To appreciate the difficulty, we start by revisiting existing analysis in the classical setting. When  $\{Z_i\}_{i=1}^n$ are observed and used as the input vectors in the empirical risk minimization \eqref{def_f_hat},  the local Rademacher complexity becomes 
\[
     \psi_z (\delta) :=  \localcomp  \l\{  f   -\fh :  f\in \cF_b, ~      \EE\l[   f (Z)-\fh(Z) \r]^2 \le \delta \r\},\qquad \forall ~ \delta \ge 0.
\]
\citet[Lemma 42]{mendelson2002geometric} shows that 
$\psi_z (\delta)$  is sandwiched by the kernel complexity 
$R(\delta)$ in \eqref{bd_h_proof},
up to some constants that only depend on the radius of $\cF_b$. Armed with this result, it is easy to prove that the fixed point $\delta_z$ of $\psi_z(\delta)$ satisfies $\delta_z \lesssim \delta_n$. 
One key argument in \cite{mendelson2002geometric} is the fact that  $\EE[\phi^2_j(Z_i)]=1$ for all $i\in [n]$ and $j\ge 1$, as $\{\phi_j\}_{j=1}^\i$, defined in \cref{ass_mercer}, form an orthonormal basis of the $\cL^2$ space induced by the probability measure $\rho$ of $Z$.

Returning to our case, the intrinsic difficulty arises from the mismatch between the predicted features $\wh g(X)$, which are used in the regression, and the latent factor $Z$ whose probability measure $\rho$ induces the integral operator $L_K$ hence the kernel complexity function 
 $R(\delta)$.
Indeed, if we directly characterize the complexity of $\cH_K$ via the integral operator $L_{x,K}$, induced by the probability measure $\rho_x$ of  $\wh g (X)$, we need to assume $L_{x,K}$ admits the Mercer decomposition as in \cref{ass_mercer} with eigenvalues $\{\mu_{x,j}\}_{j\ge 1}$ (note that this is already difficult to justify due to its dependence on  $\rho_x$ and $\wh g$). Then, by repeating the argument in \cite{mendelson2002geometric}, the local Rademacher complexity $\psi_x (\delta)$ in \eqref{def_wt_psi}
    is bounded by (in order)  
    \[
    R_x(\delta)= \biggl(\frac{1}{n}\sum_{j=1}^\i \min\{\delta,\mu_{x,j}  \}\biggr)^{1/2}.
    \]
    Establishing a direct relationship between $R_x(\delta)$ and $R(\delta)$, however, is generally intractable without imposing strong assumptions on the relationship between $\rho_x$ and $\rho$. 

To deal with this mismatch issue, we have to work on the {\em empirical} counterparts of the aforementioned quantities. Start with the  empirical counterpart  of $\psi_x (\delta)$, defined as:  for any $\delta\ge 0$, 
\begin{align}\label{target_em_locom_sec}
    \wh \psi _x (\delta) :=  \emlocalcomp  \l\{ f \circ \wh g-\fh \circ \wh g: ~ f\in \cF_b,~    \EE_n \l[  (f \circ \wh g)(X)-(\fh \circ \wh g)(X)\r]^2 \le \delta \r\}.
\end{align}
By borrowing the results in \cite{boucheron2000sharp,boucheron2003concentration,bartlett2005local},
we first establish the connection between the population and empirical local  Rademacher complexities in \cref{supp_lem_4} of \cref{app_sec_radem_emp}: 
\begin{align} \label{first_bound}
\psi _x (\delta) ~  \lesssim ~     \wh \psi _x (\delta)   +  \frac{\log(1/\eta)}{n},
\end{align}
provided that  $\delta$ is not too small. 
Subsequently,   we relate the empirical local Rademacher complexity $\wh \psi _x (\delta)$ to the empirical kernel complexity $\wh R_x (\delta)$ in \eqref{def_em_kercomp}. 
Both quantities depend on the predicted inputs $\wh g(X_1),\ldots, \wh g(X_n)$. By following the argument used in the proof of Lemma 6.6 of \citet{bartlett2005local}, 
we prove in  \cref{lem_bd_em_local_rade_gb} of \cref{sec_pf_bd_emlocom}: 
\begin{align} \label{second_bound} 
  \wh  \psi _x (\delta) ~  \lesssim ~  \wh R_x (\delta),   \qquad \forall~  \delta \ge 0.  
\end{align}

We proceed to derive an upper bound for $\wh R_x (\delta)$, a quantity that also appears in the bound of the cross-term in \cref{sketch_cross_term}. Analyzing $\wh R_x (\delta)$  is crucial for quantifying the increased complexity of $\cH_K$ caused by the predicted feature inputs. Indeed, when $\wh g(X_i)$ predicts $Z_i$ accurately,  the kernel matrix $\bK_x$ should be close the kernel matrix $\bK$, whose entries are $n^{-1} K(Z_i, Z_j)$ for $i,j\in [n]$. As a result,  the empirical kernel complexity
$\wh R_x (\delta)$ approximates 
\begin{align}\label{def_em_kercomp_true_input}
    \wh  R(\delta) =  \biggl(\frac{1}{n}\sum_{j=1}^n 
\min \l \{ \delta, \wh {\mu}_j \r\}\biggr)^{1/2}, \qquad \forall  \ \delta\ge 0, 
\end{align}
with  $\wh \mu_1\ge \cdots \ge \wh \mu_n$ being the eigenvalues of $\bK$.
Note that
$\wh  R(\delta)$  can be regarded as  the empirical counterpart of 
$R(\delta)$, and it is expected due to concentration that $\wh  R(\delta)$  and $R(\delta)$ have the same order. However, when the prediction error of $\wh g(X_i)$ is not negligible, it should manifest in the gap between 
$\wh R_x (\delta )$ and   $\wh  R(\delta)$. 
Our results in  \cref{lem_bd_wh_R,lem_bd_bar_Delta} and \cref{cor_lem7_and_lem8} of \cref{sec_bd_wh_R_delta}   characterize how the error of $\wh g(X_i)$ in predicting $Z_i$ inflates the empirical kernel complexity:  
\begin{align} \label{third_bound}
   \wh R_x (\delta) ~  \lesssim ~ \wh  R(\delta) + \sqrt{\frac{ \EE\Delta_{\wh g}}{n} } + \frac{\log (1/\eta)}{n} . 
\end{align}
The final step is to bound from above  $\wh  R(\delta)$ by its population-level counterpart $R(\delta)$. This is proved in \cref{lem_bd_wt_R} of \cref{sec_bd_wt_R}: 
\begin{align} \label{fourth_bound}
 \wh  R(\delta)  ~  \lesssim ~ R(\delta)  + \frac{\sqrt{\log(1/\eta)}}{n}.  
\end{align}
By collecting the results in \eqref{first_bound}, \eqref{second_bound}, \eqref{third_bound} and \eqref{fourth_bound}, we conclude that
\begin{align}\label{final_bound}
  \psi_x (\delta)  ~  \lesssim ~   R(\delta)+\sqrt{\frac{\EE \Delta_{\wh g} }{n} } + \frac{ \sqrt{\log(1/\eta)}}{n},
\end{align}
from which  we finally derive the order of $\delta_x$ in \cref{lem_bd_td_delta} of \cref{app_sec_lrc_kc} as 
\begin{align*}
    \delta_x ~  \lesssim ~   \delta_n  + \EE\Delta_{\wh g} + \frac{ \sqrt{\log(1/\eta)}}{n}. 
\end{align*}

\section{Application to factor-based regression models}\label{sec_PCA}

To provide a concrete application of our developed theory in Section \ref{sec_theory_KRR}, we consider the factor model  \eqref{model_X} where the observable features $X \in \cX = \RR^p$ are linearly related with the latent factor $Z \in \cZ = \RR^r$ with $r\ll p$. By centering, we can assume $Y$, $X$ and $Z$ in models  \eqref{model_X} and \eqref{model} have zero mean.  
Under such model, it is reasonable to choose $\wh g$ as a linear function to predict $Z$, as detailed in \cref{sec_method_PCA}. In \cref{sec_theory_PCA} we state the excess risk bound of the corresponding predictor while in \cref{sec_theory_lb} we prove a matching minimax lower bound, thereby establishing the minimax optimality of the proposed predictor.

\subsection{Prediction of latent factors using PCA}\label{sec_method_PCA}

We first discuss the choice of $\wh g$ for  predicting the latent factor $Z$. As mentioned in \cref{rem_g_hat}, we construct such predictor from an auxiliary data $\bX'\in \RR^{n'\times p}$ (for simplicity, we assume $n' = n$) that is independent of the training data $\cD$.
Under the factor model \eqref{model_X}, the $n\times p$ matrix $\bX' = (X_1',\ldots, X_n')^\T$ satisfies 
\begin{equation}\label{model_X_mat}
    \bX' = \bZ' A^\T + \bW',
\end{equation}
with $\bZ' = (Z_1',\ldots,Z_n')^\T$ and $\bW' = (W_1',\ldots, W_n')^\T$. Prediction of $\bZ'$ and estimation of $A$ under 
 \eqref{model_X_mat}
 has been extensively studied in the literature of factor models. One classical way  is to solve the following constrained least squares problem:
\begin{equation}\label{crit_Z_hat}
\begin{split}
    &(\wh \bZ', \wh A) ~ = ~ \argmin_{\bZ'\in \RR^{n\times r},\ A\in \RR^{p\times r}}~  {1\over np}\bigl\|
     \bX' - \bZ' A^\T 
    \bigr\|_\F^2,\\
    &\textrm{subject to}\quad  \bZ^{'\T}\bZ'  \textrm{ is diagonal},\quad  A^\T A=p ~ \bI_r.
\end{split}
\end{equation}
The above formulation uses  $r$, the true number of latent factors. In the factor model literature, there exist several methods that are provably consistent for selecting $r$. See, for instance, \cite{Ahn-2013,Bai-Ng-K,bing2020adaptive,bing2021prediction}. 
Write the singular value decomposition (SVD) of  the normalized $\bX'$ as
\[
\frac{1}{\sqrt{np}}\bX'~= ~\sum_{i=1}^{p} d_i  u_i v_i^\T.
\]
Let $U_r = (u_1,\ldots, u_r)\in \OO_{n\times r}$, $V_r = (v_1,\ldots,v_r)\in \OO_{p\times r}$  and  $D_r = \diag(d_1,\ldots, d_r)$ containing the largest $r$ singular values. 
It is known (see, for instance, \cite{Bai-factor-model-03}) that the constrained optimization in \eqref{crit_Z_hat}  admits the following closed-form solution
\begin{equation}\label{def_B_hat}
    \wh A = \sqrt{p}~  V_r ,\qquad \wh\bZ' = \sqrt{n}~ U_r D_r = \bX'   (\wh A  /  p ) =: \bX' \wh B^\T.
\end{equation}  
Given $\wh B \in \RR^{r\times p}$ as above, we can predict $Z_1,\ldots, Z_n$ in the training data  by $\wh B X_1, \ldots, \wh B X_n$ and then proceed with the kernel ridge regression in \eqref{def_f_hat} of Section \ref{sec_theory_KRR}  to obtain the regression predictor $\wh f$. For a new data point $X$ from model \eqref{model_X}, its final prediction is 
\[
  (\wh f\circ \wh B) (X) =  \frac{1}{\sqrt{n}}\sum_{i=1}^n \wh \alpha_i ~ K(\wh B X_i, \wh B 
 X).
\]
Its excess risk bound is stated in the next section, as an application of our theory in \cref{sec_theory_KRR}.

    \subsection{Upper bounds of the excess risk of \texorpdfstring{$\wh f \circ \wh B$}{\texttwoinferior}}\label{sec_theory_PCA}

    From our theory in \cref{sec_theory_KRR}, by writing  $\sz = \Cov(Z)$ and $\sw = \Cov(W)$,
    prediction error of $\wh g(X) = \wh B X$ appears in the excess risk bound in the form of  
    \begin{equation}\label{def_eps_g_linear}
        \EE\|Z- \wh B X\|^2_2 =  \|\sz^{1/2}(\wh B A - \bI_r)\|_\F^2 + \|\sw^{1/2}\wh B^\T\|_\F^2.
    \end{equation} 
   The second equality is due to the independence between $Z$ and $W$. The matrix $\wh B$ must therefore ensure that $\|\wh B A - \bI_r\|_\F$ vanishes. Since this essentially requires identifying  $A$, we adopt  the following identifiability and regularity conditions commonly used in the factor model literature (see, for instance, \cite{Bai-factor-model-03,fan2013large,bai2020simpler}).  Throughout this section, we also treat the number of latent factors $r$ as fixed. 

 \begin{assumption}\label{ass_A}
     The matrix $A$ and $\sz$ satisfy 
     $p^{-1}A^\T A = \bI_r$ and $\sz = \diag(\lambda_1, \ldots \lambda_r)$ with $\lambda_1, \ldots, \lambda_r$ being distinct. Moreover, there exist some absolute constants $0 < c\le C<\i$ such that $\max_{j\in[p]}\|[\sw]_{j\cdot}\|_1 < C$ and $c\le \lambda_r \le \lambda_1 \le C$.
 \end{assumption}

 \begin{assumption}\label{ass_tails_ZW}
     Both random vectors $Z$ and $W$ under model \eqref{model_X} have sub-Gaussian tails, that is, $\EE[\exp(u^\T Z)] \le \|u\|_2^2\gamma_z^2$ and $\EE[\exp(v^\T W)] \le  \|v\|_2^2\gamma_w^2$ for all $u\in \RR^r$ and $v\in \RR^p$.
 \end{assumption}

    Both \cref{ass_A} and \cref{ass_tails_ZW} are commonly assumed in the literature of factor models with the number of features allowed to diverge \citep{Bai-factor-model-03,fan2013large}. 
    The requirement of $\lambda_1, \ldots, \lambda_r$ being distinct in \eqref{ass_A}   ensures that  $Z$ can be consistently predicted. Such requirement can be dropped when using kernels invariant under orthogonal transformations (see also \cref{sec_theory_latent}).
    
    Under Assumptions \ref{ass_A} and \ref{ass_tails_ZW}, the existing literature (see, for instance, \cite{bai2020simpler}) ensures that 
    $
            \EE\|Z- \wh B X\|^2_2  = \cO_\PP(  n^{-1}  +p^{-1}).
    $
    Consequently, \cref{thm_risk} together with $\EE\Delta_{\wh g} \lesssim \EE\|Z- \wh B X\|^2_2$ and \cref{def_eps_g_linear} yields the following excess risk bounds of $\wh f\circ \wh B$.  
    
    \begin{corollary}\label{thm_PCA}
        Under conditions of \cref{thm_risk}, further grant model \eqref{model_X} with Assumptions \ref{ass_A} and \ref{ass_tails_ZW}. By choosing $\lambda  \asymp \delta_n \log(1 /\eta) +    \log(1 /\eta) / n  + 1/p + {\|\fh-f^*\|_\rho^2}$
        in \eqref{def_f_hat},  one has 
        \begin{align*}
        \cE(\wh f\circ \wh B)
         ~ = ~  \cO_{\PP}\left(  \delta_n  +     { {1\over n }} + {1\over p}  + {\|\fh-f^*\|_\rho^2} \right)
        \end{align*}  
    where $\cO_\PP$ is  with respect to the law of both $\cD$ and $\bX'$.
    \end{corollary}

The term $1/p$ arises exclusively from using PCA to predict the latent factors under Assumptions \ref{ass_A} and \ref{ass_tails_ZW}. For linear kernels and $f_*\in \cH_K$, we reduce to the factor-based linear regression models \eqref{model_Y_linear} and \eqref{model_X}. Our results coincide with the excess risk bounds of PCR derived in \cite{bing2021prediction}. When $f^*$ belongs to a hierarchical composition of functions in the H\"{o}lder class with some smoothness parameter $\gamma>0$, \cite{fan2024factor} derived the optimal prediction risk  $(\log n / n)^{-{2\gamma\over 2\gamma + 1}} + \log(p) / n + 1 / p$, achieved by regressing the response onto the leading PCs via deep neural nets. \cref{thm_PCA} on the other hand provides the excess risk bound of kernel ridge regression using the leading PCs under a broader function class. The obtained rate in the next section is shown to be minimax optimal over a wide range of kernel functions.
    


\subsection{Minimax lower bound of the excess risk
}\label{sec_theory_lb}

To benchmark the obtained upper bounds in \cref{thm_PCA}, 
we  establish the minimax lower bounds  
of the excess risk $\cE(h)$  for any measurable function $h: \RR^p \to \RR$ under models  \eqref{model_X} and \eqref{model}. Although the factor model  \eqref{model_X} is a particular instance for modeling the relationship between $X$ and $Z$, the obtained minimax lower bounds remain valid for the space of joint distributions of $Z$ and $X$ under more general dependence. To specify the RKHS for model \eqref{model}, we consider the following class of kernel functions  
\[
    \cK := \b\{K ~  \text{is regular,  universal, and satisfies Assumptions \ref{ass_bd_K} \& \ref{ass_mercer}} \b \}. 
\]
Regular kernels were already required in the past work of \citet{yang2017randomized} to derive the minimax lower bounds for estimating $f^*$ under the classical KRR setting with a fixed design. 
In our case, we further need the kernel function to be universal in order to characterize the effect of 
predicting $Z$. As mentioned in \cref{rem_approx_error}, there is no approximation error for universal kernels under mild assumptions. Our lower bounds below will hold pointwise for each $K\in \cK$. 

For the purpose of establishing minimax lower bounds, it suffices to choose $\epsilon \sim N(0, \sigma^2)$ under model \eqref{model} and $(Z, W)$ being jointly Gaussian under model \eqref{model_X}. For this setting, we use $\PP_\theta$ to denote the set of all distributions of $\cD$, parametrized by 
\[
    \theta \in \Theta :=\l\{(A, \sz, \sw, f^* ) :  
    \text{$A, \sz$ and $\sw$ satisfy \cref{ass_A}}, ~ f^*\in \cH_K
    \r\}.
\]
Let  $\EE_{\theta}$ denote its corresponding expectation.
The following theorem states the minimax lower bounds of the excess risk under the above specification. 

\begin{theorem}\label{thm_lb_excess_risk}
Under  models \eqref{model_X} and \eqref{model}, 
consider any kernel function $K\in \cK$. 
Assume $n$ is sufficiently large such that $d(\delta_n) \ge 128 \log 2$. 
There exists some constant $c>0$ depending on $\sigma^2$ only and some absolute constant $c'>0$ such that  
\begin{align}\label{eq_lb_excess_risk}
     \inf_{h} \sup_{\theta \in \Theta}~  \EE_{\theta} \l[   f^* (Z)-  h   (X)  \r]^2  ~  \ge~     c ~ \delta_n  +  {c'\over p}.
\end{align}
The infimum is over all measurable functions $h: \cX \to \RR$  constructed from $\cD$.
\end{theorem}  

The minimax lower bound in \eqref{eq_lb_excess_risk} consists of two terms: the first term $\delta_n$ accounts for the model complexity of $\cH_K$, which depends on the kernel function $K$, while the second term  $1/p$ reflects the irreducible error due to not observing  $Z$. In conjunction with \cref{thm_PCA}, by noting that  $\delta_n \asymp d(\delta_n)/n \ge 1/n$ and $\|\fh - f^*\|_\rho = 0$, the lower bound in \cref{thm_lb_excess_risk} is also tight and the predictor $\wh f\circ \wh B$ is minimax optimal. 

Finally, we remark that although our lower bound in \cref{thm_lb_excess_risk} is stated for universal kernels, inspecting its proof reveals that the same results hold for kernels that induce an RKHS containing all linear functions.  

\section{Extension to other loss functions}\label{sec_general_loss}
Although so far we have been focusing on KRR with predicted inputs under the squared loss function, 
our analytical framework can be adapted to general loss functions that are strongly convex and smooth.  In this section, we give details on such extension.

Specifically,  let $L(\cdot, \cdot): \RR \times \RR \to \RR$ be a general loss function 
with $f^*$ being the best predictor of $Y$ over all measurable functions $f: \cZ\to \RR$, that is, 
$$
    f^*(z) = \argmin_{f}~ \EE[L(Y,f(z)) \mid Z = z],\qquad\forall~ z\in \cZ.
$$
Fix any measurable function $\wh g: \cX\to \cZ$ that is used to predict the feature inputs $Z_1,\ldots, Z_n$. 
For the specified  loss function $L$ and the predicted inputs $\wh g(X_1),\ldots,\wh g(X_n)$,  we solve the following optimization problem: 
\begin{align}\label{def_f_hat_lip}
\widehat{f} = \argmin_{f\in \mathcal{H}_K}  \Big\{ \frac{1}{n} \sum_{i=1}^nL\l(Y_i,(f\circ \wh g)(X_i)\r) +\lambda\|f\|_{K}^2 \Big\}. 
\end{align} 
For a new data point $X\in \cX$, we predict its corresponding 
response $Y$ 
via
$(\wh f\circ \wh g) (X)$. 
Our target of this section  is to bound from above 
the excess risk of $\wh f\circ\wh g$,
\[
\cE(\wh f\circ \wh g):=\EE\l[ L(Y, (\wh f\circ \wh g)(X) ) - L(Y, f^*(Z))\r]. 
\]
In addition to Assumptions \ref{ass_f_H}--\ref{ass_mercer}, 
our analysis for general loss functions requires the following two 
assumptions. Recall that $\cF_b$ is given in \eqref{def_local_ball}.
 
\begin{assumption} \label{assu_lip_loss}
For any $y\in \RR$, the loss function  $L(y,\cdot)$ is convex. 
 Moreover, for any $y\in \RR$ and $f, f' \in \cF_b$, there exists a constant $C_\ell$ such that
    \[
    \l| L(y,(f\circ \wh g)(x)) -  L(y,(f'\circ \wh g)(x))\r|
      \le    C_\ell \l|(f\circ \wh g)(x)-(f' \circ \wh g)(x)\r|,\quad \forall~x\in \cX. 
    \]
\end{assumption}
Many commonly used loss functions satisfy \cref{assu_lip_loss}. For example, 
it holds for the logistic loss and exponential loss which are commonly used for classification problems. 
Other examples include the check loss used for quantile regressions, the hinge loss used for margin-based classification problems, and the Huber loss adopted in robust regressions. As a result of \cref{assu_lip_loss} together with the boundedness of $\cF_b$ in \eqref{def_local_ball}, the quantity $|L(y, (f\circ \wh g)(x))-  L(y,(f'\circ \wh g)(x))|$ is bounded uniformly over $f,f'\in \cF_b$, $x\in \cX$ and  $y\in \RR$. 

\begin{assumption} \label{assu_convex_and_smooth}
There exist two constants $0< C_L\le C_U <\i $ such that for all $f\in \cF_b $ and $\wh g $, 
\begin{align}\label{convex_smooth}
C_L ~  \EE\l[(f\circ \wh g)(X)-f^*(Z)\r]^2 \le
\cE(f\circ \wh g) 
\le C_U ~  \EE\l[ (f\circ \wh g)(X) - f^*(Z) \r]^2 . 
\end{align}
\end{assumption}
 
The first inequality in \eqref{convex_smooth}  is the strongly convex condition while the second one corresponds to the smoothness condition. When $Z$ is observable,  \cref{assu_convex_and_smooth} 
reduces to 
\[
\EE\l[ L(Y, f(Z)  ) - L(Y, f^*(Z))\r] 
\asymp \|f-f^*\|_\rho^2,   \qquad \forall~ f\in \cF_b,
\]
a condition, or its variants, frequently adopted in the existing literature for analyzing general loss functions  \citep{steinwart2008support, wei2017early,li2019towards,farrell2021deep}. 

It can be seen from \eqref{stru_assu} that \cref{assu_convex_and_smooth} holds for the squared loss function with  $C_L = C_U =1$. For more general loss, \cref{assu_convex_and_smooth}  essentially requires that the excess risk maintains a similar curvature as the $\cL^2$-norm around the minimizer $f^*$. Indeed, by following the same argument in  \cite{wei2017early}, it can be verified that   \cref{assu_convex_and_smooth}  holds for both the logistic loss and the exponential loss, with constants  $C_L$ and $C_U$ depending only on $\kappa, \|f^*\|_\i$ and the radius of $\cF_b$. 
For the check loss,
\cref{assu_convex_and_smooth} is satisfied when the conditional distribution of $Y$ given $Z$ is well-behaved, for instance, for any $z\in \cZ$ and $y\in \RR$, 
\begin{align} \label{condi_quantile}
    C_L |y| ~ \le ~ \l| F_{Y\mid Z=z} (f^*(z) +y)- F_{Y\mid Z=z} (f^*(z))\r| ~\le ~  C_U |y|,
\end{align}
where $F_{Y\mid Z=z} $ denotes the conditional distribution function of $Y$ given $Z=z$.   The fact that \eqref{condi_quantile} implies \cref{assu_convex_and_smooth}  can be easily verified by using Knight’s identity \citep{knight1998limiting} and adapting the argument in \cite{belloni2011}. 
In \cref{tab_loss} we summarize some common loss functions that satisfy Assumptions \ref{assu_lip_loss} and \ref{assu_convex_and_smooth}.


\begin{table}[ht]
    \centering
    \renewcommand{\arraystretch}{0.8} 
    \setlength{\tabcolsep}{7pt} 
    \caption{Losses that satisfy Assumptions \ref{assu_lip_loss} and \ref{assu_convex_and_smooth}}
    \label{tab_loss}
    \begin{tabular}{l|c|c|c}
    \toprule 
    Loss functions    & Exponential & Logistic & Check \\
    \midrule 
    \cref{assu_lip_loss}    & \checkmark  & \checkmark & \checkmark \\ 
    \cref{assu_convex_and_smooth}  & \checkmark  & \checkmark & under \eqref{condi_quantile} \\
    \bottomrule 
    \end{tabular}
\end{table}
 
Recall the kernel-related latent error 
from \eqref{def_Delta}
and  $\delta_n$ as the fixed point of $R(\delta)$ in \eqref{kernel_complexity}.  The following theorem provides non-asymptotic upper bounds of the
excess risk of $\wh f\circ \wh g$ with $\wh f$ given in \eqref{def_f_hat_lip}. 

\begin{theorem}\label{thm_risk_lip}
Grant  Assumptions  \ref{ass_f_H}--\ref{ass_mercer}, \ref{assu_lip_loss} and \ref{assu_convex_and_smooth}.   
For  any  $\eta\in (0,1)$, by choosing $\lambda$ in \eqref{def_f_hat_lip}  such that  
\begin{equation}\label{lb_lambda_lip} 
\lambda  =  C\l(      \delta_n\log(1/\eta)  +\EE \Delta_{\wh g}  +  \|\fh-f^*\|_\rho^2 + \frac{\log(1/\eta)}{n} \r),
\end{equation}
with probability at least $1-\eta$, one has
\begin{align}\label{bd_pred_rate_lip}
    \cE(\wh f\circ \wh g)  
      ~ \le ~  C'\l(  \delta_n\log(1/\eta)  +\EE \Delta_{\wh g}   +  \|\fh-f^*\|_\rho^2+ \frac{\log(1/\eta)}{n} \r) . 
\end{align}  
Here both positive constants $C$ and $C'$ depend only on $\kappa$, $\|\fh\|_K$,  $C_\ell,C_L$ and $C_U$.
\end{theorem}

\begin{proof}
   Its proof can be found in \cref{app_sec_pf_thm_risk_lip}. 
\end{proof}

The risk bound in \eqref{bd_pred_rate_lip} of \cref{thm_risk_lip} consists of the same components as in \eqref{bd_pred_rate} of \cref{thm_risk}. Under the factor model \eqref{model_X}, it readily yields the risk bound when $\wh g$ is chosen as the linear predictor by using PCA.

\begin{remark}
    By inspecting the proof of \cref{thm_risk_lip}, it can be seen that under the classical setting where $Z$ is observable and $f^*\in \cH_K$,  the smoothness condition, the second inequality in \cref{assu_convex_and_smooth},  can be dropped. 
\end{remark}

\section{Simulations}\label{app_sec_sim}
In this section, we evaluate the predictive performance of KRR using predicted features. We consider models \eqref{model_X}
 and \eqref{model} 
where the latent vector $Z=(Z_1,Z_2,Z_3)^\top$, with $r = 3$, is generated with entries independently sampled from 
$\text{Unif}(0,1)$, the regression function $f^*$ is set as 
\[
f^*(Z) = 2 \sin  (3 \pi Z_1  )+ 3 | Z_1 - 0.5|- \exp  (Z_2^2 -Z_3^2  ),
\]
the regression error $ \epsilon \sim N(0,0.8^2)$,  the loading matrix $A$  
 is generated with rows independently sampled from 
$N(0, \diag(10 ,  5.5,  1 ))$ and $W$ is generated with entries independently sampled from 
$ N(0, 1.5^2)$.  
The training data $\cT = \{(Y_i, X_i,Z_i)\}_{i=1}^n$, the auxiliary data $ \bX' = (X'_1,\ldots, X'_n)^\T$ and the  test dataset $\cT^t=  \{(Y_i^t, X_i^{t}, Z_i^t)\}_{i=1}^m$ are generated i.i.d. according to models \eqref{model_X} and \eqref{model}. 

Under model \eqref{model_X},  as  discussed  in \cref{sec_PCA}, we predict $\bZ = (Z_1,\ldots, Z_n)^\T$ by  
\[
\wh\bZ' 
= (\wh B' X_1 ,\ldots, \wh B' X_n)^\T  
\]
with $\wh B'$ constructed  from  $\bX'$ via PCA as in \eqref{def_B_hat}. 
We also consider predicting $\bZ$ by  
\[
\wh\bZ= (\wh B X_1 ,\ldots, \wh B X_n)^\T 
\]
with $\wh B$
obtained by applying PCA on 
$\bX = (X_1,\ldots, X_n)^\T$.  
The following predictors are considered in comparison: (1) KRR-$\wh\bZ'$: KRR in \eqref{def_f_hat} with predicted inputs using $\wh\bZ'$;  (2) KRR-$\wh \bZ$:  KRR in \eqref{def_f_hat} with predicted inputs using $\wh\bZ$; (3) KRR-$\bZ$:  KRR in \eqref{def_f_hat} using the true $ \bZ$; (4) KRR-$\bX$: KRR from  \eqref{def_f_hat}  with $ X_i$  in place of   $\wh g(X_i)$; and  (5) LR-$\bZ$: Regressing  $\bY$
    onto  $\bZ$ by  OLS. 
Note that KRR-$\bX$ is included for comparing the proposed KRR-$\wh\bZ'$ with the classical KRR method. 
The linear predictor LR-$\bZ$ is considered to illustrate the benefit of using non-linear predictors.
KRR-$\bZ$  uses the true $\bZ$ to illustrate the additional prediction errors of KRR-$\wh\bZ'$  due to predicting $\bZ$. Finally, 
KRR-$\wh \bZ$ is included to examine the effect of using an auxiliary data set to compute $\wh B$.

For the kernel-based methods, we select the regularization parameter $\lambda$ in \eqref{def_f_hat}  via 3-fold cross-validation and choose the Gaussian kernel, $K(z,z') = \exp(-\|z-z'\|_2^2 / \sigma^2_n)$, with $\sigma_n$ set as the median of all pairwise distances among sample points, as suggested in \cite{mukherjee2010learning}.  
For any predictor $\wh  f :  \RR^p\to \RR$, we evaluate its predictive performance by
$m^{-1} \sum_{i=1}^m ( Y^t_i - \wh f(X^t_i))^2.$
For any predictor $\wh f:  \RR^r\to \RR$ using the true $\bZ$, we evaluate its performance by 
$m^{-1} \sum_{i=1}^m ( Y^t_i - \wh f(Z^t_i))^2$. 
These metrics are computed based on $100$ repetitions in each setting. 

We examine the effect of the following parameters on the predictive performance by varying them one at a time: (1) the sample size $n\in \{200, 600,  1000,  1400, 1800, 2200\}$; (2) the number of features  $p\in \{50, 100, 500, 1000, 1500, 3000\}$; (3) the  signal-to-noise ratio  (SNR) of predicting the latent factor $Z$ in the presence of $W$, defined as $\text{SNR}:=  \lambda_r(A\Sigma_Z A^\T) /\|\Sigma_W\|_\op$,  by multiplying $A$ by a scalar $\alpha \in \{ 0.1,0.2, 0.3, 0.5,
0.75,1,1.5\}$. We fix $n =800$, $p=2000$ and $\text{SNR} = 1$ whenever they are not varied, and set $m=n$ in all settings. 
The averaged prediction errors in each setting 
are reported in \cref{tab_all}. 

\begin{table}[ht]
\centering
\caption{The averaged performance of all the predictors with varying $n$, $p$, and $\alpha$.}
\label{tab_all}
\begin{tabular}{c}\vspace{1em}
  \resizebox{0.9\textwidth}{!}{%
    \begin{tabular}{lcccccc}
    \toprule
    Sample size $n$ & 200 & 600 & 1000 & 1400 & 1800 & 2200 \\
    \midrule
    KRR-$\wh\bZ'$ & 1.59 {\footnotesize (0.32)} & 0.90 {\footnotesize (0.06)} & 0.82 {\footnotesize (0.04)} & 0.79 {\footnotesize (0.04)} & 0.77 {\footnotesize (0.03)} & 0.75 {\footnotesize (0.02)} \\
    KRR-$\wh \bZ$ & 1.63 {\footnotesize (0.32)} & 0.90 {\footnotesize (0.07)} & 0.81 {\footnotesize (0.04)} & 0.79 {\footnotesize (0.03)} & 0.77 {\footnotesize (0.03)} & 0.75 {\footnotesize (0.02)} \\
    KRR-$\bZ$     & 1.08 {\footnotesize (0.19)} & 0.77 {\footnotesize (0.05)} & 0.72 {\footnotesize (0.03)} & 0.71 {\footnotesize (0.03)} & 0.70 {\footnotesize (0.03)} & 0.68 {\footnotesize (0.02)} \\
    KRR-$\bX$     & 1.99 {\footnotesize (0.19)} & 1.91 {\footnotesize (0.10)} & 1.89 {\footnotesize (0.08)} & 1.85 {\footnotesize (0.07)} & 1.83 {\footnotesize (0.06)} & 1.82 {\footnotesize (0.05)} \\
    LR-$\bZ$      & 3.61 {\footnotesize (0.29)} & 3.58 {\footnotesize (0.16)} & 3.62 {\footnotesize (0.13)} & 3.60 {\footnotesize (0.10)} & 3.61 {\footnotesize (0.11)} & 3.60 {\footnotesize (0.08)} \\
    \bottomrule
    \end{tabular}
}\\[1em]
\vspace{1em}
\resizebox{0.9\textwidth}{!}{%
    \begin{tabular}{lcccccc}
    \toprule
    Dimension $p$ & 50 & 100 & 500 & 1000 & 1500 & 3000 \\
    \midrule
    KRR-$\wh\bZ'$ & 1.85 {\footnotesize (0.15)} & 1.44 {\footnotesize (0.12)} & 0.95 {\footnotesize (0.06)} & 0.89 {\footnotesize (0.06)} & 0.87 {\footnotesize (0.05)} & 0.84 {\footnotesize (0.05)} \\
    KRR-$\wh \bZ$ & 1.86 {\footnotesize (0.16)} & 1.43 {\footnotesize (0.12)} & 0.94 {\footnotesize (0.06)} & 0.89 {\footnotesize (0.06)} & 0.87 {\footnotesize (0.05)} & 0.83 {\footnotesize (0.05)} \\
    KRR-$\bZ$     & 0.75 {\footnotesize (0.04)} & 0.74 {\footnotesize (0.04)} & 0.74 {\footnotesize (0.04)} & 0.75 {\footnotesize (0.03)} & 0.75 {\footnotesize (0.04)} & 0.74 {\footnotesize (0.04)} \\
    KRR-$\bX$     & 2.32 {\footnotesize (0.17)} & 2.14 {\footnotesize (0.13)} & 1.95 {\footnotesize (0.10)} & 1.94 {\footnotesize (0.12)} & 1.90 {\footnotesize (0.10)} & 1.83 {\footnotesize (0.09)} \\
    LR-$\bZ$      & 3.60 {\footnotesize (0.15)} & 3.60 {\footnotesize (0.14)} & 3.59 {\footnotesize (0.14)} & 3.62 {\footnotesize (0.12)} & 3.60 {\footnotesize (0.15)} & 3.60 {\footnotesize (0.14)} \\
    \bottomrule
    \end{tabular}
  }
\\[1em]

  \resizebox{\textwidth}{!}{%
    \begin{tabular}{p{2cm}ccccccc}
    \toprule
    SNR  $\alpha$ & 0.1 & 0.2 & 0.3 & 0.5 & 0.75 & 1 & 1.5 \\
    \midrule
    KRR-$\wh\bZ'$ & 2.59 {\footnotesize (0.13)} & 1.46 {\footnotesize (0.10)} & 1.16 {\footnotesize (0.07)} & 0.94 {\footnotesize (0.06)} & 0.87 {\footnotesize (0.06)} & 0.85 {\footnotesize (0.05)} & 0.82 {\footnotesize (0.05)} \\
    KRR-$\wh \bZ$ & 2.83 {\footnotesize (0.15)} & 1.56 {\footnotesize (0.13)} & 1.14 {\footnotesize (0.06)} & 0.93 {\footnotesize (0.06)} & 0.86 {\footnotesize (0.05)} & 0.84 {\footnotesize (0.05)} & 0.82 {\footnotesize (0.05)} \\
    KRR-$\bZ$     & 0.74 {\footnotesize (0.04)} & 0.75 {\footnotesize (0.04)} & 0.74 {\footnotesize (0.04)} & 0.74 {\footnotesize (0.04)} & 0.74 {\footnotesize (0.04)} & 0.74 {\footnotesize (0.04)} & 0.74 {\footnotesize (0.04)} \\
    KRR-$\bX$     & 3.79 {\footnotesize (0.16)} & 3.74 {\footnotesize (0.16)} & 3.47 {\footnotesize (0.14)} & 2.53 {\footnotesize (0.11)} & 2.11 {\footnotesize (0.10)} & 1.88 {\footnotesize (0.08)} & 1.39 {\footnotesize (0.07)} \\
    LR-$\bZ$      & 3.62 {\footnotesize (0.14)} & 3.59 {\footnotesize (0.15)} & 3.61 {\footnotesize (0.14)} & 3.59 {\footnotesize (0.14)} & 3.58 {\footnotesize (0.13)} & 3.61 {\footnotesize (0.15)} & 3.61 {\footnotesize (0.15)} \\
    \bottomrule
    \end{tabular}
  }
\end{tabular}
\end{table}

As shown in the first two panels of 
\cref{tab_all}, the proposed KRR-$\wh\bZ'$ outperforms KRR-$\bX$ and LR-$\bZ$. Comparing to the oracle approach KRR-$\bZ$, KRR-$\wh\bZ'$ has increasingly closer performance as $n$ or $p$ grows. This is in line with our risk bound in \cref{thm_PCA} where the extra prediction error of  KRR-$\wh\bZ'$ relative to
KRR-$\bZ$ is  quantified by $\cO_\PP(1/p+1/n)$,  which vanishes with increasing $n$ and $p$.
Comparing to KRR-$\wh \bZ$ that does not use auxiliary data, the benefit of using auxiliary data in KRR-$\wh \bZ'$ is evident for small sample sizes.
 As $n$ increases, the prediction error of KRR-$\bX$ does not decrease as fast as those of KRR-$\wh\bZ'$ and KRR-$\bZ$, indicating a slower rate of convergence of the former. This is due to that KRR-$\bX$ does not exploit the low-dimensional structure.  
As $p$ increases, we also note that the performance of KRR-$\bX$ slightly increases. One explanation could be that more useful information for predicting $Y$ can be found from  $\bX$  as $p$ increases.  Since KRR-$\bZ$ and  LR-$\bZ$ use $\bZ$ instead of $\bX$, their prediction errors remain the same as $p$ increases. 

Regarding the effect of SNR,  a higher SNR implies that $\bX$ contains more information about $\bZ$, which should lead to better prediction of $\bZ$ and consequently of $Y$. 
 As seen in the third panel of \cref{tab_all},  the predictive performance of KRR-$\wh\bZ'$, KRR-$\wh\bZ$, and KRR-$\bX$ indeed improves as the SNR increases, whereas the prediction errors of KRR-$\bZ$ and LS-$\bZ$ remain unchanged. For large values of SNR, KRR-$\wh\bZ'$ performs very well, approaching to  KRR-$\bZ$, whereas  KRR-$\bX$ still performs poorly, due to the fact that it does not fully exploit the low-dimensional structure. We also observe from the comparison between KRR-$\wh\bZ'$ and KRR-$\wh \bZ$ that the benefit of using auxiliary data becomes more visible for small SNRs. 

 Finally, comparing to LS-$\bZ$ in all settings, KRR-$\wh\bZ'$ has substantial predictive advantage, implying the benefit of using KRR to capture the non-linear relationship between the response and the latent factor.

\section{Real data analysis}\label{sec_real_data}
In this section, we present real data analysis on two datasets based on  pretrained LLM embeddings, one tailored to a regression problem and the other to a classification problem.

\subsection{Analysis on Letterboxd Film Dataset}\label{sec_real_data_regress}
In this section, we investigate the predictive performance of KRR using pre-trained embedding vectors from movie reviews in the Letterboxd Film Dataset \citep{letterboxd_film_dataset_2025}.
The dataset contains information on 847,209 films from the Letterboxd platform, each with multiple textual reviews and an average user rating on a 1–5 scale. Our goal is to train a predictive model that can accurately predict the rating score from the reviews of a given film.

Unlike the classical regression setting, this task involves textual inputs that must be converted to numerical embedding vectors before applying KRR. We follow the procedure in Example 3 of the Introduction, first applying a pretrained language model to transform the textual review content into embedding vectors in $\RR^d$. To this end, we choose the General Text Embeddings (GTE) model  \cite{li2023towards}, which is trained based on the BERT framework \cite{devlin2019bert}. By leveraging multi-stage contrastive learning,  the GTE model efficiently captures sentence-level similarity in order to generate text embedding representations, rendering it broadly applicable to various natural language tasks. 
The GTE model provides three different sizes of models: \texttt{thenlper/gte-small}, \texttt{thenlper/gte-base}, and \texttt{thenlper/gte-large}, producing embedding vectors with $ d =384, 768$,  and $1024$, respectively.
In addition to generating text embeddings, we also investigate whether further exploiting the low dimensional structure of the embedding vectors by applying PCA to reduce their dimensions   offers additional benefits.
Combining these two steps yields the final 
predicted  inputs (see, Figure \ref{Illustration_transformer}, for an illustration), which are used in the downstream regression task, with the rating score of each movie review as the response. 

\begin{figure}[ht]
    \begin{center}
    \scalebox{0.8}{
\begin{tikzpicture}[>=Latex, node distance=3cm] 

\node[ellipse callout, callout relative pointer={(0.2,-0.5)}, 
      draw, text width=3.5cm, align=center] (text) 
      {Great movie, but\\ the story felt flat.};

\node[draw, rounded corners, thick, text width=1.5cm, align=center, right=of text] (vector) 
      {-0.02489\\ -0.05562\\ $\vdots$ \\ 0.02846};

\node[draw, rounded corners, thick, text width=1.5cm, align=center, right=of vector] (pc) 
      {0.8621\\ -0.0542\\ 0.0841};

\draw[-{Latex[length=3.5mm,width=2.5mm]}, line width=1.3pt, draw=black,
      shorten >=5mm, shorten <=5mm] 
    (text.east) -- (vector.west)
    node[midway, above=6mm, fill=blue!30, text=white, 
         inner sep=3pt, rounded corners, align=center] {Embedding};

\draw[-{Latex[length=3.5mm,width=2.5mm]}, line width=1.3pt, draw=black,
      shorten >=5mm, shorten <=5mm] 
    (vector.east) -- (pc.west)
    node[midway, above=6mm, fill=blue!30, text=white, 
         inner sep=3pt, rounded corners, align=center] {PCA};
\end{tikzpicture}}
\end{center}
    \caption{An illustration of the process generating predicted feature inputs with $r=3$}
    \label{Illustration_transformer}
\end{figure}

We consider three methods to predict the rating scores using the predicted inputs: (1) KRR with a Gaussian kernel, (2) KRR with a linear kernel, and (3) ordinary least squares (OLS).
We study the effect of the embedding dimension $d$  by selecting different GTE models. We also vary the retained low-dimension  $r$ in  \eqref{crit_Z_hat}, ranging from 1 to $d$.   We randomly select from the whole dataset  13864 sub-samples and for 
each setting, we randomly select $2000$ samples  for training and another $2000$ samples for testing. The remaining samples  are used in PCA to construct $\wh B$ as in  \eqref{def_B_hat}.
 By repeating 50 times in each setting, we report the averaged mean squared errors (MSE)  on the test data in
Figure \ref{Results_1}. 

\begin{figure}[ht]
    \centering
\includegraphics[width=.325\textwidth]{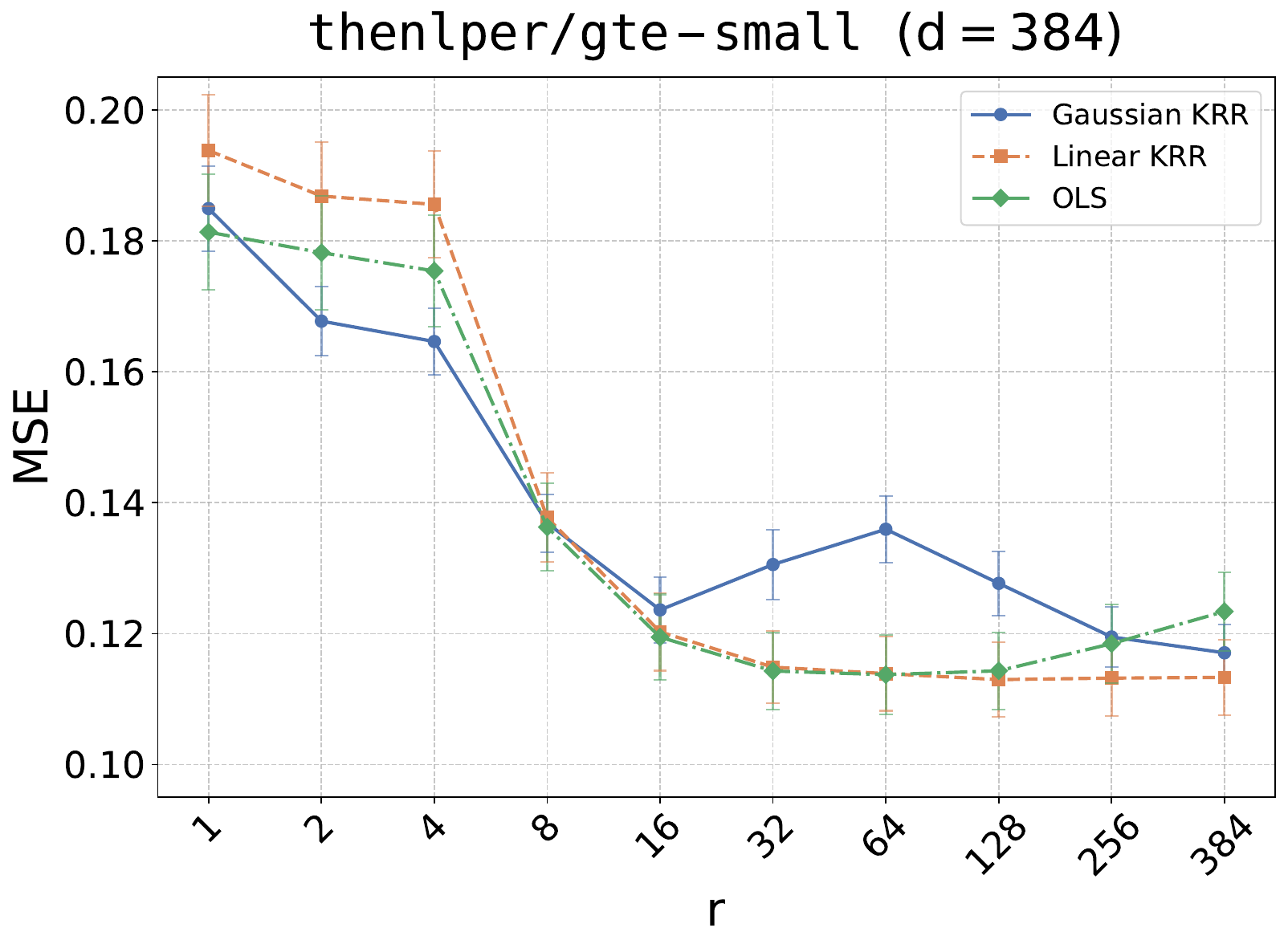}
\includegraphics[width=.325\textwidth]{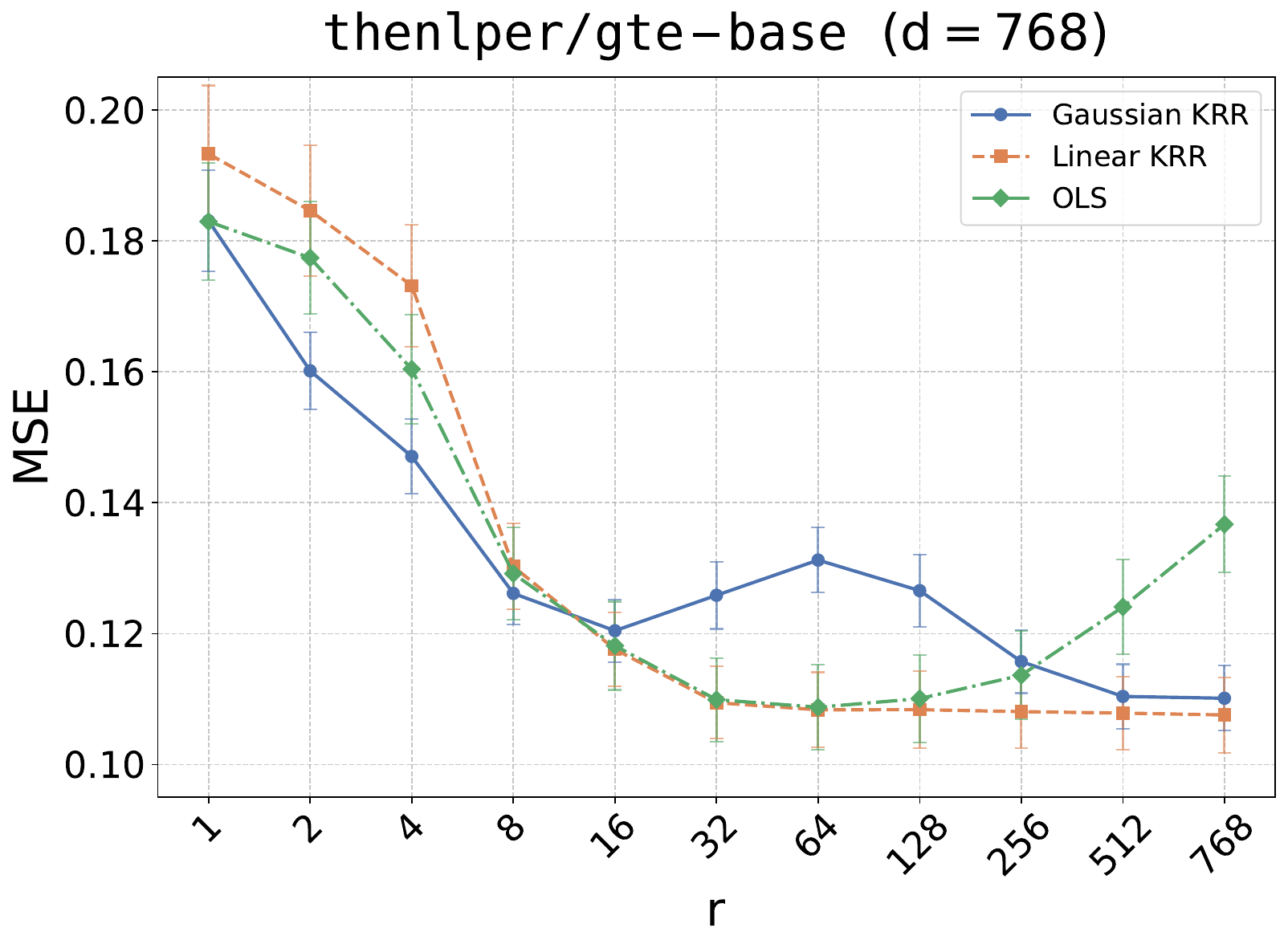}
\includegraphics[width=.325\textwidth]{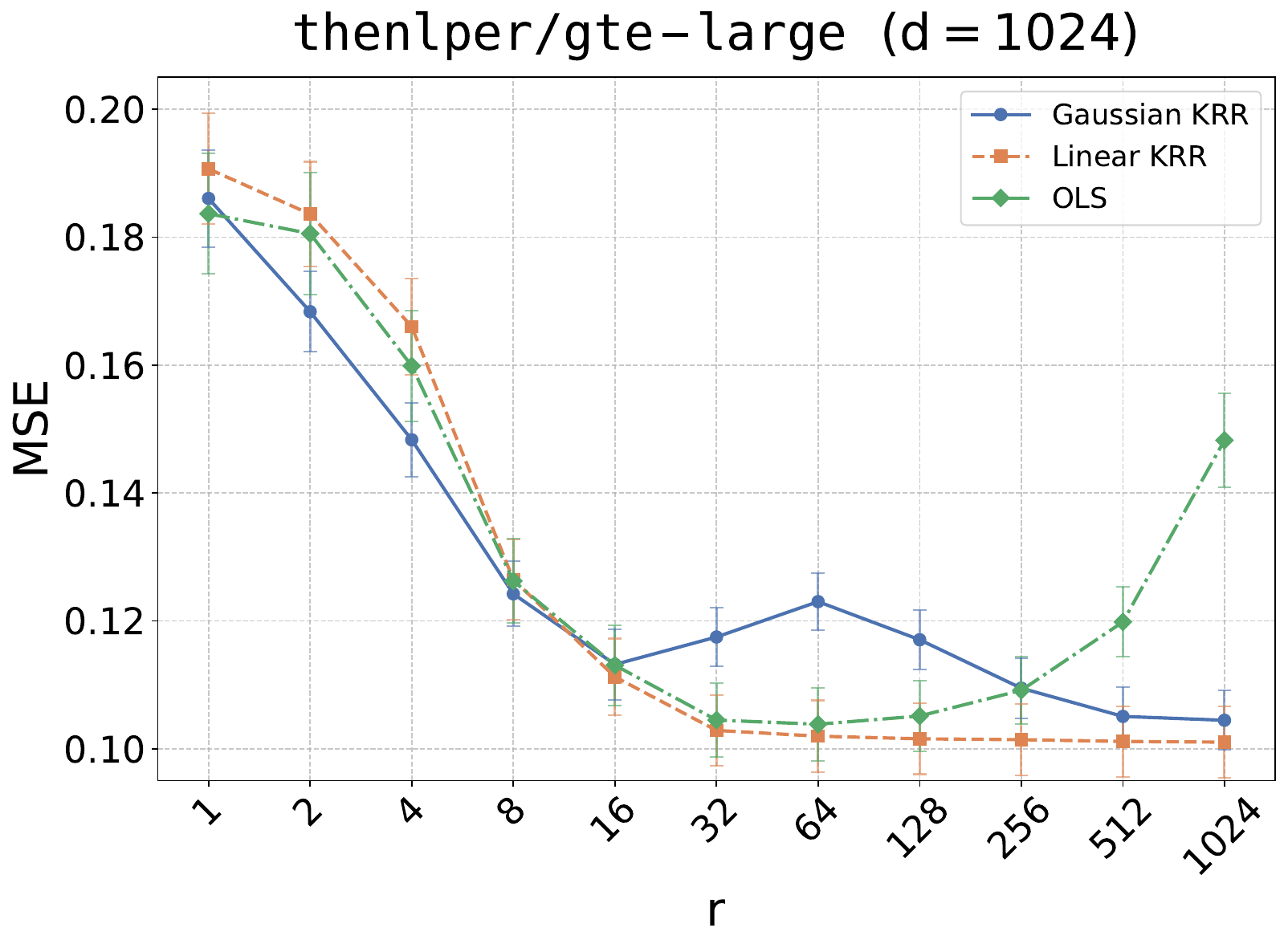}
    \caption{Averaged MSE of different methods
    with varying   $r$ and  GTE models}
    \label{Results_1}
\end{figure}

There are several interesting observations that can be made. First,
we see that for all methods, increasing the dimension of embedding vectors leads to better predictive performance. This is in line with our intuition that pre-trained embedding vectors with larger dimension  tend to encode more useful information for prediction. Second, for a given $d$, linear KRR with $r=d$ has the smallest MSE, suggesting a linear relationship between the embedding vectors and the ratings. It also indicates that such a linear hyperplane in the embedding space cannot be fully explained by its leading principal components. The other linear predictor, OLS, achieves its best performance at an intermediate $r$, striking a balance between variance and bias, as shown in Corollary 5 of \cite{bing2021prediction}. 
In contrast, linear KRR exhibits a clear advantage for relatively large $r$ because the added regularization helps prevent overfitting, in line with \cref{thm_risk} and \cref{cor_K_linear}. 
Third, the Gaussian KRR exhibits interesting predictive performance as $r$ varies. For relatively small $r$, KRR with Gaussian kernel in general outperforms the other methods. This is expected as linear methods are insufficient to characterize the dependence of responses on only a few 
features, leading to under-fitting, whereas KRR with Gaussian kernel is more flexible in capturing the non-linear relationship. 
However, unlike KRR with a linear kernel, whose performance consistently improves as $r$ increases, KRR with a Gaussian kernel shows a slight decline starting at $r=32$. One explanation is that as $r$ grows beyond this point, the model begins to incorporate principal components less relevant to the signal (also evident from the plateau of the MSE for the linear kernel), leading to larger errors in predicting the features. These errors are amplified by the nonlinear Gaussian kernel, whereas the linear kernel is less sensitive to them. As our theory suggests, the impact of feature-prediction errors depends on the kernel, though a more refined analysis is needed to fully characterize this dependence, which we leave to future work. Finally, when $r$ approaches $d$, the last principal components contribute additional signal, and the Gaussian kernel achieves performance comparable to linear KRR. Our overall recommendation is to use linear KRR when the goal is to achieve the best predictive performance with the full embedding, and to use Gaussian KRR when one also seeks to obtain a small set of extracted features for interpretability. 

\subsection{Analysis on Rotten Tomatoes Movie Review Dataset}\label{app_sec_real_data}
 In this section, we validate the kernel learning method with predicted feature inputs  on 
Rotten Tomatoes Movie Review Dataset \citep{Pang+Lee:05a}, which contains 5331 positive and 5331 negative reviews. This benchmark dataset has been widely employed in sentiment categorization research.  For this binary
classification task, the loss function in  \eqref{def_f_hat_lip} is specified 
as the hinge loss, corresponding to the kernel
support machine vector (KSVM) method. 

Similar to the treatment in the text regression analysis in \cref{sec_real_data_regress}, 
we first use pretrained GTE models to generate text embedding vectors, and then apply PCA as in \eqref{crit_Z_hat} to obtain the final predicted feature inputs.
For comparison, we consider three classifiers for sentiment categorization using predicted inputs: (1) KSVM with a Gaussian kernel; (2) KSVM with a linear kernel; (3) logistic regression.
The selection of the 
bandwidth for Gaussian kernel and the regularization parameter $\lambda$ in \eqref{def_f_hat_lip} follows 
the same procedure as in \cref{sec_real_data_regress},
We vary 
the dimension of the text embedding vectors
by selecting different GTE models from \{\texttt{thenlper/gte-small}, \texttt{thenlper/gte-base},  \texttt{thenlper/gte-large}\},
and the retained low-dimension  $r$ in PCA from 1 to the full embedding dimension $d$.  
For each setting, we randomly select  from the overall 
dataset 1000 samples  for training and another 1000 samples for testing, and the remaining data is used to perform PCA. By repeating 50 times in each setting, the averaged accuracy rates are reported in Figure \ref{Results_2}.

\begin{figure}[H]
    \centering
\includegraphics[width=.325\textwidth]{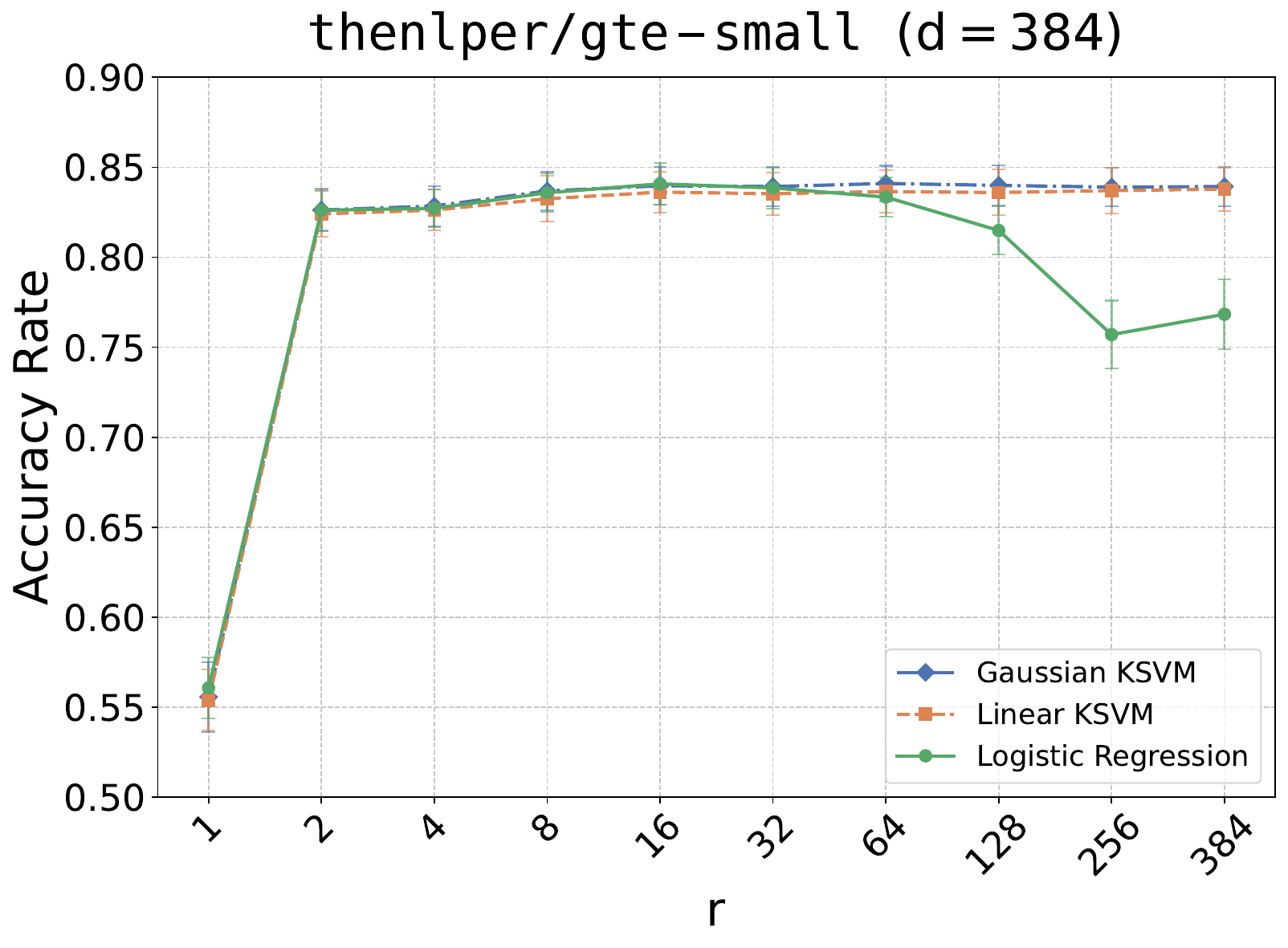}
\includegraphics[width=.325\textwidth]{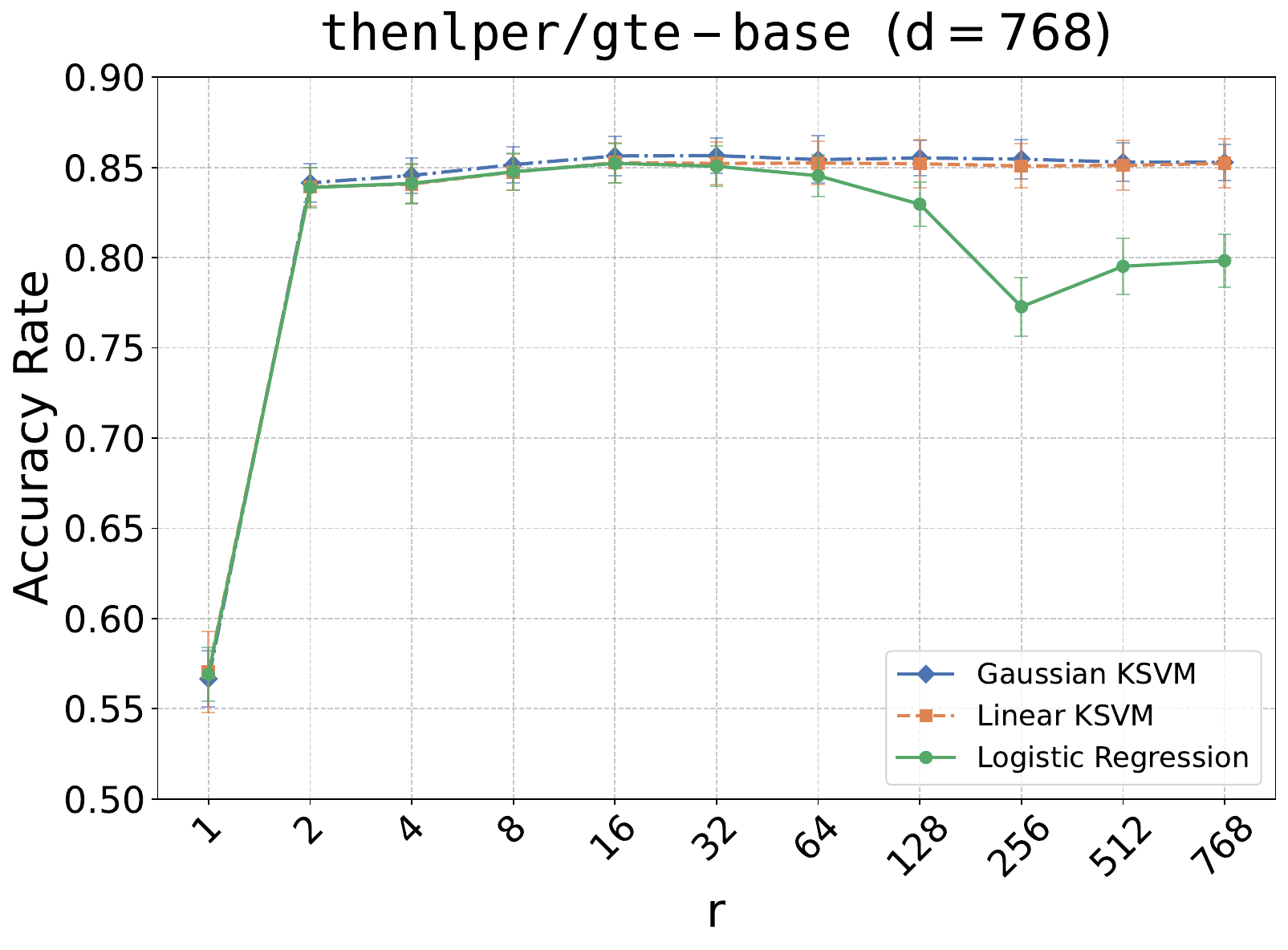}
\includegraphics[width=.325\textwidth]{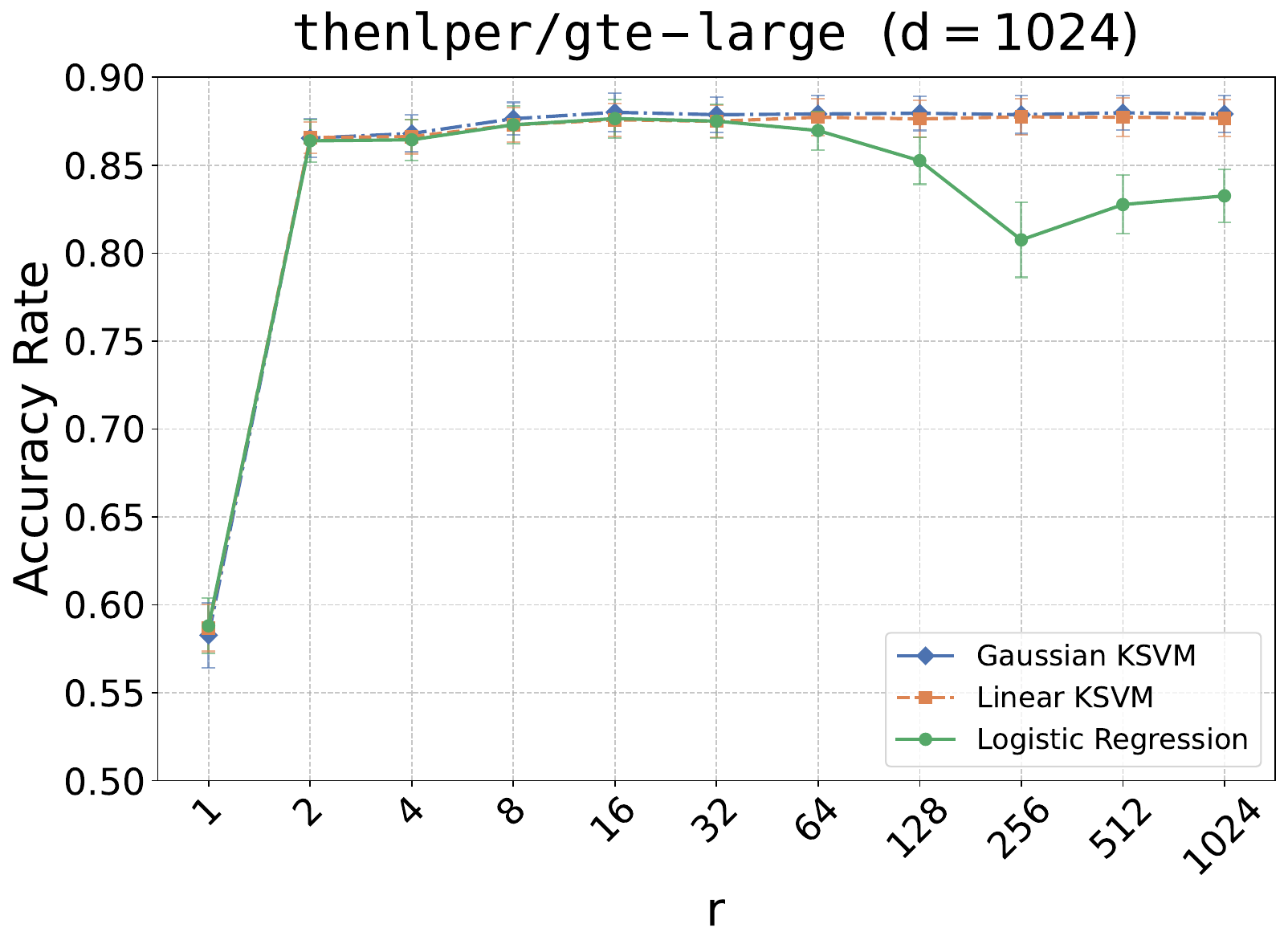}
    \caption{Averaged accuracy rates of different methods
    with varying   $r$ and  GTE models}
    \label{Results_2}
\end{figure}

 Some findings from Figure \ref{Results_2}  are in order.   First, increasing the dimension of embedding vectors generally leads to an increase in classification accuracy. This is expected as the embedding vectors with larger dimension tend to encode more useful information for classification. 
 Second,  the performance between 
 KSVM classifiers with Gaussian and linear kernels are nearly identical across all settings. Moreover, 
their accuracy rate increases substantially as $r$ changes from $1$ to $2$, after which it rapidly stabilizes to reach their best performance.
 This indicates that the relationship between 
 embedding vectors and the labels is linear and can be fully explained by the first few leading principal components. 
 Third, for large $r$, 
 the logistic classifier suffers from overfitting, leading to unsatisfactory performance in classification. This is 
 consistent with the behavior of OLS  observed from text regression analysis in \cref{sec_real_data_regress}.
Overall, our experiments suggest that, to avoid tuning the optimal $r$,   KSVM  with the full embedding---whether using a linear or Gaussian kernel---is the preferred choice for the text classification task, although using $r=2$ yields only slightly worse accuracy while remaining interpretable.



{\small 
\setlength{\bibsep}{0.85pt}{
    \bibliographystyle{plainnat} 
    \bibliography{ref}  
}
}



\appendix

\newpage

\section*{Appendix}

The main proofs are collected in \cref{app_sec_proofs} and presented on a section-by-section basis. Auxiliary lemmas are stated in \cref{app_sec_auxi}.

\section{Main Proofs}\label{app_sec_proofs}


\noindent
{\bf Notation.}
Throughout this appendix,  
given i.i.d. copies $U_1,\ldots,U_n$ of any random element $U$, for any measurable function $f$, we denote the empirical mean by 
\begin{align}\label{emp_norm_pred_input}
 \EE_n[f(U)]:
 =\frac{1}{n}\sum_{i=1}^nf(U_i).  
\end{align}
We use $\tr(\cdot)$ to denote the trace of any trace operator.  
In addition to denote the operator norm for real-valued matrices, we use $\|\cdot\|_{\op}$ to denote the operator norm for operators from $\cH_K$ to $\cH_K$.   
For any Hilbert–Schmidt  operator $A: \cH_K\to\cH_K$, we use $\|A\|_\hs$ to denote the  Hilbert–Schmidt norm, given by $\|A\|_\hs^2 = \tr(A^\T A)$. The notation $I$ is used as the identity operator, whose domain and range 
may vary depending on the context. 
Unless otherwise
specified,  we will use $\EE$ to either denote the expectation conditioning on the observations $\cD := \{(Y_i,X_i)\}_{i=1}^n$ 
and $\wh g$, 
or the expectation taking over all involved random quantities when the context is clear. 
Recall that the $\ell^2$ space is
\[
\ell^2(\mathbb{N}):= \l\{ x=(x_1,x_2,...)^\T \in \RR^\i:~ \sum_{i=1}^\i x_i^2 <\i \r\}  ,
\]
equipped with the inner product  $\langle u,v\rangle_2= \sum_{i=1}^\i u_iv_i$. 
For any two elements  $a,b\in \ell^2(\NN)$, we write $ab^\T$  as  an operator from $\ell^2(\NN)$ to $\ell^2(\NN)$,  defined as
$$
 ( ab^\T) v:  =    a  \langle b, v\rangle_2, \qquad \text{for any $v\in \ell^2(\NN)$}. 
$$

Throughout, we fix the generic predictor $\wh g: \cX \to \cZ$  that is constructed independently of $\cD$.  All our statements, unless otherwise specified, should be understood to hold almost surely, conditional on $\wh g$. 
For any $f\in \cL^2(\rho)$, define the function $h_f := h_{f\circ \wh  g}: \cZ \times \cX \to \RR$ as 
\begin{equation}\label{def_h}
    h_f\left(z, x\right)=\bigl (f^*(z)-(f\circ \wh g)(x)\bigr )^2- \bigl (f^*(z)-(\fh \circ \wh g)(x) \bigr )^2.
\end{equation}
Further define $\ell_f:= \ell_{f\circ \wh g}: \RR\times \cX \to \RR$ as
\begin{equation}\label{def_ell}
    \ell_f\left(y, x\right)=\bigl (y-(f\circ \wh g)(x)\bigr )^2-\bigl (y-(\fh \circ \wh g)(x)\bigr )^2.
\end{equation}
From \eqref{main_eq1}, 
we know that for any $f\in \cL^2(\rho)$, 
\begin{align}\label{stru_assu}
  \EE\left[\bigl(Y- (f\circ \wh g)(X)\bigr)^2-\bigl(Y- f^*(Z)\bigr)^2\right]=\EE\left[(f\circ \wh g)(X) -f^*(Z)\right]^2
\end{align}
so that  
\begin{align}\label{link_population}\nonumber
&\EE[\ell_f\left(Y, X\right)]\\\nonumber
&=\EE\left[(Y- (f\circ \wh g)(X))^2-(Y- (\fh \circ \wh g)(X))^2\right]\\\nonumber
&= \EE\left[(Y- (f\circ \wh g)(X))^2-(Y- f^*(Z))^2\right] +\EE\left[(Y- f^*(Z))^2-(Y- (\fh \circ \wh g)(X))^2\right]\\\nonumber
&=\EE\left[\bigl( f^*(Z)-(f\circ \wh g)(X) \bigr)^2-\bigl(f^*(Z)-(\fh \circ \wh g)(X) \bigr)^2\right]\\
&= \EE  [h_f(Z,X)].
\end{align}
Additionally, 
the empirical counterpart of the estimation error can be decomposed as
\begin{align}\label{link_empirical}\nonumber
\EE_n[\ell_f\left(Y, X\right)] & = \EE_n\left[(Y- (f\circ \wh g)(X))^2-(Y- (\fh \circ \wh g)(X))^2\right]\nonumber\\
&= \EE_n[h_f(Z,X)] + \frac{2}{n}\sum_{i=1}^n\epsilon_i \Bigl((\fh \circ \wh g)(X_i)-(f\circ \wh g)(X_i)\Bigr)\nonumber\\
&= \EE_n[h_f(Z,X)] + 2 \EE_n[\epsilon  (\fh \circ \wh g)(X)-\epsilon  (f\circ \wh g)(X)],
\end{align} 
where $\epsilon_i$'s are the copies of random noise term defined in \eqref{model}.


\subsection{Proof of \texorpdfstring{\eqref{bd_irre_error}}{\texttwoinferior} and the excess risk decomposition in \texorpdfstring{\eqref{eq_risk_decomp}}{\texttwoinferior}}\label{app_sec_proof_risk_decomp}

\begin{proof}
    Pick any deterministic $f:\cZ\to \RR$. By adding and subtracting terms, we have 
\begin{align*}
\cE \b (\wh f \circ \wh g \b )  =\EE \bigl [  Y - (\wh f\circ \wh g)(X) \bigr ]^2 -\EE \left [Y - (f\circ \wh g)(X)\right ]^2  +\EE \left [ Y - (f\circ \wh g)(X)\right ]^2  -\sigma^2.
\end{align*}
Note that
\begin{align}\label{main_eq1} 
\EE \left [  Y - (f\circ \wh g)(X) \right ]^2 - \sigma^2 
&=\EE \left [ \left  (Y-f^*(Z)+f^*(Z)-(f\circ \wh g)(X) \right )^2 \right ] - \sigma^2\nonumber
\\
&= \EE \left [f^*(Z)-(f\circ \wh g)(X) \right ]^2  +2\EE \left [\epsilon \left (f^*(Z)-(f\circ \wh g)(X) \right ) \right ]\nonumber
\\
&=\EE \left [f^*(Z)- (f\circ \wh g)(X) \right ]^2 ,
\end{align}
where the penultimate step uses $Y-f^*(Z)=\epsilon$ and $\epsilon$ is independent of $Z$, $X$ and $\wh g$.
\end{proof}

\begin{proof}
  We have that for any $\theta \ge 1$, the elementary inequality $(a+b)^2 \le x a^2 + b^2/x$ for $x>0$, so that
\begin{align*}
     \EE \left [  f^*(Z)- (\fh \circ \wh g)(X)   \right ]^2 & ~  \le  ~  (1+\theta) ~ \EE\left[ \fh(Z) - (\fh\circ \wh g)(X)   \right]^2   +   {1+\theta \over \theta}\|f_\cH - f^*\|_\rho^2\\
    & ~ =~  (1+\theta) ~ \EE \left \langle \fh, K_Z-K_{\wh g(X)} \right\rangle_K^2  +   {1+\theta \over \theta}\|f_\cH - f^*\|_\rho^2 \\
    &~ \le~  (1+\theta) \|\fh \|_K^2 ~ \EE\Delta_{\wh g} +  {1+\theta \over \theta}\|f_\cH - f^*\|_\rho^2.  
\end{align*}
The second step uses the reproducing property.
\end{proof}

\subsection{Proof of \texorpdfstring{\cref{thm_risk}}{\texttwoinferior}: bounding the excess risk} \label{app_sec_estimation_error_bound}

Before proving \cref{thm_risk}, we first give a useful lemma that quantifies the relative loss function in \eqref{def_ell} along the line segment between any  $f\in \cH_K$ and $\fh$. 
\begin{lemma}\label{used_lem_sec_A}
For any $f\in \cH_K$ and $ \alpha\in [0,1]$, 
define $\wt f= \alpha f+(1- \alpha) \fh$. Then we have
\begin{align*}
   \EE_n[\ell_{\wt f}(Y, X)]+ \lambda\| \wt f\|^2_K-\lambda \|f_\cH\|_K^2  
    ~ \le ~ \alpha\left( \EE_n[\ell_f(Y, X)] +\lambda \|f\|^2_K-\lambda \|f_\cH\|_K^2 
  \right).
\end{align*}
\end{lemma}
\begin{proof}
    The proof of \cref{used_lem_sec_A} is provided after  the main proof of \cref{thm_risk}.
\end{proof}

Recall  $\Delta_{\wh g}$ from \eqref{def_Delta} and $\delta_n$
is defined as $R(\delta_n) = \delta_n$. Define 
\begin{align}\label{def_delta_n_hatg}
\delta_{n,\wh g}:=(C_0^2 \vee 1) ( \gamma_\epsilon^2 \vee 1)  \l(  \delta_n  \log(1/\eta)+ \EE\Delta_{\wh g} +  \frac{(\kappa^2\vee1)}{n}  \log\frac{(\kappa^2\vee1)}{\eta} \r)  , 
\end{align}
with $\gamma_\epsilon$  introduced in \cref{ass_reg_error} and  $C_0>0$ being some sufficiently large constant.   Further define the scalars
\begin{align}\label{def_V_K}
 V_K ~ :=&~  \max\{ \kappa^2 \|\fh\|^2_K,  1\}~  ( \|\fh\|^2_K \vee 1),\\\label{def_Lambda} 
 \Lambda ~~:=& ~ \max\left\{\kappa^2\|\fh \|_K^2+\|f^*\|_\i^2,1\right\}  V_K  ~   \delta_{n,\wh g}. 
\end{align}
Note that both $\delta_{n,\wh g}$  and $\Lambda$ can be simplified to 
    \begin{align}\label{lambda_condi}
        C \l( \delta_n  \log(1/\eta)  +  \EE\Delta_{\wh g} + \frac{\log(1/\eta)}{n}   \r),
    \end{align}
for some $C = C(\kappa, \gamma_\epsilon, \|f^*\|_\i, \|\fh\|_K)$, 
    as stated in \cref{thm_risk}. 
    Below we prove  \cref{thm_risk}  in terms of $\delta_{n,\wh g}$  and $\Lambda$.

\begin{proof}[Proof of \cref{thm_risk}]
By the optimality of $\wh f$ and the feasibility of $f_{\cH}$ in \eqref{def_f_hat}, we have
\begin{align}\label{eq_primary_target}
  \EE_n  [\ell_{\wh f}(Y, X)  ]+\lambda \|\wh f\|^2_K-\lambda \|\fh \|^2_K\le 0.
\end{align}
In the case when $\|\fh\| _ K =0 $,  choosing  $\lambda\to \i$ together with \eqref{eq_primary_target} leads to $\|\wh f\|_K=0$. 
By the reproducing kernel property in $\cH_K$ and \cref{ass_bd_K}, we have that for any 
$z\in \cZ$, 
\[
\fh (z) = \langle \fh, K_z \rangle_ K \le \|K_z\|_K \|\fh\|_K= \sqrt{K(z,z)} \|\fh\|_K \le  \kappa \|\fh\|_K=0
\]
so that $\fh \equiv 0$. The same argument yields  $\wh f \equiv 0$.
Consequently, we have
$
\EE[\ell_{\wh f}(Y, X)] = 0
$
when $\|\fh\|_K  = 0$. In the following we focus on the case $\|\fh\| _ K >0$ with the choice of $\lambda$ as 
\begin{equation}\label{rate_lbd_proof}
      \lambda\|\fh\|^2_K ~ >~   C  \Lambda  + 4 \|\fh-f^*\|_\rho^2. 
\end{equation}

Recall $h_f$ from \eqref{def_h} and $\ell_f$ from \eqref{def_ell}. Define the local RKHS-ball as 
\begin{equation}\label{def_Fb}
    \cF_b:=\left\{f\in\cH_K: 
 \|f-f_{\cH}\|_K\le 3\|f_{\cH}\|_K\right\}
\end{equation} 
and recall from \eqref{def_Fb_subset} that
\begin{equation}\label{def_Fb_subset}
    \uline{\cF_b} = \left\{f\in \cF_b: \EE[\ell_f(Y,X)] \le  6\lambda \|\fh\|_K^2 \right\}.
\end{equation}
We first prove a uniform result over $\uline{\cF_b} $. 
Consider  the events 
\begin{align*}
    &\cE(h) := 
   \bigcap_{f\in \cF_b} \Bigl\{
        \EE[h_f(Z,X)] \le  2\EE_n[h_f(Z,X)] + C\Lambda + 4 \|\fh-f^*\|_\rho^2 
    \Bigr\},\\
    &\cE(\epsilon) := \bigcap_{f\in \uline{\cF_b} }\left\{ \B|4\EE_n[\epsilon  (\fh \circ \wh g)(X)-\epsilon  (f\circ \wh g)(X)] \B | \le      \lambda \|\fh\|_K^2
    \right\}.
\end{align*}
By  \cref{cor_bd_h_f} and \cref{bd_cross_term},
 the event $\cE(h) \cap \cE(\epsilon)$ holds with probability at least $1-\eta$. In the rest of the proof we work on the event  $\cE(h) \cap \cE(\epsilon)$.
 For any $f\in \uline{\cF_b}$, we have that 
\begin{align}\label{lb_En_ell}\nonumber
     \EE[\ell_f(Y,X)] 
    & = \EE[h_f(Z,X)] &&\text{by \eqref{link_population}}\\\nonumber
        &\le 2\EE_n[h_f(Z,X)] + C \Lambda + 4 \|\fh-f^*\|_\rho^2 &&\text{by }\cE(h)\\\nonumber
                    &=  2\EE_n[\ell_f(Y,X)]  - 4\EE_n[\epsilon  (\fh \circ \wh g)(X)-\epsilon  (f\circ \wh g)(X)] \\\nonumber
                    &\qquad + C\Lambda + 4 \|\fh-f^*\|_\rho^2 &&\text{by \eqref{link_empirical}}\\ 
        &\le 2\EE_n[\ell_f(Y,X)]  + 2\lambda \|\fh\|_K^2 &&\text{by $\cE(\epsilon) \& \eqref{rate_lbd_proof}$}.
\end{align}
Note that the conclusion follows trivially if $\wh f\in\uline{\cF_b}$. It remains to prove by contradiction when $\wh f\notin \uline{\cF_b}$. The latter consists of two cases:
\begin{enumerate}[noitemsep]
    \item[(1)] $\wh f \in \cF_b \setminus \uline{\cF_b}$; 
    \item[(2)] $\wh f \notin \cF_b$.
\end{enumerate}
{\bf Case (1).} For any $ \alpha\in [0,1]$, let
\[
    f_\alpha := \alpha \wh f + (1-\alpha) \fh.
\]
Since $\EE[\ell_{\fh}(Y,X)] = 0$ and $\EE[\ell_{\wh f}(Y,X)] > 6\lambda \|\fh\|_K^2 $ (as $\wh f\in \cF_b\setminus \uline{\cF_b}$), the fact that
\[
    \EE[\ell_{f_\alpha}(Y,X)] = \EE[(Y-(f_\alpha \circ \wh g)(X))^2-(Y-(\fh\circ \wh g)(X))^2]
\]
is continuous in $\alpha$ implies the existence of $\wh \alpha \in (0,1)$ such that 
\begin{equation}\label{eq_exp_f_hat_alpha}
    \EE[\ell_{f_{\wh \alpha}}(Y,X)] = 6\lambda \|\fh\|_K^2.
\end{equation}
By the convexity of $\cF_b$,\footnote{The RKHS $\cH_K$ is the completion of the linear span of the functions $K(v,\cdot)$ for $v\in \cZ$, hence must be convex.} we have  $f_{\wh \alpha} \in \cF_b$ hence $f_{\wh \alpha}  \in \uline{\cF_b}$. Applying \cref{used_lem_sec_A} with $\alpha = \wh \alpha$, $\wt f = f_{\wh \alpha}$ and $f= \wh f$ yields
\[
 \EE_n[\ell_{f_{\wh \alpha}}(Y, X)]+ \lambda\|f_{\wh \alpha}\|^2_K-\lambda \|f_\cH\|_K^2  
    ~ \le~  \wh \alpha\left \{  \EE_n[\ell_{\wh f}(Y, X)] +\lambda \|\wh f\|^2_K-\lambda \|f_\cH\|_K^2 
  \right\}.
\]
Since $\wh \alpha>0 $,  if we could show that
\begin{align} \label{contra_target}
    \EE_n[\ell_{f_{\wh \alpha}}(Y, X)] +\lambda \|f_{\wh \alpha} \|^2_K-\lambda \|\fh \|^2_K>0, 
\end{align}
then \eqref{contra_target} also holds  for $\wh f$ in lieu of $f_{\wh \alpha}$, which contradicts \eqref{eq_primary_target}. To prove \eqref{contra_target}, recalling that $f_{\wh \alpha} \in \uline{\cF_b}$,
invoking the uniform bound in \eqref{lb_En_ell} with $f=f_{\wh \alpha}$ gives
\begin{align*}
  -2\EE_n[\ell_{f_{\wh \alpha}}(Y,X)] +2\lambda \|\fh\|^2_K-2\lambda \|f_{\wh \alpha} \|^2_K  
 &\le  -   
 \EE[\ell_{f_{\wh \alpha}}(Y,X)] +4\lambda \|\fh\|^2_K-2\lambda \|f_{\wh \alpha}\|^2_K \\
  &=     - 2\lambda \|\fh\|_K^2    -2\lambda \|f_{\wh \alpha}\|^2_K  &&\text{by \eqref{eq_exp_f_hat_alpha}}\\
  &< 0. 
\end{align*}

\noindent{\bf Case (2).} Note that $\wh f \notin \cF_b$ implies $\|\wh f - \fh\|_K>  3\|\fh\|_K$. 
By letting
\[
\wt  \alpha = \frac{3\|\fh\|_K }{\|\wh f-\fh\|_K} \in (0, 1), 
\]
the function
\[
 f_{\wt \alpha}: =\wt  \alpha \wh f+(1- \wt \alpha)\fh  = \fh+\frac{3\|\fh\|_K }{ \|\wh f-\fh\|_K}(\wh f-\fh)
\]
also belongs to $\cH_K$, due to the convexity of $\cH_K$. Moreover, we have $\| f_{\wt \alpha} - \fh\|_K= 3\|\fh\|_K$ so that $ f_{\wt \alpha} \in \cF_b$. Similar as case (1), it suffices to prove 
\begin{align}\label{contra_target_prime}
    \EE_n[\ell_{ f_{\wt \alpha}}(Y, X)] +\lambda \| f_{\wt \alpha}\|^2_K-\lambda \|\fh \|^2_K>0,
\end{align}
which, by applying \cref{used_lem_sec_A} with $\alpha =\wt \alpha$, $\wt f=  f_{\wt \alpha}$ and $f = \wh f$,  would imply that \eqref{contra_target_prime} also holds for $\wh f$ in lieu of $ f_{\wt \alpha}$, thereby establishing contradiction with \eqref{eq_primary_target}. 

If $f_{\wt \alpha} \notin \uline{\cF_b}$, then repeating the same arguments of proving case (1) proves \eqref{contra_target_prime}. 

It remains to prove \eqref{contra_target_prime} for $f_{\wt \alpha} \in \uline{\cF_b}$. Observe that
\begin{align}\label{eq_E_bd_wt_f} \nonumber
       -\EE[\ell_{f_{\wt \alpha}}(Y,X)]  &= -\EE[h_{f_{\wt \alpha}}(Z,X)]&&\text{by \eqref{link_population}}\\\nonumber
&= \EE[(\fh \circ \wh g)(X)-f^*(Z)]^2 -\EE[(f_{\wt \alpha}\circ \wh g)(X)-f^*(Z)]^2 \\ 
&\le 2 \|\fh\|_K^2 ~ \EE\Delta_{\wh g} +{ 2\|\fh-f^*\|_\rho^2} &&\text{by \eqref{bd_irre_error} with $\theta =1$}.  
\end{align}
Meanwhile, since $f_{\wt \alpha} \in \uline{\cF_b}$, 
invoking  \eqref{lb_En_ell} with $f=f_{\wt \alpha}$ gives
\begin{align*}
  &-2\EE_n[\ell_{f_{\wt \alpha}}(Y,X)] +2\lambda \|\fh\|^2_K-2\lambda \|f_{\wt \alpha}\|^2_K \\
 &\le  -   
 \EE[\ell_{f_{\wt \alpha}}(Y,X)] +4\lambda \|\fh\|^2_K-2\lambda \|f_{\wt \alpha}\|^2_K \\
  &\le  2 \|\fh\|_K^2~  \EE\Delta_{\wh g} + 2\|\fh-f^*\|_\rho^2 + 4\lambda \|\fh\|^2_K -2\lambda \|f_{\wt \alpha}\|^2_K &&\text{by \eqref{eq_E_bd_wt_f}}\\
&\le   5
\lambda \|\fh\|^2_K-2\lambda \|f_{\wt \alpha}\|^2_K &&\text{by \eqref{rate_lbd_proof}}\\
&<  0,
\end{align*}
where the last step follows from $\|f_{\wt \alpha} \|_K\ge \|f_{\wt \alpha} -\fh\|_K-\|\fh\|_K=2\|\fh\|_K>0$.
This proves \eqref{contra_target_prime}, thereby completing the proof.
\end{proof} 

 \begin{proof}[Proof of \cref{used_lem_sec_A}]
For any $f_1, f_2 \in \cL^2(\rho)$ and $\alpha \in [0,1]$, {by  the convexity of quadratic loss function}, we have
\begin{equation*} 
\begin{aligned}
&\bigl(Y-((\alpha f_1+(1-\alpha)f_2) \circ \wh g)
(X)\bigr)^2-\bigl(Y- (f_{\cH}\circ \wh g)(X)\bigr)^2 \\
&\le  \alpha \bigl(Y-(f_1\circ \wh g)(X)\bigr)^2 +(1-\alpha) \bigl(Y-(f_2\circ \wh g)(X)\bigr)^2 -\bigl(Y- (f_\cH \circ \wh g)(X)\bigr)^2\\
&=\alpha\left[ \bigl(Y-(f_1\circ \wh g)(X)\bigr)^2  -\bigl(Y- (f_\cH \circ \wh g)(X)\bigr)^2 \right] \\
&\qquad + (1-\alpha)\left[ \bigl(Y-(f_2\circ \wh g)(X)\bigr)^2  -\bigl(Y- (f_\cH \circ \wh g)(X)\bigr)^2 \right]\\
&= \alpha \ell_{ f_1}(Y,X)+( 1-\alpha ) \ell_{ f_2}(Y,X). 
\end{aligned}    
\end{equation*}
By applying the above inequality 
with $f_1=  f$ and $ f_2=\fh$ and taking $\EE_n$ on both sides, we have 
\begin{align} \label{boundary_verify_1}
\EE_n[\ell_{\wt  f}(Y, X)]
\le\alpha ~ \EE_n[\ell_{f}(Y, X)].
\end{align}
On the other hand, note that
\begin{equation} \label{boundary_verify_2}
\begin{aligned}
\|\wt f\|^2_K & =\bigl\| \alpha   f+(1- \alpha)\fh\bigr\|_K^2
 \\
&\le\bigl (\alpha\|  f\|_K+(1- \alpha)\|\fh\|_K\bigr)^2&&\text{by the triangle inequality,}
 \\
&\le \alpha \| f\|_K^2 +(1- \alpha)\|\fh\|^2_K &&\text{by the convexity of the quadratic function}.
\end{aligned}    
\end{equation}
Combining \eqref{boundary_verify_1} and \eqref{boundary_verify_2} 
completes the proof. 
\end{proof}

\subsubsection{Uniform upper bounds of \texorpdfstring{$\EE[h_f(Z,X)]$}{\texttwoinferior} over \texorpdfstring{$f\in \cF_b$}{\texttwoinferior} }\label{app_sec_Rade}

In the section, we apply \cref{supp_lem_3}  to provide uniform control of  $\EE[h_f(Z,X)]$ over $f\in \cF_b$ by using the technique of local Rademacher complexity, which was developed by \cite{bartlett2005local}.

For a generic function space $\mathcal{F}$ from $\RR^r$ to $\RR$ and  $n$ fixed points $\{v_1,..., v_n\}$,
the empirical Rademacher complexity is defined as
\begin{align}\label{def_em_Rademacher}
\emlocalcomp(\mathcal{F}):=\EE_{\varepsilon}
\Big[\sup _{f \in \mathcal{F}} \frac{1}{n}\sum_{i=1}^n \varepsilon_i f\left({v}_i\right)\Big], 
\end{align}
where the expectation $\EE_{\varepsilon}$ is taken over $n$ i.i.d. Rademacher variables $\varepsilon_1, \ldots, \varepsilon_n$ with $\PP(\varepsilon_1 = 1) =\PP(\varepsilon_1 = -1) = 1/2$.
Suppose that $\{v_1,\ldots, v_n\}$ are independently generated according to some distribution $\varrho$. The Rademacher complexity is defined as $$
\localcomp(\mathcal{F}):=\EE_\varrho\left[\emlocalcomp(\mathcal{F})\right],
$$
where $\EE_\varrho$ is taken with respect to $\{v_1,..., v_n\}$. Let $T(\cdot):
\mathcal{F} \rightarrow \RR^{+}$ be  some functional. Our results will be based on  the {\em local Rademacher complexity}, which typically takes the form 
\begin{align}\label{def_Rademacher}
\localcomp(\{f\in \mathcal{F}: T(f)\le \delta\}),\quad \text{for all} \ \delta \ge 0.
\end{align}

The sub-root function is an essential ingredient for establishing fast rates in statistical learning problems.  Precisely, when deriving the bound on the error of a learning algorithm, the key step is to select an appropriate sub-root function, and its corresponding fixed point will appear in the prediction error bound.  In Appendix \ref{app_sec_subroot}  we review the definition and some useful properties of sub-root functions.

For any $f\in  \cL^2(\rho)$, 
define the function 
$\xi_{f}:= \xi_{f\circ \wh g}:  \cX \to \RR$ as
\begin{align}\label{def_g_f}
\xi_{f}(x) = (f \circ \wh g)(x)-(\fh \circ \wh g)(x)
\end{align}
and the function class $\Xi_b$ as
\begin{align}\label{def_class_g_b}
\Xi_b : = \b \{ \xi_{f}: ~  f \in \cF_b\b\}. 
\end{align}
Consider the function
\begin{align}\label{def_psi_g_b_popu}
  \psi_x (\delta):  = \localcomp \left(\big\{\xi_f \in \Xi_b: ~  \EE[\xi^2_f(X)]\le \delta \big\} \right),\qquad  \text{for all}\ \delta \ge 0.
\end{align}
It is verified in Lemma \ref{subroot_1} that $\psi_x (\delta)$ is a sub-root function. Let $\delta_x$ be its fixed point, that is, 
$$
     \psi_x(\delta_x) = \delta_x.
$$ 

The following lemma  bounds from above $\EE[h_f(Z,X)]$ by its empirical counterpart 
$\EE_n[h_f(Z,X)]$, the fixed point $\delta_x$, the kernel-related latent error $\EE \Delta_{\wh g}$ and the approximation error $\|\fh - f^*\|_\rho^2$, uniformly over $f\in \cF_b$.

\begin{lemma}\label{lem_bd_h_f}
{Grant Assumptions  \ref{ass_f_H} and \ref{ass_bd_K}.}
 Fix any $\eta\in(0,1)$. With probability at least $1-\eta$, for any $f\in \cF_b$, one has 
\begin{align*}
&\EE[h_f(Z,X)] ~ \le  ~ 2\EE_n[h_f(Z,X)]+
c_1\max\left\{\kappa^2\|\fh \|_K^2+\|f^*\|_\i^2,1\right\} ~  \delta_x   \\  
& \hspace{1 in}+ c_2 \kappa \|\fh\|_K\left(\kappa \|\fh\|_K + \|f^*\|_\i\right)  \frac{\log (1/\eta)}{n}+ 
  4\|\fh \|_K^2  \EE\Delta_{\wh g}  + 4\|\fh - f^*\|_\rho^2.
\end{align*} 
Here $c_1$ and $c_2$ are some absolute positive constants.
\end{lemma}
\begin{proof}
Recall $h_f$ from \eqref{def_h} with its corresponding function class is 
\[
    \cF_h : = \Big\{h_f: f\in \cF_b\Big\}.
\]
Define the scalar
\[
    V := 3\kappa \|\fh\|_K\left(5\kappa \|\fh\|_K + 2\|f^*\|_\i\right).
\]
Further define the functional $T:\cF_h \to \RR^+$ as  
\begin{align}\label{functional_Tf}
T(h_f):= V  \left(\EE\left[h_f(Z, X) \right]+   4\|\fh\|_K^2  \EE\Delta_{\wh g} + 4 \|\fh - f^*\|_\rho^2   \right)
\end{align} 
as well as the local Rademacher complexity of $\cF_h$ as  
\begin{align*}
    \localcomp\Big(\big\{h_f \in \cF_h: T(h_f)\le \delta\big\}\Big). 
\end{align*}
We aim to apply the first statement in Lemma \ref{supp_lem_3} with $\cF = \cF_h$, $T = T$,  $L =   4\|\fh\|_K^2  \EE\Delta_{\wh g} + 4 \|\fh - f^*\|_\rho^2  $ and  
\[
   B = -a = b =  V.
\]
The remaining proof is divided into four steps.\\

\noindent{\it Step 1: Verification of the boundedness of $\cF_h$.}    By the reproducing property,
for every $z\in \cZ$ and $f\in \cH_K$, 
we have
\[
|f(z)| = |\langle f, K_z\rangle_K |\le \|K_z\|_K \| f\|_K \le \kappa  \| f\|_K,
\]
where $K_z=K(\cdot,z)\in \cH_K$ and 
the last inequality holds due to $\|K_z\|_K^2= K(z,z)\le \kappa^2$ from \cref{ass_bd_K}. Moreover, for any  $f\in \cF_b$, 
\[
\|f\|_K\le \|\fh\|_K+\|f-\fh\|_K\le 4\|\fh\|_K. 
\]
Then, the function space $\cF_b$ can be  uniformly bounded as
\begin{align}\label{bd_F_b}
    \|f\|_\infty\le  \kappa\|f\|_K\le 4\kappa\|\fh\|_K, \quad\text{for any}\ f\in \cF_b. 
\end{align}
For every $f\in \cF_b$, we also have
\begin{align}\label{bd_F_b_diff}
\|f-\fh \|_\infty \le \kappa\|f-\fh\|_K\le 3\kappa \|\fh\|_K. 
\end{align}
Combining all these bounds, 
for every $f \in \cF_b$ and $(z,x)\in \cZ \times \cX$, we have 
\begin{align}\label{la_eq22}\nonumber
|h_f(z, x)|
&=\left|(f^*(z)-(f\circ \wh g)(x))^2-(f^*(z)-(\fh\circ \wh g)(x))^2\right|\\\nonumber
 &\le \Bigl|(\fh\circ \wh g)(x)- (f\circ \wh g)(x)\Bigr| \Bigl(|(f\circ \wh g)(x)|+|(\fh\circ \wh g)(x)|+ 2|f^*(z)|\Bigr)
\\
 &\le 3\kappa\|\fh\|_K(5\kappa \|\fh\|_K+2\|f^*\|_\i).
 \end{align}  
This shows the function class $\cF_h$ is uniformly bounded within $[-V, V]$.\\

\noindent{\it Step 2: Bounding the variance of $h_f(Z,X)$.} To apply Lemma \ref{supp_lem_3}, we need
to bound from above the variance of $h_f(Z, X)$  by $T(h_f)$. To this end,  note that 
\begin{align*}
\Var\left(h_f(Z,X)\right)
&= \Var\left((f^*(Z)-(f\circ \wh g)(X))^2-(f^*(Z)-(\fh \circ \wh g)(X))^2\right)\\
&\le  \EE\left[(f^*(Z)-(f\circ \wh g)(X))^2-(f^*(Z)-(\fh \circ \wh g)(X) )^2\right]^2\\
&\le V ~  \EE \left|(f^*(Z)-(f\circ \wh g)(X))^2-(f^*(Z)-(\fh \circ \wh g)(X))^2\right|  && \text{by \eqref{la_eq22}}.
\end{align*}
Since
\begin{align*} 
&\EE\left|(f^*(Z)-(f\circ \wh g)(X))^2-(f^*(Z)-(\fh \circ \wh g)(X))^2\right|
\\
&\le \EE[(f^*(Z)-(f\circ \wh g)(X))^2-(f^*(Z)-(\fh \circ \wh g)(X))^2 ]+2\EE[f^*(Z)-(\fh \circ \wh g)(X)]^2 \\
&=\EE\left[ h_f(Z, X)\right] + 2\EE[f^*(Z)-(\fh \circ \wh g)(X)]^2 &&\text{by \eqref{def_h}}\\
&\le \EE\left[ h_f(Z, X)\right]+4 \|\fh \|_K^2  \EE\Delta_{\wh g}+ 4\EE[\fh(Z)-f^*(Z)]^2
\end{align*}
where the last step uses \eqref{bd_irre_error} with $\theta =1$, 
 we have
\begin{align*}
\Var\left(h_f(Z,X)\right)  \le  V \left( \EE\left[ h_f(Z,  X)\right]+ 4 \|\fh \|_K^2  \EE\Delta_{\wh g}  + 4 \|\fh - f^*\|_\rho^2 \right) \overset{\eqref{functional_Tf}}{=}T(h_f).
\end{align*}
This further implies that the variance condition in Lemma \ref{supp_lem_3} holds with $B=  V$ and $L=4 \|\fh \|_K^2  \EE\Delta_{\wh g} +4  \|\fh - f^*\|_\rho^2 $ for the functional $T$ in \eqref{functional_Tf}.\\

\noindent{\it Step 3: Choosing a suitable sub-root function.} 
To apply Lemma \ref{supp_lem_3}, we need to find a suitable sub-root function $\psi(\delta)$ with its fixed point $\delta_\star$ such that for all $\delta \ge \delta_\star$, 
\begin{align}\label{bd_subroot}
  \psi(\delta) \ge   V ~ \localcomp\left(\big\{h_f \in \cF_h: T(h_f)\le \delta\big\}\right). 
\end{align} 
To this end, we notice that for any $f\in \cF_b$,
\begin{align*}
{V\over 2} \EE[ \xi_f^2(X)] &={V\over 2} ~ \EE\left[  (f\circ\wh g) (X)-(\fh \circ\wh g) (X)\right]^2  \\
&\le V  \EE[(f\circ \wh g)(X)-f^*(Z)] ^2+ V  \EE[f^*(Z)- (\fh \circ \wh g)(X)]^2    \\
&=  V \EE[h_f(Z,X) ] + 2V  \EE[f^*(Z)- (\fh \circ \wh g)(X)]^2   \\
&\le V \left(\EE[h_f(Z,X) ]  + 4\|\fh \|_K^2  \EE\Delta_{\wh g} +4 \|\fh-f^*\|_\rho^2 \right)  &&\text{by \eqref{bd_irre_error} with $\theta =1$}\\
&\overset{\eqref{functional_Tf}}{=}  T(h_f). 
\end{align*}
It follows that for any $\delta \ge 0$,
\begin{align}\label{lem4_eq0}\nonumber
\localcomp\left(\big\{h_f \in \cF_h: T(h_f)\le \delta\big\}\right)&\le  \localcomp\left(\big\{ h_f\in \cF_h:  V  \EE[\xi^2_f(X)] \le 2 \delta\big\}\right)\\
&= \EE
\left[\sup _{f\in \cF_b, ~ V \EE[\xi^2_f(X)] \le 2\delta} \frac{1}{n} \sum_{i=1}^n \varepsilon_i h_f(Z_i, X_i)\right]. 
\end{align}
Observe that, for any $z \in \cZ$, $x \in \cX$ and $f_1,f_2\in \cF_b$, 
\begin{align}\label{la_eq23}
&\left|h_{f_1}(z, x)- h_{f_2}(z,x)\right| \nonumber\\
&=\left| (f^*(z)-(f_1\circ \wh g)(x))^2-(f^*(z)-(f_2\circ \wh g)(x))^2\right|\nonumber \\
&\le \left| (f_1\circ \wh g)(x) + (f_2\circ \wh g)(x) - 2f^*(z) \right|    \left|(f_1\circ \wh g)(x)-(f_2\circ \wh g)(x)\right|
\nonumber \\
&\le (8\kappa\|\fh \|_K+2\|f^*\|_\infty)  \left|(f_1\circ \wh g)(x)-(\fh \circ \wh g)(x)- (f_2\circ \wh g)(x) + (\fh \circ \wh g)(x) \right|\nonumber\\
&= (8\kappa\|\fh \|_K+2\|f^*\|_\infty)  \left|\xi_{f_1} (x)- \xi_{f_2} (x) \right|.
\end{align}
This implies that $h_f(Z_i, X_i)$ is $(8\kappa\|\fh \|_K+2\|f^*\|_\infty)$-Lipschitz in $\xi_f(X_i)$ for all $ i\in [n]$. 
We next apply Ledoux–Talagrand contraction inequality in Lemma \ref{supp_lem_5} with $d = n$, 
\begin{align*}
    &\theta_j = \xi_f(X_j )=(f\circ \wh g)(X_j)- (\fh \circ \wh g)(X_j),\\
    &\psi_j(\theta_j) = h_f(Z_j, X_j) = \left(f^*(Z_j) - (\fh \circ \wh g)(X_j) -\theta_j \right)^2-\left(f^*(Z_j)-(\fh \circ \wh g)(X_j)\right)^2
\end{align*}   
 for all $j\in [n]$ as well as $\cT \subset \RR^n $ chosen as

$$
    \left\{\left(\xi_f (X_1) , \ldots,  \xi_f (X_n)\right)^\T: f\in \cF_b, ~ V\EE[\xi_f^2(X)]\le 2\delta \right\}.
$$
Note that 
\eqref{la_eq23} ensures that $\psi_j(\theta_j)$ is $(8\kappa\|\fh \|_K+2\|f^*\|_\infty)$-Lipschitiz.  
Invoking \cref{supp_lem_5}  yields
\begin{align}\label{lem4_eq1}
&\EE
\Big[\sup _{f\in \cF_b,  V \EE[\xi^2_f(X)]\le 2\delta} \frac{1}{n} \sum_{i=1}^n \varepsilon_i h_f(Z_i, X_i)\Big]\nonumber\\ 
&\le (8\kappa\|\fh \|_K+2\|f^*\|_\infty) ~  \EE
\l[\sup _{f\in \cF_b,   V \EE[\xi^2_f(X)]\le 2\delta} \frac{1}{n} \sum_{i=1}^n \varepsilon_i  \xi_f(X_i) \r]
\end{align}
so that, together with \eqref{lem4_eq0}, 
\begin{align}\label{combined_bd}\nonumber
&\localcomp\left(\big\{h_f \in \cF_h: T(h_f)\le \delta\big\}\right)\\
&\le (8\kappa\|\fh \|_K+2\|f^*\|_\infty) ~ \localcomp\left(\left\{ \xi_f :  f\in \cF_b, V \EE[\xi^2_f(X)]\le 2\delta\right\}\right).
\end{align}
To find a sub-root function that meets the requirement \eqref{bd_subroot}, recall that the function $\psi_x(\cdot)$ in \eqref{def_psi_g_b_popu} is sub-root. It follows from \eqref{combined_bd}  that  
\begin{align}\label{def_subroot_td}
V ~ \localcomp\left(\big\{h_f \in \cF_h: T(h_f)\le \delta\big)\right\} \le (8\kappa\|\fh \|_K+2\|f^*\|_\infty) V ~ \psi_x\Big(\frac{2 \delta}{V}\Big) = : \psi(\delta)
\end{align}
for any $\delta \ge 0$.
According to Lemma \ref{lem_subroot_comp},  the function $\psi(\delta)$  is also sub-root in $\delta$.\\

\noindent{\it Step 4: Bounding the fixed point of the sub-root function in \eqref{def_subroot_td}.} 
    By using \cref{lem_subroot_comp} again,  
the fixed point of \eqref{def_subroot_td}  defined as 
\[
\delta_\star  = \psi(\delta_\star)  =(8\kappa\|\fh \|_K+2\|f^*\|_\infty) V ~  \psi_x\Big(\frac{2\delta_\star}{V}\Big)  
\]
satisfies
\[
\delta_\star \le \max\left\{(16\kappa\|\fh \|_K+4\|f^*\|_\infty)^2,1\right\} \frac{V}{2} \delta_x, 
\]
where 
$ \delta_x$ is the fixed point of $\psi_x(\cdot)$ in \eqref{def_psi_g_b_popu}. 

Finally, the proof is completed by invoking Lemma \ref{supp_lem_3} with $-a= b= B=V$, $L=4\|\fh \|_K^2  \EE\Delta_{\wh g}  + 4\|\fh - f^*\|_\rho^2$ and $\theta = 2$.
\end{proof}


\subsubsection{Bounding from above the fixed point of the local Rademacher complexity \texorpdfstring{$\psi_x(\delta)$}{\texttwoinferior} by that of the kernel complexity \texorpdfstring{$R(\delta)$}{\texttwoinferior}}\label{app_sec_lrc_kc}
This section is devoted to deriving an upper bound of $\delta_x$ which is the fixed point of $ \psi_x(\cdot)$ in \eqref{def_psi_g_b_popu} via a series of useful lemmas in  \cref{pf_lem_bd_td_delta},  followed by a corollary of \cref{lem_bd_h_f}. 
Recall  that  $\Delta_{\wh g}$ from \eqref{def_Delta}. Further recall $\delta_{n,\wh g}$ and $V_K$ from \eqref{def_delta_n_hatg} and \eqref{def_V_K}, respectively.

\begin{lemma}\label{lem_bd_td_delta}
Grant Assumptions \ref{ass_f_H}, \ref{ass_bd_K} and  \ref{ass_mercer}.
For any $\eta\in (0,1)$,  with probability at least $1-\eta$, one has
\[
    \delta_x\le C V_K~    \delta_{n, \wh g}. 
\]
\end{lemma}

\begin{proof}
Recalling $\emlocalcomp$ from \eqref{def_em_Rademacher}, we first link $ \psi_x (\delta) $ to its following data-dependent   version 
\begin{align}\label{def_psi_g_b_empi}
   \wh \psi_x (\delta)  = \emlocalcomp \left(\big\{\xi_f \in \Xi_b: ~  \EE_n[\xi^2_f(X)]\le \delta \big\}\right). 
\end{align}
By  \eqref{bd_F_b_diff}, the function class  $\Xi_b$ is uniformly bounded as $ \|\xi_f\|_\i \le  3\kappa \|f_\cH\|_K$ for any $f\in \cF_b$.
Then by   applying \cref{supp_lem_4} to $\Xi_b$ with $G =3\kappa \|f_\cH\|_K$, 
 for any $\eta\in (0,1)$, it holds with probability at least $1-\eta$ that
\begin{align} \label{first_bound_app}
     \psi_x (\delta) \le  2\sqrt{2} ~ \wh \psi_x ( \delta)   +   6 \kappa \|\fh \|_K  \frac{\log (2/\eta)}{n}
\end{align}
whenever
\begin{align} \label{condi_delta_1}
    \delta \ge c_1 \kappa \| \fh \|_K ~    \psi_x (\delta)+ c_2   \kappa^2  \| \fh \|_K^2  \frac{\log (2/\eta) }{n}. 
\end{align}
where $c_1$ and $c_2$ are two absolute constants.

Second, recall $\wh R_x (\delta)$ from 
\eqref{def_em_kercomp} and $\wh  R(\delta)$ from 
\eqref{def_em_kercomp_true_input}. Applying  \cref{lem_bd_em_local_rade_gb} results in
\begin{align} \label{second_bound_app}
    \wh  \psi_x (\delta) \le  \sqrt{10}  ( \|\fh\|_K  \vee 1) ~ \wh R_x (\delta),   \qquad \forall~  \delta >0,
\end{align}
which holds almost surely on $X_1, \ldots , X_n$. Subsequently, by applying \cref{cor_lem7_and_lem8},  for any $\eta\in (0,1)$, whenever 
\begin{align}\label{condi_delta_2}
    30 
  \EE\Delta_{\wh g}+  \frac{560 \kappa^2 \log (2/\eta) }{3n} \le  \delta, 
\end{align}
with probability at least $1-\eta$, one has
\begin{align} \label{third_bound_app}
\wh R_x (\delta)\le  C \l( \wh  R(\delta) +  \sqrt{\frac{\EE\Delta_{\wh g} }{n}}  +  \frac{\kappa \sqrt{\log(2/\eta)}}{n} \r). 
\end{align}

To proceed,  by applying  \cref{lem_bd_wt_R}, 
for any $\eta\in (0,1)$, 
it holds with probability at least $1-\eta$ that 
\begin{align} \label{fourth_bound_app}
   \wh  R(\delta) \le  C \l(R(\delta)  + \frac{ \kappa \sqrt{\log(2/\eta)}}{n} \r), 
\end{align}
provided that
\begin{align}\label{condi_delta_3}
   n\delta \ge 32 \kappa^2  \log\l( \frac{256 (\kappa^2 \vee 1) }{\eta} \r)
\end{align}
and 
\begin{align}\label{condi_delta_4}
    R(\delta) \le \delta. 
\end{align}

Collecting the bounds \eqref{first_bound_app}, \eqref{second_bound_app}, \eqref{third_bound_app} and \eqref{fourth_bound_app} yields  
\begin{align}\label{bd_tilde_psi}
  \psi_x (\delta) \le & 
C  ( \|\fh\|_K   \vee 1) \l(  R(\delta) + \sqrt{\frac{ \EE\Delta_{\wh g}}{n} }+  \frac{\kappa  \sqrt{\log (2/\eta)}}{n}   \r)
\end{align}
provided that $\delta$ satisfies \eqref{condi_delta_1}, \eqref{condi_delta_2},  \eqref{condi_delta_3} and \eqref{condi_delta_4}.

For any  $\delta\ge 0$, define
\begin{equation*}
    \begin{split}
        & \psi_1(\delta) : =  c_1 \max\{ \kappa\|\fh \|_K, 1\} ~    \psi_x (\delta)+ c_2\kappa^2\|\fh\|_K^2\frac{\log(2/\eta)}{n}, \\
       &  \psi_2(\delta) : = C   \max\{ \kappa\|\fh \|_K ,1\}   ( \|\fh\|_K   \vee 1) \l( R(\delta)+\sqrt{\frac{\EE\Delta_{\wh g} }{n}}
       +  \frac{\kappa {\log(2/\eta)}}{n}  \r) .   
    \end{split}
\end{equation*}
 It is easy to see that both  $\psi_1$ and $\psi_2 $ 
 are sub-root. Write the fixed point of $\psi_1$ as ${\delta}_1$.
 Since $ \psi_x(\delta) \le \psi_1(\delta)$  for any $\delta\ge 0$, invoking \cref{lem_subroot_monotone} ensures
\begin{align}\label{bd_delta_star}
    \delta_x \le \delta_1. 
\end{align}
So it suffices to derive an upper bound for $\delta_1$.

To this end, note that $\delta_{n,\wh g}$ in \eqref{def_delta_n_hatg} satisfies the conditions \eqref{condi_delta_2} and \eqref{condi_delta_3} for sufficiently large $C_0$ in \eqref{def_delta_n_hatg}. By noting that $\delta_{n,\wh g}\ge \delta_n$, applying \cref{lem_subroot_monotone} yields $R(\delta_{n,\wh g})  \le  \delta_{n,\wh g}$ so that $\delta_{n,\wh g}$ also satisfies the condition \eqref{condi_delta_4}. 
By definition of fixed point,  $\delta_1$ satisfies the condition \eqref{condi_delta_1}.
If $ \delta_1  \le   \delta_{n, \wh g} $,
the proof is completed.
Otherwise, $\delta_1$ satisfies \eqref{condi_delta_1}, \eqref{condi_delta_2}, \eqref{condi_delta_3} and \eqref{condi_delta_4}
so  that the inequality \eqref{bd_tilde_psi} holds for $\delta_1$. 
By adding and multiplying terms on both sides of \eqref{bd_tilde_psi},  we find that 
\begin{align}\label{link_psi_1_psi_2}
   \delta_1  =  \psi_1(\delta_1)  
    \le  \psi_2(\delta_1)   . 
\end{align}
Furthermore, the fact that  $R(\delta_{n,\wh g})  \le  \delta_{n,\wh g}$, together with the definition of $\delta_{n,\wh g}$ in \eqref{def_delta_n_hatg}, 
leads to
\begin{equation}\label{bd_psi_2}
\begin{aligned}
  \psi_2(\delta_{n, \wh g} )\le C  \max\{ \kappa\|\fh \|_K ,1\}    ( \|\fh\|_K   \vee 1)  \delta_{n, \wh g}. 
\end{aligned}
\end{equation}
Since $\psi_2( \delta)/\sqrt{\delta}$ is non-increasing in $\delta$, we find that
\begin{align*}
\sqrt{\delta_1}  \overset{\eqref{link_psi_1_psi_2}} {\le}   \frac{\psi_2(\delta_1)}{\sqrt{\delta_1}}   \le \frac{\psi_2({\delta_{n,\wh g}})}{\sqrt{\delta_{n,\wh g}}} \overset{\eqref{bd_psi_2}}{\le}
C  \max\{ \kappa\|\fh \|_K ,1\} ( \|\fh\|_K   \vee 1)   \sqrt{\delta_{n, \wh g}}
\end{align*}
so that
\[
\delta_1  \le C  V_K ~  \delta_{n, \wh g} , 
\]
which concludes the bound in Lemma \ref{lem_bd_td_delta}.
\end{proof}

Combining Lemmas \ref{lem_bd_h_f} and \ref{lem_bd_td_delta} yields the following corollary.

\begin{corollary}\label{cor_bd_h_f}
  Grant Assumptions \ref{ass_f_H},   \ref{ass_bd_K} and \ref{ass_mercer}.   Fix any $\eta\in(0,1)$. With probability at least $1-\eta$, for any $f\in \cF_b$, one has 
    \begin{align*}
        \EE[h_f(Z,X)] ~ \le  ~ 2\EE_n[h_f(Z,X)]+
   &     C \max\left\{\kappa^2\|\fh \|_K^2 +\|f^*\|_\i^2,1\right\} V_K  ~   \delta_{n,\wh g}   + 4\|\fh - f^*\|_\rho^2  . 
    \end{align*}
\end{corollary}

\begin{proof}
 
By noting that 
\begin{align*}
  \kappa \|\fh\|_K\left(\kappa \|\fh\|_K + \|f^*\|_\i\right)  \frac{\log (1/\eta)}{n}+ 
\|\fh \|_K^2  \EE\Delta_{\wh g}~ \le ~ C  \max\left\{\kappa^2\|\fh \|_K^2 +\|f^*\|_\i^2,1\right\} V_K~ \delta_{n,\wh g}, 
\end{align*}
we complete the proof by applying
Lemmas \ref{lem_bd_h_f}  and  \ref{lem_bd_td_delta}. 
\end{proof}

\subsubsection{Local uniform upper bounds of the cross-term}\label{sec_cross_term}
In this section we establish the uniform bound of the cross-term  $\EE_n [ \epsilon  (\fh \circ \wh g)(X)- \epsilon   (f\circ \wh g)(X) ]$  over a local ball.  Recall the definitions of the function $\xi_{f}$ from \eqref{def_g_f} 
and the function class $\Xi_{b}$ from \eqref{def_class_g_b}.


To facilitate the proof in this section, 
introduce
the empirical covariance   operator  $\wh T_{x,K}$ with respect to  the predicted inputs $\wh Z_i = \wh g(X_i)$ for $i \in [n]$,  defined as 
\begin{align} \label{def_wh_T_x_K}
    \wh T_{x,K}:  \cH_K\to \cH_K, \qquad   \wh T_{x,K} f: = \frac{1}{n}\sum_{i=1}^n f(\widehat{Z}_i)K\l(\widehat{Z}_i, \cdot\r). 
\end{align}
The operator $\wh T_{x,K}$ has the following useful properties. 
First, 
by  the reproducing property in 
$\cH_K$,  for any $f,g\in \cH_K$,  we have
\begin{align*}
    \langle f, \wh T_{x,K} ~ g\rangle_K= \frac{1}{n} \sum_{i=1}^n f(\wh Z_i ) g(\wh Z_i ) 
\end{align*}
so that 
\begin{align} \label{eq_repro}
\EE_n \l[(f\circ \wh g)(X)\r]^2 = \frac{1}{n}\sum_{i=1}^n f^2(\wh Z_i )= \langle f,  \wh {T}_{x,K}f \rangle_K, \qquad
 \text{for any $f\in \cH_K$.}
\end{align}
Second, by using \cref{lem_wh_T_x},  the largest $n$ eigenvalues of  $\wh T_{x,K}$ coincide with those of $\bK_x$, while the remaining eigenvalues are all zero.
Write   $\{\wh \mu_{x,j}\}_{j=1}^\infty $ as the eigenvalues  of $\wh T_{x,K}$  arranged in descending order with $\wh \mu_{x,j} = 0$ for $j>n$, and  $\{ \wh \phi_{x,j} \}_{j=1}^\infty $ as the associated eigenfunctions which forms an orthonormal basis   $\cH_K$.

\begin{lemma}\label{lem_cross_term_1}
Grant model \eqref{model} with Assumptions \ref{ass_f_H}--\ref{ass_reg_error}.     
Fix any $\eta\in (0,1)$. 
For any $q>0$,  
with probability $1-\eta$, the following holds uniformly over $\{f\in \cF_b: \EE_n[\xi_f^2(X)] \le q\}$, 
    \[
        \EE_n [ \epsilon  (\fh \circ \wh g)(X)- \epsilon   (f\circ \wh g)(X) ] ~ \le  ~  C\gamma_\epsilon (\|\fh\|_K \vee 1 )\sqrt{\log(1/\eta)}   ~ \wh R_x ( q ).
    \]
\end{lemma}

\begin{proof}
Start by noting that
any  function  $f\in \cH_K$  can be expanded as  
$f=\sum_{j=1}^\i \alpha_j \wh \phi_{x,j} $ with $\alpha_j=\langle f,\wh \phi_{x,j} \rangle_K$. 
Similarly, $\fh$  can be written as $\fh=\sum_{j=1}^\i \alpha^*_j \wh \phi_{x,j} $ with $\alpha^*_j=\langle \fh,\wh \phi_{x,j}\rangle_K$. Write $\beta_j= \alpha_j-\alpha_j^*$ for $j=1,2,\ldots$, so that
\[
 \|f-\fh\|^2_K = \sum_{j=1}^\infty \beta_j^2. 
\]
Additionally, from \eqref{eq_repro}, we have
\[
\EE_n [\xi_f^2(X)] = \EE_n \l[(f \circ \wh g)(X)-(\fh \circ \wh g)(X)\r] ^2   = \langle f-\fh, \wh T_{x,K}(f-\fh)\rangle_K =\sum_{j=1}^\infty \wh {\mu}_{x,j}\beta_j^2.  
\]
This implies  that $f \in \cF_b $ satisfies $\EE_n[\xi_f^2(X)]\le q$  
if and only if the vector  $\beta = (\beta_1, \beta_1,\ldots)^\T\in \RR^\i$ 
satisfies 
\[
\sum_{j=1}^\infty \wh {\mu}_{x,j }\beta_j^2\le q, \qquad \sum_{j=1}^\infty  \beta_j^2 \le  9\|\fh\|_K^2 . 
\]
Additionally, any vector $\beta$ 
satisfying  the above constraints 
must belong to  the ellipse class
\begin{align} \label{def_elli_cE}
    \mathcal{E}: = \l\{\beta \in \RR^\infty:  
\sum_{j=1}^\i  \beta_j^2  {\nu}_j \le 1+9\|\fh\|_K^2  \r\},
\end{align}
where we write  $\nu_j = \max\{\wh \mu_{x,j}/q ,1\}$ for any $j=1,2,\ldots$. 

Observe that
\begin{equation}\label{key_ineq_cross_term}
    \begin{aligned} 
     &  \sup_{f\in \cF_b,~ \EE_n [\xi^2_f(X)]\le q} \left|\frac{1}{n}\sum_{i=1}^n\epsilon_i\bigl((\fh \circ \wh g) (X_i)-(f \circ \wh g) (X_i)\bigr)\right| \\
 &  \le   \sup_{\beta \in  \cE} \left|\frac{1}{n}\sum_{i=1}^n\epsilon_i \sum_{j=1}^\i \beta_j  (\wh \phi_{x,j}  \circ \wh g) (X_i) \right|\\
&  \le   \sup_{\beta \in  \cE} \frac{1}{n} \l({\sum_{j=1}^\i \beta_j^2\nu_j }\r)^{1/2 }\l(\sum_{j=1}^\i \frac{1}{\nu_j}\l(\sum_{i=1}^n \epsilon_i (\wh \phi_{x,j}  \circ \wh g) (X_i)  \r)^2\r)^{1/2} &&\text{by Cauchy-Schwarz inequality}
  \\
 &   \le  \frac{\sqrt{1+9\|\fh\|_K^2}}{n} \l(\sum_{j=1}^\i \frac{1}{\nu_j}\l(\sum_{i=1}^n \epsilon_i (\wh \phi_{x,j}  \circ \wh g) (X_i)  \r)^2\r)  ^{1/2}&&\text{by $\cE$.}
    \end{aligned}
\end{equation}
To proceed,     let $H\in \RR^{n\times n}$ with
\[
[H]_{ab}=\sum_{j=1}^\infty \frac{1}{\nu_j } (\wh \phi_{x,j}  \circ \wh g) (X_a) (\wh \phi_{x,j}  \circ \wh g) (X_b), \qquad\forall~   a,b\in[n]. 
\]
Then, with $\epsilon=(\epsilon_1,...,\epsilon_n)^\T$, we have  
\begin{align}\label{la_eq15}
\sum_{j=1}^\infty  \frac{1}{\nu_{j}}\l(\sum_{i=1}^n \epsilon_i (\wh \phi_{x,j}  \circ \wh g) (X_i) \r)^2 = \epsilon^\T H ~ \epsilon. 
\end{align}
Applying  \cref{tail_eq_qua_gaussian} with $A=H$ and   $\xi= \epsilon$  gives that for any $\eta \in (0,1)$,
\begin{align}\label{ineq_quad}
   \PP\left\{\epsilon^\T H \epsilon\le \gamma_\epsilon^2\left( \tr(H)+2\sqrt{\tr(H^2)\log(1/\eta)}+2\|H\|_\op \log(1/\eta)\right) \right\} \ge 1-\eta
\end{align}
where $\gamma_\epsilon^2$ is the sub-Gaussian constant in Assumption \ref{ass_reg_error}. 
Note that $\tr(H^2)\le \|H\|_\op \tr(H)$ and
\begin{align*}
  \|H\|_\op\le \tr(H) &=\sum_{i=1}^n\sum_{j=1}^\infty\frac{1}{\nu_{j}}  \b((\wh \phi_{x,j}  \circ \wh g) (X_i)\b)^2 \\
  & = n  \sum_{j=1}^\infty \frac{1}{\nu_j}  \EE_n \l [ (\wh \phi_{x,j}\circ \wh g)(X)\r] ^2  
  \\
    & =  n   \sum_{j=1}^\infty \frac{1}{\nu_j} \l \langle \wh \phi_{x,j} , \wh T_{x,K} \wh \phi_{x,j} \r \rangle_K&&\text{by \eqref{eq_repro}} \\
  & =  n   \sum_{j=1}^\infty \frac{\wh \mu_{x,j}  }{\nu_j}\\
 & = n \sum_{j=1}^\infty  \min \{q, \wh \mu_{x,j} \} =  n^2 \wh R_x ^2(q),
 \end{align*}
 where the last step follows from  the definition of $\wh R_x(\delta)$ in \eqref{def_em_kercomp} and  \cref{lem_wh_T_x}. 
By combining with \eqref{la_eq15} and \eqref{ineq_quad},
for any
$\eta\in (0,1)$, with probability at least $1-\eta$, we have
\begin{align}\label{bd_by_Wh_R}
  \frac{1}{n}\l(\sum_{j=1}^\i \frac{1}{\nu_j}\l(\sum_{i=1}^n \epsilon_i (\wh \phi_{x,j}  \circ \wh g) (X_i)  \r)^2\r)  ^{1/2}  \le  C  \gamma_\epsilon \sqrt{\log(1/\eta)}~ \wh  R_x \l(q\r) .
\end{align}
Since \eqref{bd_by_Wh_R} does not depend on $f\in \cF_b$, combining with \eqref{key_ineq_cross_term} implies 
\begin{align*}
    & \sup_{f\in\cF_{b}, ~ \EE_n[\xi_f^2(X)]\le q} ~ \left|\frac{1}{n}\sum_{i=1}^n\epsilon_i\bigl((\fh \circ \wh g) (X_i)-(f \circ \wh g) (X_i)\bigr)\right|\\
&\le ~ C\gamma_\epsilon (\|\fh\|_K \vee 1 )\sqrt{\log(1/\eta)}  ~ \wh R_x  (q),
\end{align*}
which completes the proof.  
\end{proof}

The following lemma bounds $\EE_n[\xi_f^2(X)]$ by its population-level counterpart.

\begin{lemma}\label{lem_cross_term_2}
Grant Assumptions  \ref{ass_f_H} and \ref{ass_bd_K}. 
      Fix any $\eta\in (0,1)$.
Then with probability at least $1-\eta$, for any $f\in \cF_b$,  one has
\begin{align*}
    \EE_n \l [  \xi_f (X)^2 \r]~ \le  ~ 
    2\EE \l [   \xi_f (X)^2  \r]
    +  C \l(   \max\{\kappa^2 \|\fh\|_K^2, 1 \}  ~ \delta_x +  \kappa^2 \|\fh\|_K^2 \frac{\log(1/\eta)}{n}\r). 
\end{align*} 
\end{lemma}

\begin{proof}
Define the function class
\begin{equation}\label{def_class_g_b_square}
\Xi_{b}^2:= \b\{\xi^2_{f}: ~  \xi_{f} \in \Xi_{b} \b\} 
\end{equation}
and the scalar $V: =9 \kappa^2 \|\fh \|_K^2$. 
Moreover, define the 
functional $T: \Xi^2_{b} \to  \RR^+$
as 
\[
T(\xi_f^2) : = V~ \EE[\xi_f^2(X)]
\]
and the local Rademacher complexity of $\Xi_b^2$ as
\[
\localcomp\left( \bigl\{\xi_f^2 \in \Xi^2_{b}: ~  T(\xi_f^2)\le \delta\bigr\} \right). 
\]
Below we 
apply Lemma \ref{supp_lem_3} to the function class $\Xi^2_{b}$ with $T=T, L=0$ and $B= -a=b=V$.
\\

\noindent{\it Step 1: Verification of the boundedness of $\Xi^2_{b}$.}
By using  \eqref{bd_F_b_diff}, the function class  $\Xi^2_{b}$
is uniformly bounded with range $[-V, V]$.
\\

\noindent{\it  Step 2: Bounding the variance of $\xi_f^2(X)$.}
For any $\xi_f^2\in \Xi^2_{b}$, we have
$$\Var(\xi^2_f(X))\le \EE\l[\xi^4_f(X) \r] \le V~  \EE\l[\xi^2_f(X)\r]= T(\xi^2_f),
$$
which verifies the variance condition  with $T=T, B=V$ and $L=0$ required in  Lemma \ref{supp_lem_3}.
\\

\noindent{\it Step 3: Choosing a suitable sub-root function.} 
Define 
\[
\psi(\delta) : =  8\kappa\|\fh\|_K V~    \psi_x \left(\frac{\delta}{V}\right)
\]
for any $\delta\ge 0$,  where $\psi_x$ is a sub-root function defined in \eqref{def_psi_g_b_popu}. Using  \cref{lem_subroot_comp}  ensures that $\psi(\delta)$ is  sub-root in  $\delta$. 
 Let $\delta_\star$ be the fixed point of $\psi(\cdot)$.

By applying \eqref{bd_F_b}, 
for any $f_1,f_2\in \cF_{b}$ and $x\in \cX$,  we have
\begin{align*}
   \left|\xi^2_{f_1}(x)- \xi^2_{ f_2}(x)\right|
=&  \bigl |(f_1\circ \wh g)(x)+  (f_2\circ \wh g)(x)\bigr|~   \bigl |\xi_{f_1}(x)-\xi_{f_2}(x) \bigr|
\\
\le &  8 \kappa \|\fh\|_K~  
\bigl |\xi_{f_1}(x)-\xi_{f_2}(x) \bigr|
\end{align*}
which implies that 
$\xi_f^2( X_i)$ is $(8 \kappa \|\fh\|_K)$-Lipschitz in $\xi_{f}(X_i)$ for all $ i\in [n]$.
Then, by repeating the similar arguments of proving \eqref{lem4_eq1}, we can apply Ledoux–Talagrand contraction inequality in Lemma \ref{supp_lem_5} to obtain 
\begin{align*}
    \localcomp\left( \bigl\{\xi_f^2 \in \Xi_b^2, ~ T(\xi_f^2)\le \delta\bigr\} \right)
 \le&   8 \kappa \|\fh\|_K~\localcomp\left( \bigl\{\xi_{f} \in \Xi_b, 
 T(\xi_f^2) \le \delta\bigr\} \right)\\
 =& 8 \kappa \|\fh\|_K  ~  \psi_x\left(\frac{\delta}{V}\right)
\end{align*}
so that for every $\delta\ge \delta_\star$, 
\[
\psi(\delta)\ge   V  ~   \localcomp\left( \bigl\{\xi_f^2 \in \Xi_b^2,~  T(\xi^2_{ f})\le \delta\bigr\} \right).
\]
\vspace{0.04in}

\noindent{\it Step 4: Bounding  $\delta_\star$ by  $\delta_x$.} 
Using \cref{lem_subroot_comp}, one can deduce
\[
\delta_\star \le  \max\{ 64\kappa^2\|\fh\|_K^2,1  \} V ~ \delta_x 
\]
 with $\delta_x$ being the fixed point of $ \psi_x$.

Finally,  applying the second statement in Lemma \ref{supp_lem_3} by setting $\theta=2$ yields that for any $\eta\in (0,1)$, with probability at least $1 -\eta$, the following inequality holds uniformly over $f\in \cF_{b}$, 
\begin{align}\label{bd_g_f_em}
    \EE_n [\xi_f^2(X)]  \le 2 \EE [\xi_f^2 (X)] + C \max\{\kappa^2 \|\fh\|_K^2, 1 \} \delta_x { +  \kappa^2 \|\fh\|_K^2{ \log(1/\eta) \over n}} . 
\end{align}
This completes the proof. 
\end{proof}

For any $q>0$, let 
\begin{equation}\label{def_Fb_q}
 \cF_b(q) := \{f\in \cF_b:\EE[\ell_{f}(Y, X)]\le q\}.
\end{equation}
The following lemma gives a uniform bound of the cross-term over $\cF_b(q)$ by combining \cref{lem_cross_term_1} and \cref{lem_cross_term_2} together with relating $\EE [\xi_f^2(X)]$ to $\EE[\ell_f(Y,X)]$. 

\begin{lemma}\label{lem_epsilon}
 Grant model \eqref{model} with Assumptions \ref{ass_f_H}--\ref{ass_reg_error}.  
Fix any $\eta\in (0,1)$ and any $q>0$. With probability at least $1-\eta$, the following holds for any $f\in \cF_b(q) $, 
\begin{align*}
 \left|\frac{1}{n}\sum_{i=1}^n\epsilon_i\Bigl( (\fh \circ \wh g)(X_i)- (f \circ \wh g)(X_i)\Bigr)\right|\le   C\gamma_\epsilon (\|\fh\|_K \vee 1 )\sqrt{\log(1/\eta)}   ~ \wh R_x (\bar q ),
\end{align*}  
where   $C\ge 0$ is  some absolute constant and 
 \begin{equation}\label{def_bar_q}
    \bar q  := 4q + 16 \|\fh \|_K^2 \EE\Delta_{\wh g}    +  16\|\fh -f^*\|_\rho^2 + C    \max\{\kappa^2 \|\fh\|_K^2, 1 \}  ~ \delta_x +  C \kappa^2 \|\fh\|_K^2 \frac{\log(1/\eta)}{n}.
 \end{equation}
\end{lemma}



\begin{proof}

Observe that
\begin{equation}\label{la_eq21}
\begin{aligned}
\EE[\xi^2_f(X)]
&= \EE[(f \circ \wh g)(X)- (\fh \circ \wh g)(X) ]^2 \\
&\le 2\EE[(f \circ \wh g)(X)- f^*(Z)]^2 +2\EE[(f^*(Z)-(\fh  \circ \wh g)(X)]^2\\
&=2\EE[(f \circ \wh g)(X)- f^*(Z)]^2-2\EE[f^*(Z)-(\fh  \circ \wh g)(X)]^2+4\EE[f^*(Z)-(\fh  \circ \wh g)(X)]^2\\
&= 2\EE\big[(Y- (f \circ \wh g)(X))^2-(Y- (\fh  \circ \wh g)(X))^2\big]+4\EE[f^*(Z)-(\fh  \circ \wh g)(X)]^2 &&\text{by \eqref{stru_assu}}. 
\end{aligned} 
\end{equation}
By using \eqref{bd_irre_error} with $\theta =1$, we  have
\begin{align*}
    \EE[f^*(Z)-(\fh  \circ \wh g)(X)]^2 
   \le  2\|\fh-f^*\|_\rho^2 + 2  \|\fh \|_K^2 \EE\Delta_{\wh g} . 
\end{align*}
By combining the above results,
for $f\in \cF_{b}$,  we have
\[
\EE[\xi^2_f(X)]
\le 2\EE\left[\ell_{f}(Y, X) \right]  +8 \|\fh \|_K^2 \EE\Delta_{\wh g}  +8\|\fh-f^*\|_\rho^2.  
\]
so that applying \cref{lem_cross_term_2} yields that
for any $\eta\in (0,1)$, with probability at least $1-\eta$, 
\[
   \EE_n [\xi_f^2(X)]  \le \bar q, \qquad\text{for any }f\in \cF_b(q).
\]
Since $\wh R_x (\delta)$ is non-increasing in $\delta$, using \cref{lem_cross_term_1} 
yields the following inequality which holds for any $f\in \cF_b(q)$. 
\[
  \EE_n [ \epsilon  (\fh \circ \wh g)(X)- \epsilon   (f\circ \wh g)(X) ] ~ \le  ~  C \gamma_\epsilon (\|\fh\|_K \vee 1 )  \sqrt{\log(1/\eta)}  ~  \wh R_x \l(\bar q\r).
\]
So the proof is complete. 
\end{proof}

As a consequence of \cref{lem_epsilon},
the following corollary is useful for proving \cref{thm_risk}. Recall $\uline{\cF_b}$ from \eqref{def_Fb_subset}.

\begin{corollary}\label{bd_cross_term}
Grant the conditions in \cref{lem_epsilon}. Fix any  $\eta\in (0,1)$. 
Then with probability at least $1-\eta$, 
the following holds uniformly over $f\in \uline{\cF_b}$,
\begin{align*}
& \Bigl|4\EE_n [ \epsilon  (\fh \circ \wh g)(X)- \epsilon   (f\circ \wh g)(X) ] \Bigr| ~ \le ~   \lambda \|\fh\|_K^2.
 \end{align*}
\end{corollary}
\begin{proof}
Recall from \eqref{def_bar_q} the definition of 
$\bar q$. In this proof, 
for any $f\in \uline{\cF_b}$,  define
\begin{align}\label{def_q_lbd}
    q_\lambda : = 24 \lambda \|\fh\|_K^2    + 16\|f^*-\fh\|_\rho^2 +   C V_K \max\{\kappa^2 \|\fh\|_K^2,  1 \}     ~ \delta_{n,\wh g} 
\end{align}
with $V_K$ and $\delta_{n,\wh g}$ given in \eqref{def_V_K} and \eqref{def_delta_n_hatg}.
Invoking  
\cref{lem_bd_td_delta} yields that for any 
$\eta \in (0,1)$, with probability at $1-\eta$, we have $\bar q\le q_\lambda$.
Since $\wh R_x (\delta)$ is non-decreasing in $\delta$ and 
by applying  \cref{lem_epsilon}, 
 we conclude that
 for any $\eta\in (0,1)$, 
the following holds uniformly over $f\in \uline{\cF_b}$, 
\[
\PP\l(\B|4\EE_n [ \epsilon  (\fh \circ \wh g)(X)- \epsilon   (f\circ \wh g)(X) ] \B|\le  C (\|\fh\|_K \vee 1 ) \gamma_\epsilon \sqrt{\log(1/\eta)} ~ \wh R_x (q_\lambda)  \r) \ge 1-\eta. 
\]
Below we apply \cref{cor_lem7_and_lem8} and  \cref{lem_bd_wt_R}   to derive an upper bound for $\wh R_x (q_\lambda)$. 

Define the function
\[
\psi(\delta) :=  C_0 (\|\fh\|_K \vee 1 )   \gamma_\epsilon  \sqrt{\log(1/\eta)}~   R(\delta), \qquad \forall~  \delta\ge 0 , 
\]
with $C_0$ given in \eqref{def_delta_n_hatg}.
Recall that $\delta_n$ is the fixed point of $R(\delta)$. 
Using the second statement of    \cref{lem_subroot_comp} ensures that
$\psi(\cdot)$ is a sub-root function and its fixed point $\delta_\star$ satisfies
\[
\delta_\star \le  (C_0^2 \vee 1) (\|\fh\|^2_K \vee 1 )  ( \gamma_\epsilon  ^2 \vee 1){\log(1/\eta)}~   \delta_n. 
\]
By the definition of $q_\lambda$, it is easy to see that
\[
q_\lambda \ge   (C_0^2 \vee 1) (\|\fh\|^2_K \vee 1 )  ( \gamma_\epsilon  ^2 \vee 1){\log(1/\eta)}~  \delta_n
\]
so that  invoking  \cref{lem_subroot_monotone} gives 
\begin{align}\label{bd_R_delta_cross_term}
\psi(q_\lambda) =    C_0 (\|\fh\|_K \vee 1 ) 
 \gamma_\epsilon  \sqrt{\log(1/\eta)}~   R(q_\lambda) \le q_\lambda. 
\end{align}
Observe that
\begin{align*}
&   C_0 (\|\fh\|_K \vee 1 )   \gamma_\epsilon \sqrt{\log(1/\eta)}   ~ \wh R_x (q_\lambda) \\
 &\le   C C_0 (\|\fh\|_K \vee 1 )   \gamma_\epsilon \sqrt{\log(1/\eta)}   \l(   \wh R(q_\lambda) +  \sqrt{\frac{\EE\Delta_{\wh g} }{n}}  +  \frac{\kappa \sqrt{\log(2/\eta)}}{n}  \r)   &&\text{by \cref{cor_lem7_and_lem8}}\\
  & \le  C' C_0 (\|\fh\|_K \vee 1 )   \gamma_\epsilon \sqrt{\log(1/\eta)}  \l(  R(q_\lambda) +  \sqrt{\frac{\EE\Delta_{\wh g} }{n}}  +  \frac{\kappa  \sqrt{\log (2/\eta)}}{n}   \r) &&\text{by \cref{lem_bd_wt_R} }\\
  & \le  C'q_\lambda +  C' C_0 (\|\fh\|_K \vee 1 )   \gamma_\epsilon  \l(    \sqrt{\frac{\EE\Delta_{\wh g} \log(1/\eta)}{n}}  +  \frac{\kappa {\log (2/\eta)}}{n}   \r) &&\text{by \eqref{bd_R_delta_cross_term} }\\
    & \le  C'q_\lambda +  C' C_0 (\|\fh\|_K \vee 1 )   \gamma_\epsilon   \l(  {\frac{1 }{2}} \EE\Delta_{\wh g} +  \frac{3 (\kappa \vee 1)  {\log (2/\eta)}}{2n}   \r) \\
    & \le C'' q_\lambda &&\text{by \eqref{def_q_lbd}}.
\end{align*}
Finally,
since $q_\lambda \lesssim \lambda \|\fh\|_K^2 $ by \eqref{rate_lbd_proof}, 
choosing $C_0$ in \eqref{def_delta_n_hatg} sufficiently large and using \eqref{rate_lbd_proof}  complete the proof.
\end{proof}

The following lemma is used in the proof of this section, which states an essential connection between the eigenvalues of $\wh T_{x,K}$ and $\bK_x$. 

\begin{lemma} \label{lem_wh_T_x}
The number of non-zero eigenvalues of  
$\wh T_{x,K}$ is at most $n$. Additionally, 
the eigenvalues of $\bK_x$ are the same as the $n$ largest eigenvalues of $\wh T_{x,K}$. 
\end{lemma}
\begin{proof}
The proof can be done by largely following the argument used in the proof of Lemma 6.6. in \cite{bartlett2005local}. 
First, it is easy to see that the non-zero eigenvalues of  
$\wh T_{x,K}$ are at most $n$ in number because  the range of $\wh T_{x,K}$ is at most $n$-dimensional. 

To prove the second statement of \cref{lem_wh_T_x}, 
it suffices to consider the non-zero eigenvalues. To this end, 
for any $f\in \cH _K$, write
\[
f_n : = (f(\wh Z_1), \ldots, f(\wh Z _n ))^\T \in \RR^n. 
\]
Suppose that $\mu$ is  a non-zero
eigenvalue of $\wh T_{x,K}$ and   $f $ is its associated non-trivial eigenfunction, that is $\langle f,  \wh T_{x,K} f \rangle_K>0 $. 
Then for all $j\in [n]$,  
 \[
 \mu f ( \wh Z_j ) = (\wh T_{x,K} f)(  \wh Z_j ) \overset{\eqref{def_wh_T_x_K}}{=} \frac{1}{n}\sum_{i=1}^n f(\widehat{Z}_i)K( \wh Z_j,\widehat{Z}_i),
 \]
 implying that $\bK_x f_n  =  \mu f_n. $  Using  \eqref{eq_repro} to deduce
 \[
    \|f_n \|_2^2 = n  \langle f,  \wh T_{x,K} f \rangle_K >  0 
 \]
implying that $\bK_x f_n  =  \mu f_n. $  Together with 
 \[
    \|f_n \|_2^2 \overset{\eqref{eq_repro}}{=} n  \langle f,  \wh T_{x,K} f \rangle_K >  0 
 \]
 implies that $\mu$ is also 
 an eigenvalue of $\bK_x$. 
 On the other hand, suppose that $v$ is an non-trivial eigenvector (that is,  $ v^ \T \bK_x v>0 $) of $\bK_x$  with non-zero
 eigenvalue $\mu$. 
 Then by
 letting $f_v: =  n^{-1/2} \sum_{j=1}^n  v_j K( \wh Z_j, \cdot) $, we find that
 \begin{align*}
     \wh T_{x,K} f_v    = \frac{1}{\sqrt{n}}  \sum_{j=1}^n  v_j  \l( \frac{ 1}{n}  \sum_{i=1}^n  K( \wh Z_i, \wh Z_j) K(  \wh Z_i,\cdot)  \r )&  = \frac{1}{\sqrt{n}}   \sum_{i=1}^n  K(  \wh Z_i,\cdot)  
[ \bK_x v ]_i  \\
& =   \frac{ \mu }{\sqrt{n}}   \sum_{i=1}^n  v_i K(   \wh Z_i,\cdot)=\mu f_v,
 \end{align*}
implying that   either $\langle f_v ,  f_v \rangle_K   =0  $ or $\mu$ is  an eigenvalue of $ \wh T_{x,K} $. The former is impossible since 
\begin{align*}
 \langle f_v ,  f_v \rangle_K
 & =  \Bigl \|\frac{1}{\sqrt{n}} \sum_{j=1}^n  v_j K(  \wh Z_j, \cdot) \Bigr \|_K^2   =v^ \T \bK_x v
\end{align*}
by the reproducing property in $\cH_K$. 
This concludes the proof.
\end{proof}

\subsection{Technical lemmas for proving \texorpdfstring{\cref{lem_bd_td_delta}}{\texttwoinferior}
} 
\label{pf_lem_bd_td_delta}
Recall that  $\delta_x$ is fixed point of $ \psi_x$  in \eqref{def_psi_g_b_popu}. This section provides some useful lemmas for proving \cref{lem_bd_td_delta} that states an upper bound on $\delta_x$.

\subsubsection{Relating local Rademacher complexity to its empirical counterpart}\label{app_sec_radem_emp}

The following lemma states 
that  local Rademacher complexity of any 
bounded function class  can be bounded
 in terms of its empirical counterpart in probability. 

\begin{lemma}\label{supp_lem_4}
Let  $\mathcal{F}$ be a function class
with ranges in $[- G, G]$ and $U_1,\ldots ,U_n$ be the i.i.d. copies of some random variable $U$.
Fix any $t\ge 0$. 
For any $\delta$ such that
\[
\delta \ge 10 G ~  \localcomp\left(\l\{f\in \cF :~ \EE\l[f^2(U)\r]\le \delta\r\}\right) + 11 G^2t,
\]
with probability at least $1-2 e^{-nt}$, one has
\[
\localcomp\left(\l\{f\in \cF :~ \EE\l[f^2(U)\r]\le \delta\r\}\right) ~ \le ~2\sqrt{2}~  \emlocalcomp\left(\l\{f\in \cF :~ \EE_n \l[f^2(U)\r]\le \delta\r\}\right)  + 
2Gt.  
\]
\end{lemma}
\begin{proof}
For short, for any $\delta\ge 0$,  we write
\begin{equation*}
    \begin{split}
        & A(\delta) : = \localcomp\left(\l\{f\in \cF : ~ \EE\l[f^2(U)\r]\le \delta\r\}\right) ;\\
        &\wh A(\delta) : =    \emlocalcomp\left(\l\{f\in \cF : ~ \EE_n\l[f^2(U)\r] \le \delta\r\}\right). 
    \end{split}
\end{equation*}
For any $\delta\ge 0$,  define an intermediate variable as
\begin{align*}
\bar A (\delta):  = 
\wh \cR_n \left(\{f\in \cF : \EE[f^2(U)]\le \delta\}\right) =\mathbb{E}_{\varepsilon}\l[\sup_{f\in\cF,~  \EE[f^2(U)]
\le \delta} \frac{1}{n}\sum_{i=1}^n  \varepsilon_if(U_i)  \r].
\end{align*}
Taking expectations on both sides of the above equality gives
\[
  A (\delta) =   \mathbb{E} \l [\bar A (\delta) \r]. 
\] 
From  Lemma A.4  in \cite{bartlett2005local}, for any $t\ge 0$, with  probability at least $1-e^{-nt}$, one has
\begin{align}\label{supp_lem_4_eq4}
A(\delta )\le 2  \bar A(\delta) + 2G t. 
\end{align}
Additionally,  Corollary 2.2 in \cite{bartlett2005local} implies that whenever
\[
\delta \ge 10 G ~   A (\delta )+ 11 G^2t,
\]
with probability at least $1-2 e^{-nt}$,
one has
\begin{align*}
\bar A (\delta )\le \wh A (2\delta).
\end{align*}
Together with \eqref{supp_lem_4_eq4} yields that
\begin{align}\label{bd_A_delta}
A (\delta ) \le 2\wh A(2\delta )+
2Gt. 
\end{align} 
Finally, 
using   \cref{subroot} ensures that $\wh A(\delta )$ is sub-root in $\delta$. By  the definition  of sub-root function, $\wh A(\delta )/\sqrt{\delta}$ is non-increasing in $\delta$ so that
\[
\frac{\wh A(2\delta )}{\sqrt{2\delta}}\le \frac{\wh A(\delta )}{\sqrt{\delta}},
\]
which, combining with \eqref{bd_A_delta}, completes the proof. \end{proof}

\subsubsection{Relating empirical local Rademacher complexity in \texorpdfstring{\eqref{target_em_locom_sec}}{\texttwoinferior} to empirical kernel complexity \texorpdfstring{$\wh R_x (\delta)$}{\texttwoinferior}} \label{sec_pf_bd_emlocom}

This section is toward bounding from above the empirical Rademacher complexity in \eqref{target_em_locom_sec}. 

Recall the definitions of  $\xi_f$ from  \eqref{def_g_f} and  $\Xi_b$ from \eqref{def_class_g_b}. 
The empirical local Rademacher complexity in \eqref{target_em_locom_sec}  can be written as: for any $\delta\ge0$, 
\begin{equation}\label{target_em_locom}
\begin{aligned}
&\emlocalcomp\left(\big\{ \xi_{f}\in \Xi_b: ~ \EE_n[\xi^2_f(X)]\le \delta \big\}\right)\\
&= \EE_\varepsilon  \left[\sup_{\xi_{f}\in \Xi_b, ~ \EE_n[\xi^2_f(X)]\le \delta} \frac{1}{n}\sum_{i=1}^n \varepsilon_i \xi_{f}(X_i)\right]\\
&=\EE_\varepsilon \left[\sup _{f \in \mathcal{F}_{b}, ~ \EE_n[\xi^2_f(X)]\le \delta} \frac{1}{n}\sum_{i=1}^n \varepsilon_i ((f\circ\wh g)(X_i)-(\fh\circ\wh g)(X_i))\right] . 
\end{aligned}
\end{equation}
where the expectation $\EE_{\varepsilon}$ is taken over $n$ i.i.d. Rademacher variables $\varepsilon_1, \ldots, \varepsilon_n$.  
\begin{lemma}\label{lem_bd_em_local_rade_gb} 
Grant \cref{ass_mercer}. Then the following holds almost surely on  $X_1,...,  X_n$. 
\begin{align*}
\emlocalcomp\left(\big\{ \xi_{f}\in \Xi_b: ~ \EE_n[\xi^2_f(X)]\le \delta \big\}\right)\le  \sqrt{10}~  (\|\fh\|_K\vee 1)~ \widehat{R}_x (\delta).
\end{align*}
\end{lemma}

\begin{proof}
Define the 
scalar $V:= 1 + 9\|\fh\|_K^2$. 
By repeating the argument for proving \eqref{key_ineq_cross_term}, we have
\begin{equation} \label{bd_em_loc_1}
    \begin{aligned}
        &\EE_\varepsilon \left[\sup _{f \in \mathcal{F}_{b}, ~ \EE_n[\xi^2_f(X)]\le \delta} \frac{1}{n}\sum_{i=1}^n \varepsilon_i ((f\circ\wh g)(X_i)-(\fh\circ\wh g)(X_i))\right] \\
   &\le  \EE_\varepsilon \left[\sup_{\beta \in  \cE} \frac{\sqrt{V }}{n} \l({\sum_{j=1}^\i \frac{1}{\nu_j}\l(\sum_{i=1}^n \varepsilon_i (\wh \phi_{x,j}  \circ \wh g) (X_i)  \r)^2} \r)^{1/2}\right]. 
    \end{aligned}
\end{equation}
where $ \cE$ is defined in \eqref{def_elli_cE} and $\nu_j = \max\{\wh \mu_{x,j}/ \delta,1 \}$ for $j=1,2,\ldots$. Applying Jensen's inequality gives
\begin{align} \label{bd_em_loc_2}\nonumber
    &  \EE_\varepsilon \left[\sup_{\beta \in  \cE} \frac{\sqrt{V }}{n} \l(\sum_{j=1}^\i \frac{1}{\nu_j}\l(\sum_{i=1}^n \varepsilon_i (\wh \phi_{x,j}  \circ \wh g) (X_i)  \r)^2\r)^{1/2} \right]
     \\
     & \le \sup_{\beta \in  \cE} \frac{\sqrt{V }}{n} \l(\sum_{j=1}^\i \frac{1}{\nu_j}  \EE_\varepsilon \left[\l(\sum_{i=1}^n \varepsilon_i (\wh \phi_{x,j}  \circ \wh g) (X_i)  \r)^2 \right]\r)^{1/2}. 
\end{align}
Note that $\varepsilon_1,\ldots,\varepsilon_n$  are Rademacher variables so  that for any $j$, 
\[
 \EE_\varepsilon\l[\l(\sum_{i=1}^n \varepsilon_i (\wh \phi_{x,j}\circ \wh g)(X_i)\r)^2 \r] = \sum_{i=1}^n (\wh \phi^2_{x,j}\circ\wh g)(X_i). 
\]
Then we have
\begin{equation}\label{bd_em_loc_3}
   \begin{aligned}
       \sum_{j=1}^\i  \frac{1}{\nu_j}  \EE_\varepsilon \l[
\l(\sum_{i=1}^n\varepsilon_i (\wh \phi_{x,j}\circ \wh g)(X_i) \r)^2 \r] = & \sum_{j=1}^\i \frac{1}{\nu_j}\l(\sum_{i=1}^n (\wh \phi^2_{x,j}\circ\wh g)(X_i) \r)\\
=  &n \sum_{j=1}^\i \frac{1}{\nu_j} \l \langle \wh T_{x,K} \wh \phi_{x,j}, \wh \phi_{x,j}\r\rangle_K&&\text{by  \eqref{eq_repro}}\\
=& n\sum_{j=1}^\i \frac{\wh \mu_{x,j}}{\nu_j} = n \sum_{j=1}^\infty \min \big \{\delta, \wh \mu_{x,j}\big\}. 
   \end{aligned}
\end{equation}
Collecting \eqref{bd_em_loc_1}, \eqref{bd_em_loc_2}, \eqref{bd_em_loc_3} yields
\[
\EE_\varepsilon \left[\sup _{f \in \mathcal{F}_{b}, ~ \EE_n[\xi^2_f(X)]\le \delta} \frac{1}{n}\sum_{i=1}^n \varepsilon_i ((f\circ\wh g)(X_i)-(\fh\circ\wh g)(X_i))\right] \le \sqrt{\frac{V}{n} \sum_{j=1}^\infty \min \big \{\delta, \wh \mu_{x,j}\big\}}. 
\]
By the definition of $\wh R_x(\delta)$ in \eqref{def_em_kercomp}  and using \cref{lem_wh_T_x} to deduce
\[
\sqrt{\frac{1}{n} \sum_{j=1}^\infty \min \big \{\delta, \wh \mu_{x,j}\big\}} =\wh R_x(\delta), 
\]
 we complete the proof. 
\end{proof}


\subsubsection{Relating \texorpdfstring{$\wh  R_x(\delta)$}{\texttwoinferior} to  \texorpdfstring{$\wh R(\delta)$}{\texttwoinferior}} \label{sec_bd_wh_R_delta}
Recall  $\wh Z_i = \wh g(X_i)$ for any $i\in [n]$. By  Mercer's expansion of $K$ in \eqref{eq_eigen_decomp}, for each $j \in [n]$,  we have 
\begin{align}\label{trace_eq}
\b[\bK_x\b]_{jj}  = \frac{1}{n} \Phi(\wh Z_j)^\T 
\Phi(\wh Z_j) \quad\text{and}\quad \b[\bK\b]_{jj}  = \frac{1}{n} \Phi( Z_j)^\T  \Phi( Z_j), 
\end{align}
where we write
\[
\Phi(z)= (\sqrt{\mu_1} \phi_1(z), \sqrt{\mu_2}\phi_2(z), \ldots)^\T \qquad \forall ~ z\in \cZ. 
\]
By \cref{ass_bd_K} and  \eqref{eq_eigen_decomp}, it is easy to check that for any
$z\in \cZ$, $\Phi(z)\in \ell^2(\NN) $. 
Further define for any $\delta> 0$, 
\begin{align}\label{def_D_delta}\nonumber
    & \wh D _x (\delta)  : = 
    \tr\l((\bK_x+\delta \bI_n)^{-1}\bK_x\r)= \sum_{j=1}^n \frac{\wh \mu_{x,j} }{ \wh \mu_{x,j} +\delta},\\
    &\wh D(\delta)  := \tr\l((\bK+\delta \bI_n)^{-1}\bK\r)= \sum_{j=1}^n \frac{\wh \mu_j }{ \wh \mu_j +\delta} . 
\end{align}
Observe that for any positive values $a$ and $b$, 
\begin{align}\label{sandwich}
      \frac{a \wedge b}{2} =\frac{1}{2\max\{1/a,1/b\}}\le  \frac{ab}{a+b} \le   \frac{1}{\max\{1/a,1/b\}} = a \wedge b.  
\end{align}
As a result, for any $\delta>0$,  we have
\begin{align}\label{re_R_D_wh}
\frac{1}{2} \wh D_x(\delta) \le   \frac{n }{ \delta} \wh R_x^2 (\delta) \le  \wh D_x(\delta)  \qquad\text{and}\qquad \frac{1}{2} \wh D(\delta) \le   \frac{n}{ \delta}  \wh R^2 (\delta)\le  \wh D (\delta) . 
\end{align}
For each $i\in [n]$,  write
\begin{align}\label{def_Delta_i}
    \Delta_{i,\wh g}:  = \| K(Z_i,\cdot )- K(\wh g(X_i),\cdot)\|_K^2.
\end{align}
Note that  $ \Delta_{1,\wh g}, \ldots, \Delta_{n,\wh g}$ are i.i.d. copies of $\Delta_{\wh g}$ in \eqref{def_Delta}. 
Let $\bar \Delta_{\wh g}$ 
be the average of them, that is
$$
\bar \Delta_{\wh g}=\frac{1}{n}\sum_{i=1}^n \Delta_{i,\wh g}.
$$

\begin{lemma}\label{lem_bd_wh_R}
Grant Assumptions \ref{ass_bd_K} and  \ref{ass_mercer}. 
On the event $\{20  \bar \Delta_{ \wh g} \le \delta \}$,  we have 
\[
 \wh R_x (\delta) ~ \le  ~ C   \wh R(\delta) + C \sqrt{\frac{\bar\Delta_{\wh g}}{n}},
\]
where $C>0$  is some absolute constant. 
\end{lemma}

\begin{remark}
The result stated in \cref{lem_bd_wh_R}
is conditioning on both $X_1,...,  X_n$ and $Z_1,...,  Z_n$,  and all statements in this section should be understood to hold almost surely on them. 
\end{remark}

\begin{proof}
In view of \eqref{re_R_D_wh}, we turn to focus on bounding the difference between  $\wh D_x(\delta) $ and $\wh D(\delta)$.  To this end, define 
the operators
\begin{equation}\label{def_sigma}
    \begin{split}
        & \Sigma_x : \ell^2(\mathbb{N}) \to \ell^2(\mathbb{N}), \qquad    \Sigma_x= \frac{1}{n} \sum_{i=1}^n   \Phi(\wh Z_i)   \Phi(\wh Z_i)^\T;\\
        & \Sigma : \ell^2(\mathbb{N}) \to \ell^2(\mathbb{N}), \qquad   ~~  \Sigma= \frac{1}{n} \sum_{i=1}^n      \Phi( Z_i)   \Phi( Z_i)^\T. \\
    \end{split}
\end{equation}
For short, for any $\delta>  0$, 
we write
\begin{equation*}
    \begin{split}
        & A(\delta) = (\Sigma+\delta I)^{-\frac{1}{2}} ( \Sigma- \Sigma_x)(\Sigma+\delta I)^{-\frac{1}{2}};\\
        &B(\delta) = (\Sigma+\delta I)^{-1} ( \Sigma- \Sigma_x)(\Sigma+\delta I)^{-1}. \\
    \end{split}
\end{equation*}
It follows from  \eqref{trace_eq}  that
\begin{align*}
| \wh D_x(\delta)- \wh D(\delta)| =  \l| \tr \B( \b( \Sigma_x  +\delta I \b)^{-1}  \Sigma_x-\b (\Sigma  +\delta I \b )^{-1} \Sigma \B ) \r|. 
\end{align*}
Since
\[
\wh D_x(\delta) = \tr\b(( \Sigma_x + \delta I)^{-1}  \Sigma_x  \b)=\tr\b(I -\delta( \Sigma_x +\delta I)^{-1}\b),
\] 
and  similarly,  $\wh D(\delta)= \tr\b(I -\delta(\Sigma  +\delta I)^{-1}\b)$, we
have
\begin{align*}
 \b | \wh D_x(\delta)- \wh D(\delta)\b| 
= \delta \left| \tr\b((\Sigma+\delta I)^{-1}-(\Sigma_x+\delta I)^{-1}\b) \right| . 
\end{align*}
Further using the equality  $A^{-1}-B^{-1}= A^{-1}(B-A)B^{-1}$ for any invertible operators $A$ and $B$ gives
\begin{align}\label{diff_D_eq_1}
 \b | \wh D_x(\delta)- \wh D(\delta)\b| 
= \delta\left| \tr\b( (\Sigma+\delta I)^{-1}(\Sigma_x-\Sigma)(\Sigma_x+\delta I)^{-1}\b )\right|. 
\end{align}
On the event  
$\{\|A(\delta)\|_\op < 1\}$,
we find that
\begin{align*}
(\Sigma_x+ \delta{I})^{-1}&=  (\Sigma+\delta {I}+\Sigma_x-\Sigma)^{-1}\\
&=\B( (\Sigma+\delta{I})^{\frac{1}{2}}({I}- (\Sigma+\delta {I})^{-\frac{1}{2}}(\Sigma-\Sigma_x)(\Sigma+\delta {I})^{-\frac{1}{2}})(\Sigma+\delta{I})^{\frac{1}{2}}\B)^{-1}\\
&=\B((\Sigma+\delta{I})^{\frac{1}{2}}({I}-A(\delta))(\Sigma+\delta{I})^{\frac{1}{2}}\B)^{-1}\\
&=(\Sigma+\delta {I})^{-\frac{1}{2}}({I}-A(\delta))^{-1}(\Sigma+\delta {I})^{-\frac{1}{2}} 
\end{align*}
so that together with \eqref{diff_D_eq_1} gives
\begin{equation}\label{bd_wh_R_x_delta_1}
    \begin{aligned}
        & \b | \wh D_x(\delta)- \wh D(\delta)\b| \\
&=\delta\left|\tr\B( (\Sigma+\delta{I})^{-1} (\Sigma_x-\Sigma) (\Sigma+\delta {I})^{-\frac{1}{2}}({I}-A(\delta))^{-1}(\Sigma+\delta {I})^{-\frac{1}{2}}\B )\right|\\
&= \delta\left| \tr\B( (\Sigma+\delta {I})^{-\frac{1}{2}}A(\delta) ({I}-A(\delta))^{-1}(\Sigma+\delta{I})^{-\frac{1}{2}}\B )\right|
\\
&= \delta\left|  \tr\B( (\Sigma+\delta {I})^{-\frac{1}{2}}A(\delta) ({I}-A(\delta))^{-1}({I}-A(\delta)+A(\delta)) (\Sigma+\delta{I})^{-\frac{1}{2}} \B)\right|\\
&\le   \delta\left| \tr\b(B(\delta)\b)\right|+ \delta\left|\tr\B( (\Sigma+\delta {I})^{-\frac{1}{2}}A(\delta)({I}-A(\delta))^{-1}A(\delta)(\Sigma+\delta{I})^{-\frac{1}{2}} \B)\right|.
    \end{aligned}
\end{equation}
By the definition of the  trace of an operator, we have
\begin{equation}\label{bd_wh_R_x_delta_2}
    \begin{aligned}
    &\delta \tr\B((\Sigma+\delta {I})^{-\frac{1}{2}}A(\delta)({I}-A(\delta))^{-1}A(\delta)(\Sigma+\delta{I})^{-\frac{1}{2}} \B)  \\
&=\delta\B\| (\Sigma+\delta {I})^{-\frac{1}{2}}A(\delta)({I}-A(\delta))^{-\frac{1}{2}} \B\|_\hs^2\\
&\overset{(\romannumeral1)}{\le}  \delta \| (\Sigma+\delta{I})^{-1}\|_\op \|A(\delta)\|^2_\hs  \|({I}-A(\delta))^{-1} \|_\op\\
&\overset{(\romannumeral2)}{\le}   \frac{\|A(\delta)\|^2_\hs }{1- \|A(\delta) \|_\op}, 
\end{aligned}
\end{equation}
where (i) follows from the fact that $\|AB\|_\hs \le \|A\|_\op \|B\|_\hs$, 
(ii) holds since
\[
\|({I}-A(\delta))^{-1}\|_\op = \frac{1}{ \lambda_{\min}({I}-A(\delta))}=\frac{1}{ 1-\|A(\delta)\|_\op}.  
\]
By combining \eqref{bd_wh_R_x_delta_1} and \eqref{bd_wh_R_x_delta_2}, on the event  
$\{\|A(\delta)\|_\op < 1\}$, we  have the upper bound on $| \wh D_x(\delta)- \wh D (\delta)|$  of form 
\begin{align}\label{bd_D_diff}
 \b | \wh D_x(\delta)- \wh D(\delta)\b| 
\le \delta\big|\tr(B(\delta))\big|+  \frac{\|A(\delta)\|^2_\hs}{1- \|A(\delta) \|_\op}.
\end{align}
It remains to bound from above $\|A(\delta)\|_\op$, $ \delta  |\tr(B(\delta))|$ and $\|A(\delta)\|_\hs$ separately.\\

To this end,  start with the following decomposition using \eqref{def_sigma}.   
\begin{equation}\label{decom_sigma}
    \begin{aligned}
    \Sigma_x -\Sigma&= {1\over n}\sum_{i=1}^n \left\{ \big( \Phi(\wh Z_i) - \Phi(Z_i)\big)  \Phi(Z_i)^\T  +  \Phi(Z_i) \big(\Phi(\wh Z_i) - \Phi(Z_i)\big)^\T  \right.\\
		&\hspace{3cm} \left. +  \big(\Phi(\wh Z_i) - \Phi(Z_i)\big) \big(\Phi(\wh Z_i) - \Phi(Z_i)\big)^\T  \right\}.      
    \end{aligned}
\end{equation}
Observe that for each $i\in [n]$, 
\begin{align*}
    &\b\| \Phi(\wh Z_i) - \Phi(Z_i)\b\|_2^2=\sum_{j=1}^\i  \mu_j \b( \phi_j (\wh Z_i) - \phi_j (Z_i)\b)^2\\
&=\sum_{j=1}^\i  \mu_j\phi_j ( Z_i) \phi_j ( Z_i) -2 \sum_{j=1}^\i  \mu_j\phi_j ( Z_i) \phi_j (\wh Z_i) +\sum_{j=1}^\i  \mu_j\phi_j (\wh Z_i) \phi_j (\wh Z_i)  \\
&=K(Z_i, Z_i) -2 K(Z_i,\wh Z_i) +K(\wh Z_i,\wh Z_i) &&\text{by \eqref{eq_eigen_decomp}. }
\end{align*}
Combining with reproducing property and \eqref{def_Delta_i} to deduce
\begin{align*}
   & K(Z_i, Z_i) -2 K(Z_i,\wh Z_i) +K(\wh Z_i,\wh Z_i) \\
&=\langle K(Z_i, \cdot), K(Z_i, \cdot)\rangle_K  -2 \langle K(Z_i, \cdot), K(\wh Z_i, \cdot)\rangle_K +\langle K(\wh Z_i, \cdot), K(\wh Z_i, \cdot)\rangle_K \\
& =\b \| K(Z_i, \cdot)- K( \wh Z_i, \cdot)\b\|_K^2 =  \Delta_{i,\wh g}, 
\end{align*}
we have
\begin{equation}\label{lip_K}
\b\| \Phi(\wh Z_i) - \Phi(Z_i)\b\|_2^2=  \Delta_{i,\wh g}.
\end{equation}
\vspace{0.04in}

\noindent{\bf Bounding $\|A(\delta)\|_\op$.}   For any $u\in \ell^2(\NN)$, 
by writing
\begin{align*}
T_1(\delta) &:=   {1\over n} \sum_{i=1}^n     \big( \Sigma + \delta {I} \big)^{-\frac{1}{2}}\Phi(Z_i)  \big(\Phi(\wh Z_i) - \Phi(Z_i)\big)^\T  \big(\Sigma + \delta {I}\big)^{-\frac{1}{2}}, \quad  \rI_a :=   u^\T T_1(\delta)  u ; \\
T_1'(\delta)  &:=  {1\over n} 
\sum_{i=1}^n  \big( \Sigma + \delta {I} \big)^{-\frac{1}{2}}  \big( \Phi(\wh Z_i) - \Phi(Z_i)\big)   \Phi(Z_i)^\T \big(\Sigma + \delta {I}\big)^{-\frac{1}{2}}, \quad  \rI'_a :=  u^\T T'_1(\delta)  u ; \\
T_2(\delta) &:= 
{1\over n}\sum_{i=1}^n  \big( \Sigma + \delta I\big)^{-\frac{1}{2}}  \big(\Phi(\wh Z_i) - \Phi(Z_i)\big) \big(\Phi(\wh Z_i) - \Phi(Z_i)\big)^\T \big( \Sigma + \delta\big)^{-\frac{1}{2}},  \quad  \rII_a :=  u^\T T_2(\delta)  u ,
\end{align*}
it follows from \eqref{decom_sigma} and the triangle inequality that
\begin{align} \label{eq_decom_A_delta}
\|A(\delta) \|_\op \le   \|T_1(\delta)\|_\op+ \|T_1'(\delta)\|_\op+\|T_2(\delta)\|_\op. 
\end{align}
For $\rI_a$, by using  Cauchy-Schwarz inequality,  observe that
\begin{equation*}
\begin{aligned}
\rI_a \le&  \frac{1}{n}
\left( \sum_{i=1}^n  \B(  u^\T   \big( \Sigma + \delta {I} \big)^{-\frac{1}{2}}\Phi(Z_i) \B)^2\right)^{\frac{1}{2}}
\left(  \sum_{i=1}^n  \B( \big(\Phi(\wh Z_i) - \Phi(Z_i)\big)^\T  \big(\Sigma + \delta {I}\big)^{-\frac{1}{2}} u \B)^2 \right)^{\frac{1}{2}}\\
\le&  \frac{1}{n}
\left( \sum_{i=1}^n  \B(  u^\T   \big( \Sigma + \delta {I} \big)^{-\frac{1}{2}}\Phi(Z_i) \B)^2\right)^{\frac{1}{2}}
\left(  \sum_{i=1}^n  \b\| \Phi(\wh Z_i) - \Phi(Z_i) \b\|_2^2~   \b\|  \big(\Sigma + \delta {I}\big)^{-\frac{1}{2}} u \b\|_2^2 \right)^{\frac{1}{2}}\\
\le&  \frac{1}{n}
\left( \sum_{i=1}^n  \B(  u^\T   \big( \Sigma + \delta {I} \big)^{-\frac{1}{2}}\Phi(Z_i) \B)^2\right)^{\frac{1}{2}}
\left( \frac{1}{\delta}  \sum_{i=1}^n  \Delta_{i,\wh g}  \|u\|_2^2 \right)^{\frac{1}{2}} && \text{by \eqref{lip_K}}\\
= & \sqrt{\frac{\bar \Delta_{\wh g}}{\delta}}  \|u\|_2  
\left(\frac{1}{n}  \sum_{i=1}^n   u^\T   \big( \Sigma + \delta {I} \big)^{-\frac{1}{2}}\Phi(Z_i)\Phi(Z_i)^\T \big( \Sigma + \delta {I} \big)^{-\frac{1}{2}} u \right)^{\frac{1}{2}}\\
= &  \sqrt{\frac{\bar \Delta_{\wh g}}{\delta}}\|u\|_2
\left(   u^\T   \big( \Sigma + \delta {I} \big)^{-\frac{1}{2}}\Sigma \big( \Sigma + \delta {I} \big)^{-\frac{1}{2}} u \right)^{\frac{1}{2}} &&\text{by \eqref{def_sigma}}\\
\le & \sqrt{\frac{\bar \Delta_{\wh g}}{\delta}}  \|u\|_2^2,
\end{aligned}
\end{equation*}
implying
\begin{align}\label{bd_op_T_1}
    \|T_1(\delta)\|_\op\le \sqrt{ \bar \Delta_{\wh g}/\delta}. 
\end{align}
Repeating similar arguments gives
\begin{align}\label{bd_op_T_prime_1}
\|T_1'(\delta)\|_\op \le \sqrt{ \bar \Delta_{\wh g}/\delta}. 
\end{align}
For $\rII_a$, by using  Cauchy-Schwarz inequality and \eqref{lip_K},  we have
\begin{equation*}
\begin{aligned}
\rII_a\le&  \frac{1}{n}
 \sum_{i=1}^n  \B( \big(\Phi(\wh Z_i) - \Phi(Z_i)\big)^\T  \big(\Sigma + \delta {I}\big)^{-\frac{1}{2}} u \B)^2
\le  \frac{1}{n\delta} 
\sum_{i=1}^n  \Delta_{i,\wh g}  \|u\|_2^2  
=\frac{1}{\delta}\bar \Delta_{\wh g}  \|u\|_2^2, 
\end{aligned}
\end{equation*}
implying
\begin{align}\label{bd_op_T_2}
    \|T_2(\delta)\|_\op\le \bar \Delta_{\wh g}/\delta. 
\end{align}
Combining \eqref{eq_decom_A_delta}, \eqref{bd_op_T_1}, \eqref{bd_op_T_prime_1} and \eqref{bd_op_T_2} yields
\begin{align*}
    \|A(\delta)\|_\op\le 2 \sqrt{\bar \Delta_{\wh g}/\delta} + \bar \Delta_{\wh g}/\delta. 
\end{align*}
Under the condition that $20    {\bar \Delta_{\wh g}} \le  \delta$, one can deduce 
    $\|A(\delta)\|_\op\le 1/2. $\\ 


\noindent{\bf Bounding $\delta |\tr(B(\delta))|$.}   
By \eqref{decom_sigma} and writing
\begin{align*}
\rI_b &:=  \B| {1\over n} \sum_{i=1}^n\tr\B(  \big( \Sigma + \delta {I} \big)^{-1}  \Phi(Z_i)  \big(\Phi(\wh Z_i) - \Phi(Z_i)\big)^\T  \big(\Sigma + \delta {I}\big)^{-1} \B)\B|\\
\rI'_b  &:= \B|  {1\over n} 
\sum_{i=1}^n \tr\B(\big( \Sigma + \delta {I} \big)^{-1}   \big( \Phi(\wh Z_i) - \Phi(Z_i)\big)   \Phi(Z_i)^\T    \big(\Sigma + \delta {I}\big)^{-1}\B)\B|\\
\rII _b &:= 
{1\over n}\sum_{i=1}^n  \tr\B(  \big( \Sigma + \delta I\big)^{-1}  \big(\Phi(\wh Z_i) - \Phi(Z_i)\big) \big(\Phi(\wh Z_i) - \Phi(Z_i)\big)^\T \big( \Sigma + \delta I \big)^{-1}  \B),
\end{align*}
 we have
\[
\delta |\tr(B(\delta)) | ~ \le~  \delta~ \left (	\rI _b+\rI'_b  +\rII _b\right).
\]
For any elements $a_1,\ldots,a_n \in \ell^2(\NN)$ and $b_1,\ldots,b_n\in \ell^2(\NN)$, applying Cauchy-Schwarz inequality twice yields
\begin{align}\label{cs_ineq}
   \sum_{i=1}^n \tr(a_ib^\T_i)\le  \sum_{i=1}^n \|a_i\|_2~ \| b_i\|_2\le  \l(\sum_{i=1}^n \|a_i\|^2_2\r)^{\frac{1}{2}}\l(\sum_{i=1}^n \|b_i\|^2_2\r)^{\frac{1}{2}}.
 \end{align}
 Regarding $\rI_b$,  observe that
\begin{align*}
    \rI_b 
\le & \frac{1}{n}
\left( \sum_{i=1}^n  \B\| \big( \Sigma + \delta {I} \big)^{-1}    \big( \Phi(\wh Z_i) - \Phi(Z_i)\big) \B\|_2^2 \right)^{\frac{1}{2}}
\left(  \sum_{i=1}^n  \B\| \big(\Sigma + \delta {I}\big)^{-1} \Phi(Z_i)^\T \B\|_2^2 \right)^{\frac{1}{2}} 
&&\text{by   \eqref{cs_ineq}}\\
\le & \frac{1}{n \delta}
\left( \sum_{i=1}^n 
\left\|   \big( \Phi(\wh Z_i) - \Phi(Z_i)\big)  \right\|_2^2 \right)^{\frac{1}{2}}
\left(  \sum_{i=1}^n  \left\| \big(\Sigma + \delta {I}\big)^{-1}   \Phi(Z_i)^\T \right\|_2^2 \right)^{\frac{1}{2}}\\
\le &   \sqrt{\frac{\bar \Delta_{\wh g}}{\delta^2}} 
~ \left( \frac{1}{n} \sum_{i=1}^n  \left\| \big(\Sigma + \delta {I}\big)^{-1}   \Phi(Z_i)^\T \right\|_2^2 \right)^{\frac{1}{2}} &&\text{by  \eqref{lip_K}}\\
=    &  \sqrt{\frac{\bar \Delta_{\wh g}}{\delta^2 }} 
~ \left( \frac{1}{n} \sum_{i=1}^n  \tr\left( \big(\Sigma + \delta {I}\big)^{-1}     \Phi(Z_i)   \Phi(Z_i)^\T   \big(\Sigma + \delta {I}\big)^{-1}  \right) \right)^{\frac{1}{2}}\\
\le &   \sqrt{\frac{\bar \Delta_{\wh g}}{\delta^2 }} 
~ \sqrt{\frac{1 }{\delta}} \left(  \frac{1}{n} \sum_{i=1}^n  \tr\left( \big(\Sigma + \delta {I}\big)^{-\frac{1}{2}}     \Phi(Z_i)   \Phi(Z_i)^\T   \big(\Sigma + \delta {I}\big)^{-\frac{1}{2}}  \right) \right)^{\frac{1}{2}}\\
=  &  \sqrt{\frac{\bar \Delta_{\wh g}}{\delta^2 }} 
~\sqrt{\frac{\wh D(\delta) }{\delta}} &&\text{by \eqref{def_sigma}}. 
\end{align*}
Repeating similar arguments gives
\[
\rI'_b \le   \sqrt{\frac{\bar \Delta_{\wh g}}{\delta^2 }} 
~\sqrt{\frac{\wh D(\delta) }{\delta}}. 
\]
For $\rII_b$, we note that
\begin{equation}\label{bd_tr_T_2}
 \begin{aligned}
    \tr\b( T_2 (\delta) ) & = {1\over n}\sum_{i=1}^n  \tr\left(  \big( \Sigma + \delta I\big)^{-\frac{1}{2}}  \big(\Phi(\wh Z_i) - \Phi(Z_i)\big) \big(\Phi(\wh Z_i) - \Phi(Z_i)\big) ^\T  \big( \Sigma + \delta I\big)^{-\frac{1}{2}}  \right) \\
& \le \frac{1}{n} \sum_{i=1}^n \Big\| \big( \Sigma + \delta I\big)^{-\frac{1}{2}}  \big(\Phi(\wh Z_i) - \Phi(Z_i)\big) \Big\|_2^2  &&\text{by  \eqref{cs_ineq}}\\
& \le  \frac{\bar \Delta_{\wh g}}{\delta} &&\text{by  \eqref{lip_K}. }
\end{aligned}   
\end{equation}
Invoking Von Neumann's trace inequality gives
\begin{align*}
\rII_b = \tr\B( \big(\Sigma + \delta {I}\big)^{-\frac{1}{2}} T_2 (\delta) \big(\Sigma + \delta {I}\big)^{-\frac{1}{2}}  \B)\le  \frac{\bar \Delta_{\wh g}}{\delta^2}.
\end{align*}
Summing  the bounds of  $\rI_b,\rI'_b,\rII_b$ gives
\[
\delta |\tr(B(\delta)) | \le  2 \sqrt{ \frac{\bar \Delta_{\wh g} \wh D(\delta) }{\delta}}+  \frac{\bar \Delta_{\wh g}}{\delta}.
\]

 \noindent{{\bf Bounding} $\|A(\delta)\|^2_\hs$.}   
Start by noting that
\[
\|A(\delta)\|_\hs^2 =\tr\B( (\Sigma+\delta I)^{-\frac{1}{2}} ( \Sigma- \Sigma_x)(\Sigma+\delta I)^{-1} ( \Sigma- \Sigma_x)(\Sigma+\delta I)^{-\frac{1}{2}}
\B). 
\]
Further by applying \eqref{decom_sigma}, we have
\begin{equation}\label{decom_A_delta_hs}
\begin{aligned}
    \|A(\delta)\|_\hs^2\le &  \l|\tr\b(   T_1(\delta) T_1(\delta) \b)\r| +  \l|\tr\b(   T_1'(\delta) T_1'(\delta) \b)\r|  +   \l|\tr\b(   T_2(\delta) T_2(\delta) \b)\r|+\\
    &2~\l|\tr\b(   T_1(\delta) T_1'(\delta) \b)\r| +  2~\l|\tr\b(   T_1(\delta) T_2(\delta) \b)\r|  +   2~ \l|\tr\b(   T_1'(\delta) T_2(\delta) \b)\r|.  
    \end{aligned}
\end{equation}
For $\l|\tr\b(   T_1(\delta) T_1(\delta) \b)\r|$, it follows from \eqref{cs_ineq} that
\begin{align*}
&\l|\tr\b(   T_1(\delta) T_1(\delta) \b)\r|\\
&\le  \frac{1}{n}
\left( \sum_{i=1}^n  \|T_1(\delta)\|^2_\op  \B\| \big( \Sigma + \delta {I} \big)^{-\frac{1}{2}}    \big( \Phi(\wh Z_i) - \Phi(Z_i)\big) \B\|_2^2 \right)^{\frac{1}{2}}
\left(  \sum_{i=1}^n  \B\| \big(\Sigma + \delta {I}\big)^{-\frac{1}{2}} \Phi(Z_i) \B\|_2^2 \right)^{\frac{1}{2}} \\
&\le  \sqrt{\frac{\bar \Delta_{\wh g}}{\delta}} 
\left( \frac{1}{n}\sum_{i=1}^n  \B\| \big( \Sigma + \delta {I} \big)^{-\frac{1}{2}}    \big( \Phi(\wh Z_i) - \Phi(Z_i)\big) \B\|_2^2 \right)^{\frac{1}{2}}
\left(  \frac{1}{n}\sum_{i=1}^n  \B\| \big(\Sigma + \delta {I}\big)^{-\frac{1}{2}} \Phi(Z_i) \B\|_2^2 \right)^{\frac{1}{2}}  &&\text{by  \eqref{bd_op_T_1}} \\
&\le    \frac{\bar \Delta_{\wh g}}{\delta} 
\left(  \frac{1}{n}\sum_{i=1}^n  \B\| \big(\Sigma + \delta {I}\big)^{-\frac{1}{2}} \Phi(Z_i) \B\|_2^2 \right)^{\frac{1}{2}} &&\text{by  \eqref{lip_K}}\\
 &  \le  \frac{\bar \Delta_{\wh g} \wh D(\delta)}{\delta} 
 &&\text{by \eqref{def_sigma}}. 
\end{align*}
For $\l|\tr\b(   T_1'(\delta) T_1'(\delta) \b)\r|$ and $\l|\tr\b(   T_1(\delta) T_1'(\delta) \b)\r|$, repeating the similar arguments, and together with \eqref{bd_op_T_prime_1},  leads to
\begin{align*}
\l|\tr\b(   T_1'(\delta) T_1'(\delta) \b)\r|
 \le  \frac{\bar \Delta_{\wh g} }{\delta}\wh D(\delta)
 \end{align*}
 and 
 \[
 \l|\tr\b(   T_1(\delta) T_1'(\delta) \b)\r|\le \frac{\bar \Delta_{\wh g} }{\delta}\wh D(\delta). 
 \]
 For $\l|\tr\b(   T_2(\delta) T_2(\delta) \b)\r|$, we have
\begin{align*}
\l|\tr\b(   T_2(\delta) T_2(\delta) \b)\r| &\le \|T_2 (\delta)\|_\op  \tr\b( T_2 (\delta) \b) &&\text{by Von Neumann's trace inequality}\\
   & \le  \frac{\bar \Delta_{\wh g}}{\delta} \tr\b( T_2 (\delta) \b) &&\text{by \eqref{bd_op_T_2}}\\
      & \le  \l(\frac{\bar \Delta_{\wh g}}{\delta}\r)^2 &&\text{by \eqref{bd_tr_T_2}}.
\end{align*}
For $\l|\tr\b(   T_1(\delta) T_2(\delta) \b)\r|$, by  \eqref{cs_ineq}, \eqref{def_sigma}, \eqref{lip_K} and \eqref{bd_op_T_2},  
 observe that
\begin{align*}
&\l|\tr\b(  T_1(\delta) T_2(\delta) \b)\r|\\
&\le  \frac{1}{n}
\left( \sum_{i=1}^n  \|T_2(\delta)\|^2_\op  \B\| \big( \Sigma + \delta {I} \big)^{-\frac{1}{2}}  \big( \Phi(\wh Z_i) - \Phi(Z_i)\big) \B\|_2^2 \right)^{\frac{1}{2}}
\left(  \sum_{i=1}^n  \B\| \big(\Sigma + \delta {I}\big)^{-\frac{1}{2}} \Phi(Z_i)  \B\|_2^2 \right)^{\frac{1}{2}} \\
 &  \le \l( \frac{\bar \Delta_{\wh g} }{\delta}\r)^{\frac{3}{2}}\wh D(\delta) . 
\end{align*}
Repeating similar arguments gives
\begin{align*}
\l|\tr\b(  T'_1(\delta) T_2(\delta) \b)\r|  \le \l( \frac{\bar \Delta_{\wh g} }{\delta}\r)^{\frac{3}{2}} \wh D(\delta).
\end{align*}
Summing  the bounds for each term in the right-hand side of \eqref{decom_A_delta_hs} yields
\[
    \|A(\delta)\|_\hs^2\le 4~  \frac{\bar \Delta_{\wh g} \wh D(\delta)}{\delta} +4~  \l( \frac{\bar \Delta_{\wh g} }{\delta}\r)^{\frac{3}{2}} \wh D(\delta)+  \l(\frac{\bar \Delta_{\wh g}}{\delta}\r)^2. 
\]
\vspace{0.15in}

Finally, from \eqref{bd_D_diff}, 
summarizing the bounds of  $\|A(\delta)\|_\op$,  $\delta |\tr(B(\delta))|$ and $\|A(\delta)\|_\hs$  concludes that 
\begin{align} \label{bd_final_D_diff}
 \b | \wh D_x(\delta)- \wh D(\delta)\b|  \le 
 2~ \sqrt{ \frac{\bar \Delta_{\wh g} \wh D(\delta) }{\delta}}+  \frac{\bar \Delta_{\wh g}}{\delta}+ 8~  \frac{\bar \Delta_{\wh g} \wh D(\delta)}{\delta} +8~  \l( \frac{\bar \Delta_{\wh g} }{\delta}\r)^{\frac{3}{2}} \wh D(\delta)+ 2~ \l(\frac{\bar \Delta_{\wh g}}{\delta}\r)^2  
\end{align}
so that  combining with the elementary inequality $2 \sqrt{ab} \le a+b $ for any non-negative numbers $a,b$  gives
\begin{align*}
 \wh D_x(\delta)  \le  2 \wh D(\delta)+   \frac{2\bar \Delta_{\wh g}}{\delta}+ 8~  \frac{\bar \Delta_{\wh g} \wh D(\delta)}{\delta} +8~  \l( \frac{\bar \Delta_{\wh g} }{\delta}\r)^{\frac{3}{2}} \wh D(\delta)+ 2~ \l(\frac{\bar \Delta_{\wh g}}{\delta}\r)^2 .    \\
\end{align*}
So on  the event $\{ 20  \bar  \Delta_{\wh g} \le \delta\}$,  it is  easy to deduce
\begin{align*}
    \wh D_x(\delta)  \le 3 ~  \wh D(\delta)+   \frac{3\bar \Delta_{\wh g}}{\delta} . 
\end{align*}
Together with \eqref{re_R_D_wh} yields
\begin{align*}
    \wh R_x (\delta) \le \sqrt{ \frac{\delta \wh D_x(\delta)}{n}} \le & \sqrt{ \frac{ 3 \delta \wh D(\delta)}{n} + \frac{ 3 \bar\Delta_{\wh g}}{n}} \le  \sqrt{ 6 \wh R^2 (\delta) + \frac{ 3  \bar\Delta_{\wh g}}{n}} \le  C \l( \wh R(\delta) +  \sqrt{\frac{\bar\Delta_{\wh g}}{n}}  \r), 
\end{align*}
hence we complete the proof. 
\end{proof}


 The following lemma bounds from above $\bar \Delta_{\wh g} =(1/n)\sum_{i=1}^n \Delta_{i, \wh g} $.

\begin{lemma}\label{lem_bd_bar_Delta}
Grant \cref{ass_bd_K}. Fix any $\eta\in (0,1)$. Then with probability at least $1-\eta$, one has
\[
 \bar \Delta_{\wh g} 
  \le \frac{3}{2} 
  \EE\Delta_{\wh g}+ \frac{28 \kappa^2 \log ( 2/\eta) }{3 n}. 
\]
\end{lemma}

\begin{proof}
For each $i\in [n]$,   observe that
\begin{align*}
    \Delta_{i,\wh g} = &  \| K(Z_i, \cdot)- K( \wh g(X_i), \cdot)\|_K^2\\
    = & \langle K(Z_i, \cdot), K(Z_i, \cdot)\rangle_K  -2 \langle K(Z_i, \cdot), K(\wh g(X_i), \cdot)\rangle_K +\langle K(\wh g(X_i), \cdot), K(\wh g(X_i), \cdot)\rangle_K \\
     = & K(Z_i, Z_i) -2 K(Z_i,\wh g(X_i)) +K(\wh g(X_i),\wh g(X_i)), 
\end{align*}
where the last step holds by reproducing property in $\cH_K$.
Then by \cref{ass_bd_K}, one has the following upper bound uniformly over $i\in [n]$. 
\[
  \Delta_{i,\wh g} \le 4\kappa^2. 
\]
It follows that 
\[
\max_{i\in [n]}~  \EE[\Delta^2_{i,\wh g} ]\le  4\kappa^2 ~ \EE\Delta_{\wh g}.
\]
Invoking Bernstein's inequality in \cref{bern_rv} implies that for any $\eta\in (0,1)$, it holds with probability at least $1-\eta$ that
\[
\left|\frac{1}{n} \sum_{i=1}^n  \big( \Delta_{i,\wh g} - \EE[ \Delta_{i,\wh g}] \big) \right| \leq \frac{16 \kappa^2 \log ( 2/\eta)}{3 n}+\sqrt{\frac{8\kappa^2 \EE\Delta_{\wh g}  \log ( 2/\eta)}{n}}.
\]
Rearranging the terms and using the elementary inequality 
$2\sqrt{ab} \le a+b$ gives
\begin{align*}
  \bar \Delta_{\wh g} ~ \le ~  & \EE\Delta_{\wh g}+ \frac{16 \kappa^2 \log ( 2/\eta)}{3 n}+\sqrt{\frac{8\kappa^2 \EE\Delta_{\wh g} \ \log ( 2/\eta)}{n}}\\
  ~ \le ~ & \frac{3}{2} 
  \EE\Delta_{\wh g}+ \frac{28 \kappa^2  \log ( 2/\eta)}{3 n}. 
\end{align*}
This completes the proof.
\end{proof}

Combining 
Lemmas \ref{lem_bd_wh_R} and \ref{lem_bd_bar_Delta} gives the following corollary.
\begin{corollary}\label{cor_lem7_and_lem8}
Grant Assumptions \ref{ass_bd_K} and  \ref{ass_mercer}.  Fix any $\eta\in (0,1)$. For any $\delta$ such that
\[
 \l( 30 
  \EE\Delta_{\wh g}+  \frac{560 \kappa^2  \log ( 2/\eta)}{3n} \r)\le  \delta,  
\]
with probability at least $1-\eta$, we have
\[
 \wh R_x (\delta)\le  C \l( \wh R(\delta) +  \sqrt{\frac{\EE\Delta_{\wh g} }{n}}  +  \frac{\kappa \sqrt{\log(2/\eta)}}{n} \r),
\]
where $C>0$  is some absolute constant. 
\end{corollary}


\subsubsection{Relating   \texorpdfstring{$\wh R(\delta)$}{\texttwoinferior} to  \texorpdfstring{$R(\delta)$}{\texttwoinferior}} \label{sec_bd_wt_R}
Recall  the definitions of   $R(\delta) $ from \eqref{kernel_complexity} and $\wh R(\delta)$ from \eqref{def_em_kercomp_true_input}. 
In this part, we aim to derive an upper bound for   $\wh R(\delta)$ in terms of   $R(\delta)$. 

Introduce the  empirical covariance operator $\wh T_K$ with respect to $Z_1,\ldots,Z_n$, defined as
\[
\wh {T}_{K}: 
\cH_K \to \cH_K, \qquad 
\wh T_{K}f: = \frac{1}{n}\sum_{i=1}^n f({Z}_i )K( {Z}_i,\cdot). 
\]
By repeating the same  argument of proving \eqref{eq_repro}, we have 
\begin{align*}
    \EE_n[f(Z)] = \langle f,  \wh {T}_Kf \rangle_K,\qquad  \forall ~ f\in \cH_K. 
\end{align*}
Furthermore, by repeating the same  argument as that in the proof of \cref{lem_wh_T_x}, we claim that 
the eigenvalues of $\bK$ are the same as the $n$ largest eigenvalues of $\wh T_K$, and the remaining eigenvalues of $\wh T_K$ are all zero.
We are therefore allowed to write the associated eigenvalues with $\wh {T}_K$ as  $\{\wh {\mu}_j\}_{j=1}^\infty$, arranged in descending order. 

Note that $\wh  {T}_K$ depends on $Z_1,\ldots ,Z_n $, and 
its population counterpart  is known as the covariance  operator $T_K$, defined as
\begin{align*}
 {T}_{K}: 
\cH_K\to \cH_K, \qquad T_K f: =\int_{\mathcal{Z}} K_z f(z)\d\rho(z). 
\end{align*}
The operator $T_K$
is defined in the same way as the integral operator $L_K$ in \eqref{def_L_K} 
except that their domains and ranges are distinct.
It is worth mentioning that  $L_K$ and $T_K$ share the same eigenvalues $\{\mu_j\}_{j\ge 1}$, as discussed in
\cite{caponnetto2007optimal}. 
Recall the definition of  $\wh D(\delta)$ from
\eqref{def_D_delta}. 
Since  $\wh \mu_j=0$ for all $j> n$, for any $\delta>0$, we have
\begin{align*}
\wh D(\delta) = \sum_{j=1}^n \frac{\wh \mu_j }{ \wh \mu_j +\delta} = \sum_{j=1}^\i \frac{\wh \mu_j }{ \wh \mu_j +\delta} = \tr( ( \wh T_K+\delta I)^{-1} \wh T_K). 
\end{align*}
In the proof of this section, define for any $\delta>0$, 
\begin{align*}
    D(\delta):  = \tr( ( T_K+\delta I)^{-1}  T_K) = \sum_{j=1}^\i \frac{\mu_j}{\mu_j+\delta}. 
\end{align*}
Then by \eqref{sandwich},  we have
\begin{align}\label{re_R_D}
 \frac{1}{2}  D(\delta)\le     \frac{n R^2 (\delta)}{ \delta} \le D(\delta) . 
\end{align}

\begin{lemma}\label{lem_bd_wt_R}
Grant Assumptions \ref{ass_bd_K} and  \ref{ass_mercer}.
Fix  any  $\eta\in(0,1)$.   
For  any   $\delta>0 $  such that  $R(\delta)\le \delta$
and
\begin{align} \label{lb_n_delta_1}
n\delta \ge 32 \kappa^2  \log\l( \frac{256 (\kappa^2 \vee 1) }{\eta} \r) , 
\end{align}
 with  probability at least $1- \eta$, one has 
\begin{align*}
 \wh R(\delta) \le  C \l(R(\delta)  + \frac{ \kappa \sqrt{\log(2/\eta)}}{n} \r),  
\end{align*}
where $C>0$  is some absolute constant. 
\end{lemma}

\begin{proof}
We split the proof into two cases: $\mu_1< \delta$ and $\mu_1\ge  \delta$. \\

\noindent {\bf Case (1):} Start by observing that
\begin{equation}\label{case_1_bd_wt_D}
    \begin{aligned}
         \wh D (\delta) & = \tr\l( (\wh T_K +\delta I)^{-1}\wh T_K \r )\\
   &  \le \frac{1}{\delta} \tr(\wh T_K ) &&\text{by Von Neumann's trace inequality}\\
   &  \le  \frac{1}{\delta}  \l( \tr(  T_K ) +  \b| \tr( T_K) - \tr( \wh T_K  )\b | \r). 
    \end{aligned}
\end{equation}
Applying \cref{lem_bd_wt_T_K_minus_T_K} yields that for any $\eta\in(0,1)$,  with probability at least $1-\eta$,
\[
\l| \tr( T_K) - \tr(  \wh T_K   )  \r|  \le  \frac{1}{2} \tr(T_K)+  \frac{7 \kappa^2 \log (2/\eta)}{3 n}. 
\]
Together with \eqref{case_1_bd_wt_D} gives
\[
 \wh D (\delta)
 \le \frac{1}{\delta}  \l( \frac{3}{2}  \tr( T_K)  +  \frac{7 \kappa^2 \log (2/\eta)}{3 n} \r)
\]
By combining with the fact $\mu_1< \delta $ and  $\mu_1\ge\mu_2\ge\ldots $ to deduce
\begin{align}\label{eq_T_K_case_1}
    \tr(T_K) = \sum_{j=1}^\i\mu_j  =   \sum_{j=1}^\i \min\{ \mu_j , \delta\}  
    =n  R^2(\delta),  
\end{align}
we have
\[
 \wh D (\delta)
 \le  \frac{1}{\delta}  \l( \frac{3}{2}  n  R^2(\delta)  +  \frac{7 \kappa^2 \log (2/\eta)}{3 n} \r)
\]
so that combining   with \eqref{re_R_D_wh} leads to
\[
 \wh R(\delta) \le \sqrt{\frac{\delta \wh D (\delta)}{n }}
 \le \sqrt{ \frac{3}{2} R^2(\delta) + \frac{7 \kappa^2 \log(2/\eta)}{3n^2} }  \le C \l(R(\delta)  + \frac{ \kappa \sqrt{\log(2/\eta)}}{n} \r). 
\]

\medskip

\noindent {\bf Case (2):}
To be  short, for any $\delta >0$,   we write
\begin{align*}
 &  A(\delta )=(T_K+\delta I)^{-\frac{1}{2}}(T_K-\wh{T}_K)(T_K+\delta I)^{-\frac{1}{2}};
   \\
&   B(\delta )= (T_K+\delta I )^{-1}(T_K-\wh{T}_K)(T_K+\delta I )^{-1}. 
\end{align*}
By repeating the same argument of  proving \cref{lem_bd_wh_R},  on the event $\{\|A(\delta ) \|_\op\le \frac{1}{2} \} $, we have 
\begin{align}\label{decom_wt_D_diff_D}
 | \wh D (\delta)- D(\delta)|=
  \big|\delta \tr(B(\delta))\big|+ \frac{\big\|A(\delta) \big\|_{\hs}^2}{1- \|A(\delta ) \|_\op}. 
\end{align}
We proceed to bound $\|A(\delta ) \|_\op$, $\|A(\delta)\|_\hs^2$ and   $|\delta \tr(B(\delta))|$.\\

 To this end, for any $z\in \cZ$,  define the operator $K_z$ as
\[
K_z: \RR \to \cH_K, \qquad  y\mapsto y K(\cdot,z), \qquad \text{for any $y\in \RR$. }
\]
Let $K_z^\T: \cH_K \to \RR $ denote the adjoint operator of $K_z$ that satisfies
\begin{align*}
    K_z^\T f=\langle K(z,\cdot),f\rangle_K = f(z), \qquad  \text{for any }\ f\in \cH_K. 
\end{align*}
According to the definitions of $T_K$ and $ \wh  {T}_K$, we find that
\begin{align}\label{def_T_K}
  T_K f=\int_\mathcal{Z}   K_z  K_z^\T f \d\rho(z)=\int_\mathcal{Z}   K_z  f(z) \d\rho(z)
\end{align}
and 
\begin{align}\label{def_wt_T_K}
\wh  {T}_Kf= \frac{1}{n}\sum_{i=1}^n K_{Z_i}  K_{Z_i}^\T f.
\end{align}
\medskip

\noindent{\bf Bounding $\|A(\delta)\|_\op$.} 
Applying \cref{used_lem_wt_T_K} 
yields that for any $\eta\in (0,1)$, with probability at least $1-\eta$, we have
\begin{align} \label{bd_normalized_T_K_wh_T_K_op}
    \l  \|(T_K+\delta I)^{-\frac{1}{2}}(T_K-\wh{T}_K)(T_K+\delta I)^{-\frac{1}{2}}\r\|_\op  \le 
   \frac{4\kappa^2}{3n\delta }\log\l(\frac{4  \delta D(\delta)}{\eta
(\mu_1\wedge \delta)}\r)+\sqrt{\frac{2\kappa^2}{n\delta}\log\l(\frac{4  \delta D(\delta) }{\eta(\mu_1\wedge \delta) }\r)}.
\end{align}
Observe that
\begin{align*}
 \log\l(\frac{4  \delta D(\delta)}{\eta
(\mu_1\wedge \delta)}\r) 
 \overset{\text{$\mu_1 \ge \delta$}}{=}  \log\l(\frac{4  D(\delta)}{\eta
}\r) 
 \overset{\eqref{re_R_D}}{\le} &  \log\l(\frac{8 n R^2(\delta)}{\eta
\delta}\r) \\ 
 \le & \log\l(\frac{8 n \delta}{\eta}\r) &&\text{by $R(\delta) \le \delta$,}
\end{align*}
so that together with \eqref{bd_normalized_T_K_wh_T_K_op} gives
 \[
  \l  \|(T_K+\delta I)^{-\frac{1}{2}}(T_K-\wh{T}_K)(T_K+\delta I)^{-\frac{1}{2}}\r\|_\op  \le 
   \frac{4\kappa^2}{3n\delta }\log\l(\frac{8 n \delta}{\eta}\r)+\sqrt{\frac{2\kappa^2}{n\delta}\log\l(\frac{8 n \delta}{\eta}\r)}.
 \]
By noting that the right hand side of the above inequality decreases with $n\delta$, choosing any  $\delta$ such that \eqref{lb_n_delta_1}
yields
\begin{align}\label{bd_A_delta_op}
\|A(\delta)\|_\op= \Big\|\frac{1}{n}\sum_{i=1}^n \big(\EE[\zeta(Z_i)]-\zeta(Z_i)\big) \Big\|_\op  \le \frac{1}{12}  + \frac{1}{2\sqrt{2}} < \frac{1}{2}.
\end{align}
\medskip

\noindent{\bf Bounding} {$\|A(\delta)\|_\hs$.} 
For any $z\in \cZ$,  define the operator $\zeta(z)$  as
$$
\zeta(z)=(T_K+\delta I)^{-\frac{1}{2}} K_z K_z^\T (T_K+\delta I)^{-\frac{1}{2}}.
$$
It can be verified by applying \eqref{def_T_K} that  
\[
\EE\l[(T_K+\delta I)^{-\frac{1}{2}} K_Z K_Z^\T (T_K+\delta I)^{-\frac{1}{2}}\r]= (T_K+\delta I )^{-\frac{1}{2}}T_K(T_K+\delta I)^{-\frac{1}{2}}.
\]
Together with  \eqref{def_wt_T_K}, we can write
\begin{align} \label{A_delta_Zeta}
    A(\delta) = (T_K+\delta I)^{-\frac{1}{2}}(T_K - \wh{T}_K )(T_K +\delta I)^{-\frac{1}{2}} = \frac{1}{n}\sum_{i=1}^n \b(\EE[\zeta(Z_i)]-\zeta(Z_i)\b). 
\end{align}
For any  $z\in\cZ$, we have 
\begin{align} \label{trace_ineq_basic}
 \tr (K_z K_z^\T ) = \tr ( K_z^\T K_z ) =   K (z,z)\le \kappa^2 
\end{align}
so that
\begin{align*}
 \| \zeta(z) \|_\hs^2= \tr\l( \l(  ( 
 T_K +\delta I)^{-\frac{1}{2}} K_z K_z^\T (T_K +\delta I)^{-\frac{1}{2}}\r)^2\r)\le  \frac{\kappa^4}{\delta^2}.
\end{align*}
Furthermore,
we have
\begin{align*}
&\EE\big[\|\zeta(Z)\|_\hs^2\big]\\
&=\EE\l[\tr\B( \big((T_K+\delta I)^{-\frac{1}{2}} K_Z K_Z^\T (T_K+\delta I)^{-\frac{1}{2}}\big)^2\B)\r]\\
&\le   \sup_{z \in \cZ} \tr\l( K_z^\T (T_K+\delta I)^{-1}  K_z \r)  \EE\Big[\tr\big((T_K+\delta I)^{-\frac{1}{2}} K_Z K_Z^\T (T_K+\delta I)^{-\frac{1}{2}}\big)\Big]\\
&\le  \frac{\kappa^2}{\delta}  ~ \EE\Big[\tr\big((T_K+\delta I)^{-\frac{1}{2}} K_Z K_Z^\T (T_K+\delta I)^{-\frac{1}{2}}\big)\Big] &&\text{by \eqref{trace_ineq_basic}}\\
&=   \frac{\kappa^2D (\delta)}{\delta} &&\text{by \eqref{def_T_K}}. 
\end{align*}
Then,  by applying Lemma \ref{lem_concen} for the 
the Hilbert space of Hilbert-Schmidt
operators on $\cH_K$
with $H= 2\kappa^2\delta^{-1}$ and $S = \kappa^2\delta^{-1} D(\delta)$, for any $\eta\in (0,1)$, 
the following holds with probability at least $1-\eta$,  
\begin{align}\label{bd_A_delta_hs}
\|A(\delta)\|_\hs \overset{\eqref{A_delta_Zeta}}{=}  \Big\|\frac{1}{n}\sum_{i=1}^n\big (\EE[\zeta(Z_i)]-\zeta(Z_i)\big) \Big\|_\hs \le  \frac{4\kappa^2}{n\delta}\log \frac{2}{\eta}+ \sqrt{ \frac{2\kappa^2 D(\delta)}{n\delta}\log \frac{2}{\eta} }. 
\end{align}
\vspace{0.04in}

\noindent{\bf Bounding $\delta |\tr(B(\delta)) |$.} 
For any $z \in \cZ$, 
define
the random  variable  $\xi(z)$  as
$$
\xi(z) : =\delta \tr\b((T_K+\delta I)^{-1} K_z K_z^\T (T_K+\delta I)^{-1}\b),
$$
and it is clear that
\[
\EE \l[\xi(Z)\r]= \delta \tr((T_K+\delta I)^{-1} T_K (T_K+\delta I)^{-1}). 
\]
Therefore, we have
\begin{align*}
 \delta  |\tr(B(\delta)) | = \Big|\frac{1}{n}\sum_{i=1}^n\b(\EE[ \xi(Z_i)]-\xi(Z_i)\b)\Big|.
\end{align*}
For any  $z\in \cZ$, observe that
\begin{equation}\label{bd_xi_z}
    \begin{aligned}
        |\xi(z) |&=   \big| \delta \tr((T_K+\delta I)^{-1} K_z K_z^\T (T_K+\delta I)^{-1})\big| \\
&\le \tr( (T_K+\delta I)^{-1} K_z K_z^\T ) &&\text{by Von Neumann's trace inequality}\\
&\le \frac{ \kappa^2}{\delta} && \text{by \eqref{trace_ineq_basic}, }
    \end{aligned}
\end{equation}
which, by using Jensen's inequality,  further implies $ \b|\EE[\xi(Z)]\b|\le \kappa^2\delta^{-1}$.
Then we have
\[
 \sup_{i\in [n]}~ \big|  \EE[\xi(Z_i)]-\xi(Z_i)\big|\le \frac{2\kappa^2}{\delta}. 
\]
Further observe that
\begin{align*}
&\EE\l[ \EE[\xi(Z)]-\xi(Z)\r] ^2\\
&\le \big(\sup_{z\in \cZ}  \xi(z) \big) ~ \EE\b[ \xi(Z)\b]\\
&\le \frac{\kappa^2}{\delta}   ~  \EE\l[ \delta \tr\big((T_K+\delta I)^{-1} K_Z K_Z^\T (T_K+\delta I)^{-1}\big)\r]        &&\text{by \eqref{bd_xi_z}}\\
&\le  \frac{\kappa^2}{\delta} ~   \EE\l[ \tr((T_K+\delta I)^{-1}  K_Z K_Z^\T )\r] 
         &&\text{by Von Neumann’s trace inequality} \\
&= \frac{\kappa^2 D(\delta)}{\delta} &&\text{by \eqref{def_T_K}. }
\end{align*}
Then  by applying Lemma \ref{bern_rv} with $H =2\kappa^2\delta^{-1}$ and $S= \kappa^2\delta^{-1}D  (\delta)$, for any $\eta\in (0,1)$, 
the following holds with probability at least $1-\eta$, 
\begin{align}\label{bd_delta_tr_B}
 \delta  |\tr(B(\delta)) | = \Big|\frac{1}{n}\sum_{i=1}^n(\EE[ \xi(Z_i)]-\xi(Z_i))\Big| \le   \frac{4\kappa^2}{3n\delta}\log\frac{2}{\eta}+\sqrt{\frac{2\kappa^2 D(\delta)}{n\delta}\log\frac{2}{\eta}},
\end{align}
\vspace{0.15in}

Combining \eqref{decom_wt_D_diff_D}, \eqref{bd_A_delta_op}, \eqref{bd_A_delta_hs}  and \eqref{bd_delta_tr_B} gives
\begin{align}\label{bd_diff_D_and_wt_D}
    | \wh D(\delta)- D(\delta)|\le   \frac{4\kappa^2}{n\delta}\log \frac{2}{\eta}+ \sqrt{ \frac{2\kappa^2 D(\delta)}{n\delta}\log \frac{2}{\eta} }+ 2\l( \frac{4\kappa^2}{3n\delta}\log\frac{2}{\eta}+\sqrt{\frac{2\kappa^2 D(\delta)}{n\delta}\log\frac{2}{\eta}}\r)^2
\end{align}
so that \eqref{lb_n_delta_1} implies 
\begin{align*}
\wh D(\delta) \le C \l(   D(\delta) +  \frac{ \kappa^2}{n\delta}\log \frac{2}{\eta}\r).
\end{align*}
Together with \eqref{re_R_D} and repeating the same argument in Case 1 gives
\[
 \wh  R(\delta) \le C \l(R(\delta)  + \frac{ \kappa \sqrt{\log(2/\eta)}}{n} \r). 
\]

Combining both cases completes the proof. 
\end{proof}


The following two lemmas are used in the proof of  \cref{lem_bd_wt_R}. 
The first lemma bounds from above the trace of $T_K - \wh T_K$.

\begin{lemma}\label{lem_bd_wt_T_K_minus_T_K}
Grant  \cref{ass_bd_K}. 
For any fixed $\eta\in (0,1)$, the following inequality holds with probability at least $1-\eta$,
\[
|\tr(T_K)- \tr(\wh T_K)|  \le \frac{1}{2} \tr(T_K)+  \frac{7 \kappa^2 \log (2/\eta )}{3 n}. 
\]
\end{lemma}

\begin{proof}
By applying  \eqref{def_T_K} and  \eqref{def_wt_T_K}, we can write
\begin{align}\label{bd_T_K_wt_T_K}
    |\tr(T_K)- \tr(\wh T_K)|  = \l|  \frac{1}{n}\sum_{i=1}^n\b( \EE[\tr(K_{Z_i}  K_{Z_i} )] -\tr(K_{Z_i}  K_{Z_i}) \b)\r|. 
\end{align}
For each $i\in [n]$, according to \cref{ass_bd_K},  we have
\[
\tr(K_{Z_i}  K_{Z_i} ^\T)   =\tr(K_{Z_i} ^\T K_{Z_i}  )  =  K(Z_i,Z_i) \le \kappa^2, 
\]
which further implies that 
\[
\sup_{i\in [n]} ~  \EE\l[\tr\l(K_{Z_i}  K_{Z_i} ^\T\r) ^2\r] \le  \kappa^2 ~ \EE\l[\tr\l(K_{Z_i}  K_{Z_i} ^\T\r) \r] =  \kappa^2 \tr(T_K).
\]
Invoking Bernstein's inequality in \cref{bern_rv} implies that for any $\eta\in (0,1)$, it holds with probability at least $1-\eta$ that
\begin{align*}
     |\tr(T_K)- \tr(\wh T_K)|  =& \l| \frac{1}{n} \sum_{i=1}^n\b( \EE[\tr(K_{Z_i}  K_{Z_i} )] -\tr(K_{Z_i}  K_{Z_i}) \b)\r| &&\text{by \eqref{bd_T_K_wt_T_K}}\\
      \leq & \frac{4 \kappa^2 \log (2/\eta)}{3 n}+\sqrt{\frac{2\kappa^2 \tr(T_K) \log (2/\eta)}{n}}.
\end{align*}
Using the elementary inequality 
$2\sqrt{ab} \le a+b$ for any two non-negative numbers $a,b$ completes the proof. 
\end{proof}

The next lemma bounds from above the operator norm  of $(T_K+\delta I)^{-\frac{1}{2}}(T_K-\wh{T}_K)(T_K+\delta I)^{-\frac{1}{2}}$.

\begin{lemma}\label{used_lem_wt_T_K}
   Grant Assumptions \ref{ass_bd_K} and  \ref{ass_mercer}.  Fix any $\eta\in (0,1)$ and $\delta>0$.
Then with probability at least $1-\eta$, one has
\begin{align} \label{bd_target_lem14}
    \l  \|(T_K+\delta I)^{-\frac{1}{2}}(T_K-\wh{T}_K)(T_K+\delta I)^{-\frac{1}{2}}\r\|_\op  \le 
   \frac{4\kappa^2}{3n\delta }\log\l(\frac{4  \delta D(\delta)}{\eta
(\mu_1\wedge \delta)}\r)+\sqrt{\frac{2\kappa^2}{n\delta}\log\l(\frac{4  \delta D(\delta) }{\eta(\mu_1\wedge \delta) }\r)}.
\end{align}
\end{lemma}

\begin{proof}
We prove this lemma by following the same notation in the proof of \cref{lem_bd_wt_R}.\\

For any  $z\in \cZ$, we have
\begin{equation} \label{bd_zeta_z_op}
    \begin{aligned}
\|\zeta(z)\|_{\op}&
= \big\|(T_K+\delta I)^{-\frac{1}{2}} K_z K_z^\T (T_K+\delta I)^{-\frac{1}{2}} \big\|_{\op}
\\
&\le \frac{1}{\delta}  \big\|  K_z K_z^\T \big\|_{\op} \le \frac{1}{\delta}  \tr (K_z K_z^\T ) \overset{\eqref{trace_ineq_basic}}{\le}  \frac{\kappa^2}{\delta}. 
\end{aligned}
\end{equation}
It follows from   Jensen's inequality that
\begin{align*}
  \big\|\EE[ \zeta(Z) ] \big\|_{\op} \le  \EE\left[  \big\|\zeta(Z) \big\|_{\op} \right]  \le \frac{\kappa^2}{\delta}.
\end{align*}
Combining the bounds of $\|\zeta(z)\|_{\op}$ and $ \|\EE[ \zeta(Z) ]\|_{\op}$ gives
\begin{align*}
\sup_{i\in [n]}~  \big\|  \EE[\zeta(Z_i)]- \zeta(Z_i)\big\|_{\op}\le  \frac{2\kappa^2}{\delta}. 
\end{align*}
Regarding $\EE [ \EE\zeta(Z)-\zeta(Z)]^2$,  observe  that
\begin{align*}
&\EE \l [ \EE\zeta(Z)-\zeta(Z)\r]^2 \le \EE  \big[\zeta(Z)^2 \big] \\
&\le \l(\sup_{z\in \cZ} \zeta(z)  \r) ~  \EE \l [\zeta(Z) \r] 
\\
& \le \frac{\kappa^2}{\delta}
\EE \l[(T_K+\delta I)^{-\frac{1}{2}} K_ZK_Z^\T (T_K+\delta I)^{-\frac{1}{2}}\r]  
&&\text{by \eqref{bd_zeta_z_op}}\\
&=\frac{\kappa^2}{\delta}  (T_K+\delta I)^{-\frac{1}{2}}T_K (T_K+\delta I)^{-\frac{1}{2}} &&\text{by \eqref{def_T_K}}.
\end{align*}
By applying Lemma \ref{lemd3} with $H=2\kappa^2\delta^{-1}$ and $S= \kappa^2 \delta^{-1}(T_K+\delta I)^{-\frac{1}{2}}T_K (T_K+\delta I)^{-\frac{1}{2}}$  and noting that
\[
\|S\|_\op= \frac{\kappa^2}{\delta} \big \| (T_K+\delta I)^{-\frac{1}{2}}T_K (T_K+\delta I)^{-\frac{1}{2}} \big\|_\op\le \frac{\kappa^2}{\delta} ,
\]
for any $\eta\in (0,1)$,
the following holds with probability at least $1-\eta$, 
\begin{align}\label{lem15_eq3}
\Big\|\frac{1}{n}\sum_{i=1}^n \b(\EE[\zeta(Z_i)]-\zeta(Z_i)\b) \Big\|_\op 
\le   \frac{4\kappa^2}{3n\delta }\log\l(\frac{2\tr(S)}{\eta\|S\|_\op}\r)+\sqrt{\frac{2\kappa^2}{n\delta}\log\l(\frac{2\tr(S)}{\eta\|S\|_\op}\r)}.
\end{align}
By using \eqref{sandwich}, we find that
\[
 \b \| (T_K+\delta I)^{-\frac{1}{2}}T_K (T_K+\delta I)^{-\frac{1}{2}} \b\|_\op =\frac{{\mu}_1}{{\mu}_1+\delta}\\
 \ge  \frac{\mu_1\wedge \delta}{2\delta}
\]
so that  
\[
 \frac{\tr(S)}{\|S\|_\op}=  \frac{{D}(\delta)}{\| 
 (T_K+\delta I)^{-\frac{1}{2}}T_K (T_K+\delta I)^{-\frac{1}{2}}\|_\op} \le \frac{2  \delta {D}(\delta)}{(\mu_1\wedge \delta)}.
\]
Combining with \eqref{lem15_eq3}  completes the proof. 
\end{proof}

\begin{remark}
    It is worth mentioning that bounding the LHS in \eqref{bd_target_lem14}  is a key step in the proof of some recent work on the kernel-based methods when adopting the operator technique, as seen, for instance,  \cite{lin2020distributed,lin2020convergences}. However,  the proof in \cite{lin2020distributed} requires  $\mu_1 \ge c\delta$ for some constant $c>0$, a lower bound condition on $\mu_1$, and  the derived upper bound in  \cite{lin2020convergences} 
    is slightly weaker than ours in \cref{used_lem_wt_T_K}. Precisely, Lemma 23 in \cite{lin2020convergences}  states that the LHS in \eqref{bd_target_lem14} can be bounded from above (in probability) by 
    \[
      \frac{4\kappa^2}{3n\delta }\log\l(\frac{4 (D(\delta)+1)}{\eta\mu_1}\r)+\sqrt{\frac{2\kappa^2}{n\delta}\log\l(\frac{4 (D(\delta)+1)}{\eta\mu_1}\r)}.
    \]
    By comparison, we observe that if $\mu_1\le \delta$, the term $\delta D(\delta)/(\mu_1\wedge \delta)$ in our bound in \cref{used_lem_wt_T_K} reduces to $\delta D(\delta)/\mu_1$, which is much smaller $(D(\delta)+1)/\mu_1$ since  when applying  \cref{used_lem_wt_T_K}, we  typically choose $\delta$ to depend on $n$ such that $\delta\to 0$ as $n\to\i$. If $\mu_1> \delta$,  the term $\delta D(\delta)/(\mu_1\wedge \delta)$  reduces to $D(\delta)$ which is also much smaller $(D(\delta)+1)/\mu_1$ when $\mu_1 \to 0$ (note that $\mu_1\le \tr(T_K)\le \kappa^2$). 
\end{remark}

\subsection{Proof of the claim in \texorpdfstring{\cref{rem_pred_Z}}{\texttwoinferior}}\label{app_sec_proof_rem_pred_Z}
\begin{proof}
To verify the statements, by using the closed-form solution of KRR in \eqref{def_f_hat_closedform} 
and \eqref{coeff}, we first note that 
\[
    (\wh f\circ \wh g)(X)  = {1\over n}\sum_{i=1}^n \be_i^\T \left(\bK_x + \lambda \bI_n\right)^{-1} \bY ~   K(\wh g(X_i), \wh g(X))
\]
with $[\bK_x]_{ij} = n^{-1}K(\wh g(X_i), \wh g(X_j))$ for $i,j\in [n]$. For $K$ that is invariant to orthogonal transformations, we have 
\[
    (\wh f\circ \wh g)(X) = (\wh f \circ (Q \wh g))(X)
\]
for all fixed $Q\in \OO_{r\times r}$. Therefore, repeating the arguments in \eqref{eq_risk_decomp} and \eqref{bd_irre_error} gives 
\begin{align*}
    \cE(\wh f \circ \wh g) \le  
     \EE&\bigl[ (Y  - (\wh f\circ Q\wh g)(X) )^2 -  (Y  - ( \fh \circ Q \wh g)(X) )^2  \bigr]   \\  
    &+  (1+\theta) ~  \EE\left[\fh(Z) - (\fh\circ (Q\wh g))(X)   \right]^2  +   {1+\theta \over \theta}\|f_\cH - f^*\|_\rho^2
\end{align*}
for any $\theta\ge1$. 
Under \cref{ass_Lip_K}, we further have 
\[
   \EE\left[\fh(Z) - (\fh \circ (Q\wh g))(X) \right]^2    \le C_K^2 \|\fh\|_K^2 \EE\left[\|Z - Q\wh g(X)\|_2^2\right].   
\] 
Since $Q$ is taken arbitrarily, the claim follows from inspecting the proof of \cref{thm_risk}.
\end{proof}

\subsection{Proof of \texorpdfstring{\cref{cor_slow_rates}}{\texttwoinferior}}\label{slow_rate_delta}

\begin{proof}
    Under Assumptions  \ref{ass_bd_K} and \ref{ass_mercer}, we have
    \begin{align*}
    n R^2 (\delta) = \sum_{j=1}^\infty
    \min \left \{\delta, {\mu}_j  \right \}   ~ \le ~  \sum_{j=1}^\infty \mu_j \EE\left[ \phi^2_j(Z)   \right]
     ~ \overset{\eqref{eq_eigen_decomp}}{=}~  \EE\left[ 
        K(Z,Z) 
    \right] ~ \le ~  \kappa^2.
    \end{align*}
    Then $\delta=  \kappa/\sqrt{n}$  must satisfy 
    $R (\delta) \le \delta$. 
    Applying Lemma \ref{lem_subroot_monotone} yields
    $\delta_n \le \kappa/\sqrt{n}$. 
\end{proof}

\subsection{Proof of Corollary \ref{cor_K_linear}}\label{app_sec_cor_K_linear} 

 \begin{proof}
First note that  for any $\delta \ge  0$, 
\begin{align*}
 \bigl[R(\delta)\bigr]^2 &=   {1\over n} \sum_{j=1}^r  \min\left\{\delta,  \sigma_j      \right\} \le {r\over n} \delta.
\end{align*}
The statement then follows by using \cref{lem_subroot_monotone} to deduce 
\[
\delta_n \le \frac{r}{n}
\]
and invoking \cref{thm_risk}. 
 \end{proof}

\subsection{Proof of Corollary \ref{cor_K_poly}}\label{app_sec_cor_K_poly}
\begin{proof}
    First, we note that 
    \[
        R^2(\delta) =\frac{1}{n}\sum_{j=1}^\infty\min \big \{\delta , {\mu}_j\big \} \le \frac{1}{n} \sum_{j=1}^\i \mu_j ={1\over n} \sum_{j=1}^\i \mu_j \EE[\phi^2_j(Z)]  \overset{\eqref{eq_eigen_decomp}}{=} {1\over n}\EE[K(Z,Z)] \le {1\over n},
    \]
    where the last step follows from \cref{ass_bd_K} with $\kappa=1$. 
    By \cref{lem_subroot_comp}, we must have 
    \[
        \delta_n \le {1\over \sqrt n}.
    \]
    Now fix any $0<\delta \le 1/\sqrt{n}$ and     recall that
      $d(\delta)= \max \{j \ge 0: {\mu}_j\ge  \delta\}$ with $\mu_0 = \i$.
    We claim that 
    \begin{equation}\label{bd_d_delta}
         d(\delta) +1  \le C'  \delta ^{-1/(2\alpha)}
    \end{equation}
    for some constant $C' = C'(C, \alpha)>0$. 
    This follows trivially by using $\alpha >  1/2$ if $d(\delta) = 0$. For $d(\delta) \ge 1$, 
by definition of $d(\delta)$ and ${\mu}_j\le C j^{-2\alpha}$ for all $j\ge 1$, we have 
    \begin{align*}
        \delta \le  \mu_{d(\delta)} \le  C d(\delta)^{-2\alpha}
    \end{align*}
    from which \eqref{bd_d_delta} follows. 

    To derive the rate of $\delta_n$, 
    we have 
\begin{align*}
R^2(\delta) =\frac{1}{n }
\sum_{j=1}^\infty\min \big \{\delta , {\mu}_j\big \}
& \le \frac{\delta}{n } (d(\delta)+1)+\frac{1}{n }  \sum_{j=d(\delta)+2}^\infty \mu_j\\
&\le \frac{C'}{n}  \delta^{\frac{2\alpha-1}{2\alpha}}+ {C\over n} \sum_{j=d(\delta)+2}^\infty j^{-2\alpha} \\   
 & \le \frac{C'}{n}   \delta^{\frac{2\alpha-1}{2\alpha}}+ \frac{C }{n}  \int_{d(\delta)+1}^\infty t^{-2\alpha}d t \\
& =   \frac{C'}{n}  \delta^{\frac{2\alpha-1}{2\alpha}}+   \frac{C }{n}   \frac{1}{2\alpha-1} (d(\delta)+1)^{1- 2\alpha}  \\
& \le  
\frac{C_\alpha}{n} \delta^{\frac{2\alpha-1}{2\alpha}},
\end{align*}
where $C_\alpha$ is some constant  depending only on  $\alpha$. 
Then  solving 
$
\sqrt{C_\alpha/ n} ~ \delta^{\frac{2\alpha-1}{4\alpha}}=\delta
$
and applying Lemma \ref{lem_subroot_monotone} yield 
\begin{equation*}
    \delta_n ~ \le ~ C_\alpha' n^{-\frac{2\alpha}{2\alpha+1}}.
\end{equation*}
Invoking \cref{thm_risk} gives the bound (in order)
\[
  n^{-\frac{2\alpha}{2\alpha+1}}  \log\l(1/\eta \r) +\EE \|\wh g(X)-Z\|_2^2  + \frac{\log (1/\eta)}{n} .
\] 
This completes the proof. 
\end{proof}

\subsection{Proof of Corollary \ref{cor_K_exp}}\label{app_sec_cor_K_exp}
\begin{proof}
Fix any $0< \delta \le 1/\sqrt{n}$. Repeating the arguments in the proof of \cref{cor_K_exp} gives that the statistical dimension $d(\delta)$ under $\mu_j \le  \exp(-\gamma j)$ for $j\ge 1$ satisfies 
\begin{equation}\label{ub_d_exp}
    d(\delta) +1 \le {1\over \gamma} \log(C/\delta).
\end{equation}
Also by similar reasoning  in the proof of \cref{cor_K_exp},  we have
\begin{align*}
R^2(\delta) & ~ \le  ~    \frac{\delta}{n}(d(\delta) +1)+ \frac{1}{n}\int_{d(\delta)+1}^\infty \exp (-\gamma t) dt\\
 & ~ \le  ~    \frac{\delta}{ n} (d(\delta) +1)  +\frac{1}{\gamma n}\exp (-\gamma (d(\delta) +1) )  )\\
 & ~ \lesssim ~ {\delta\over n} \log{1/\delta}  &&\text{by \eqref{ub_d_exp}}.
\end{align*}
This leads to 
$$\delta_n \lesssim \frac{\log n}{n} .$$ 
Invoking \cref{thm_risk} gives the bound (in order)
\[
\frac{\log n  \log\l(1/\eta \r)}{n}  +\EE \|\wh g(X)-Z\|_2^2 +\frac{\log (1/\eta)}{n}  .
\]
This completes the proof. 
\end{proof}

\subsection{Proof of \texorpdfstring{\cref{thm_lb_excess_risk}}{\texttwoinferior}: the minimax lower bound}\label{app_sec_proof_lower_bound}

\begin{proof}
    To establish the claimed result, it suffices to prove each term in the minimax lower bounds separately.\\

\noindent
{\bf Step 1:} To prove the term $\delta_n$ in the lower bound, we consider the model $ X = A Z $ for some deterministic matrix $A \in\RR^{p\times r}$  and 
    $Y = f^*(Z)  +  \epsilon$ with  $\epsilon \sim  N(0,\sigma^2)$.
In this case, we observe that
\begin{equation}\label{low_bound_case_1}
    \begin{aligned}
         \inf_{h} \sup_{\theta \in \Theta}~  \EE_{\theta}  \l[ f^* (Z)- h (X) \r]^2  
    \ge ~ 
    &\inf_{h} \sup_{\theta \in \Theta}~  \EE_{\theta} \l[  f^* (Z)- (h \circ A)   (Z)  \r]^2 \\
    \ge ~ &
    \inf_{f} \sup_{f^*\in \cH_K, \|f^*\|_K\le 1}~  \|f-f^*\|_\rho ^2 
    \end{aligned}
\end{equation}
where the infimum is over all measurable function $f :\RR^r \to \RR$ constructed from $\{(Y_i, Z_i)\}_{i=1}^n$.

To proceed, 
for any $\delta>0$,  define the ellipse
\[
\cE(\delta):=  \l\{ \theta\in \RR^\i :~ \sum_{j=1}^\i  \frac{\theta_j^2}{\min\{ \delta, \mu_j\}} \le 1 \r\}. 
\]
By applying \cref{lem_entropy}, there exists a $(\sqrt{\delta}/2)$-separated collection of points $\{p^{(1)}, \ldots , p^{(m)}\}$ in $\mathcal{E}(\delta)$ under the metric $\|\cdot\|_2$
such that $\log m \ge d(\delta) / 64$.
We can therefore construct $f^{(1)}, \ldots, f^{(m)} \in \cH_K$
as 
\[
f^{(k)} = \sum_{j=1}^\i p^{(k)}_j  \phi_j . 
\]
By construction, we have  for any pair $(i,j)\in [m]\times [m]$, 
\[
\| f^{(i)}-f^{(j)} \|_\rho^2 = \|p^{(i)}-p^{(j)} \|_2^2, 
\]
implying that $f^{(1)}, \ldots, f^{(m)}$ is also $(\sqrt{\delta}/2)$-separated under the metric $\|\cdot\|_\rho$. 
For each $k\in [m]$, let $\rho^{(k)}$ correspond to the underlying distribution of the collected data $\{(Y_i,Z_i)\}_{i=1}^n$ when $f^*= f^{(k)}$ in \eqref{model}. 

Let $\text{KL}(\cdot ~\|~  \cdot)$ denote the  KL divergence between two distributions. For any $(i,j)\in [m]\times [m]$, 
observe that
\begin{align*}
\text{KL}\l(\rho^{(i)} ~\|~  \rho^{(j)}\r) &= \sum_{k=1}^{n} \EE\l[ \text{KL}\l( N(f^{(i)}(Z_k), \sigma^2) ~\|~  N(f^{(j)}(Z_k), \sigma^2)\r) \r] \\
&=  \frac{n}{2\sigma^2 } ~ \| f^{(i)} -f^{(j)} \|^2_{\rho}\\
&\le  \frac{2n\delta}{\sigma^2 }, 
\end{align*}
where the first equality follows from the chain rule of $\text{KL}$ divergence and the last second step follows from the fact that
\[
\text{KL}\l(\rho_1 ~\|~ \rho_2\r)= \frac{(\mu_1-\mu_2)^2}{2 \sigma^2}
\]
for any 
$ \rho_1 =  N(\mu_1, \sigma^2) $ and $\rho_2  =  N(\mu_2, \sigma^2) $. 
Applying  Fano’s lemma 
yields the lower bound as
\begin{align}\label{lb_step1}
\inf_{f} \sup_{f^* \in \cH_K, \|f^*\|_K\le 1}~  \|f-f^*\|_\rho^2 \ge \delta,
\end{align}
provided that
\begin{align}\label{condi_lb}
    \frac{ 2n\delta/\sigma^2+\log 2 }{ d(\delta)/64} <1. 
\end{align}
By the  definition of regular kernel, we have
$d(\delta_n) \ge c n\delta_n$ for some positive constant $c$. Rewriting 
$\delta' = c_1\delta_n$ for some sufficiently small $c_1 = c_1(\sigma^2)$, by $d(\delta_n)\ge 128 \log 2$, we find
\[
 \frac{d(c_1 \delta_n )}{64}  \ge  \frac{d(\delta_n)}{64} \ge \frac{c n \delta_n}{128} +\log 2,
\]
implying that \eqref{condi_lb} holds for $\delta'$ with  $c_1$ being sufficiently small. 
Then using  \eqref{lb_step1} with $ \delta = \delta'$
gives
\[
\inf_{f} \sup_{f^* \in \cH_K,\|f^*\|_K\le 1 } ~  \|f-f^*\|_\rho^2 \ge c_1  \delta_n .
\]
Together with \eqref{low_bound_case_1} yields
\[
  \inf_{h} \sup_{\theta \in \Theta}~\EE_\theta  \l[  f^* (Z)-  h   (X)  \r]^2    \ge c_1  \delta_n .
\] 

\medskip

\noindent {\bf Step 2:}
We prove the appearance of the second term in the lower bound by using the property the universality of the RKHS $\cH_K$ induced by some universal kernel $K$. 
Specifically, fix any $\vartheta>0$. For any $\beta \in \SS^{r-1}$, the unit sphere of $\RR^r$, define the function $f_\beta $ as $f_\beta (z) =z^\T \beta $. 
By the universality of the kernel $K$, for any  $\beta \in \SS^{r-1}$, there must exist 
a function $f_{\vartheta,\beta} \in \cH_K$ such that
\[
 \|f_{\vartheta,\beta} - f_\beta  \|_\rho \leq  \|f_{\vartheta,\beta} - f_\beta  \|_\i\le  \vartheta. 
\]
 For any $h:\RR^p \to \RR$, 
observe that
\begin{align}\label{eq1}  \nonumber
    \sup_{\theta \in \Theta}
    ~ \EE _\theta  \l[ f^* (Z)-   h  (X) \r]^2   
     &\ge   \sup_{\beta \in \SS^{r-1}} ~  \EE _\theta  \l[ f_{\vartheta , \beta} (Z)-  h  (X) \r]^2  \\\nonumber
  &  \ge  \sup_{\beta \in \SS^{r-1}} ~ \frac{1}{2} \EE _\theta \l[  f_\beta (Z) -  h  (X)\r]^2  - \EE _\theta \l[  f_\beta (Z) - f_{\vartheta,\beta} (Z) \r]^2  \\
  &  \ge  \sup_{\beta \in \SS^{r-1}}  ~ \frac{1}{2} \EE_\theta  \l[   f_\beta (Z)  -  h (X) \r]^2  - \vartheta^2 . 
\end{align}  
Since $\vartheta$ can be arbitrarily small and the leftmost side of \eqref{eq1} is independent of $\vartheta$, we obtain
\begin{align*}
    \inf_{h}  \sup_{\theta \in \Theta}~ \EE _\theta \l[  f^* (Z)-  h  (X) \r]^2  
    &~ \ge ~  \frac{1}{2} \inf_{h} \sup_{\beta \in \SS^{r-1}}  ~  \EE _\theta \l[  Z^\T\beta  -  h (X) \r]^2 \\
    &~ = ~ \frac{1}{2} \sup_{\beta \in \SS^{r-1}}  ~  \EE _\theta  \l[  Z^\T\beta  - \EE\b[ Z^\T\beta ~| ~X\b]\r]^2 \\
    & ~ =~   \frac{1}{2} \sup_{\beta \in \SS^{r-1}}  ~ \beta^\T \Sigma_{Z\mid X} \beta\\
    & ~ =~  \frac{ 1}{2}~ \| \Sigma_{Z\mid X} \|_\op
\end{align*}
where we write $\Sigma_{Z\mid X}  = \EE[\Cov(Z\mid X)]$.
Since model \eqref{model_X} with  $Z\sim N(0_r, \bI_r)$ and $W \sim N(0_p, \bI_p)$ implies 
\[
    \Sigma_{Z\mid X} = \bI_r - A^\T (A A^\T +\bI_p)^{-1} A = (A^\T A + \bI_r)^{-1}.
\]
Invoking \cref{ass_A} yields the second term $1/p$ in the lower bound, thereby completing the proof.
\end{proof}

Our proof of lower bound is based on the following lemma which is proved in  Lemma 4 of \cite{yang2017randomized} for the fixed design setting.  Since their argument can be used for the random design as well, we omit the proof. 

\begin{lemma} \label{lem_entropy}
For any $\delta>0$, 
there is a collection of $(\sqrt{\delta}/2)$-separated points   $\{p^{(1)}, \ldots , p^{(m)}\}$ in $ \mathcal{E} (\delta)$ under the metric $\|\cdot\|_2$ such that 
\[
\log m\ge \frac{d(\delta)}{64}. 
\]
\end{lemma}

\subsection{Proofs of \texorpdfstring{\cref{sec_general_loss}}{\texttwoinferior}: general convex loss functions} \label{app_sec_pf_thm_risk_lip}

For any $f:\cZ\to \RR$, let $\ell_f: \RR \times \cX \to \RR^+$ be the loss function  (relative to $\fh$)
$$
   \ell_f (y,x):=  \ell_{f\circ \wh g}(y,x) : = L(y,(f\circ \wh g)(x)) - L(y,(\fh \circ \wh g)(x)).
$$
For any measurable function $f: \cZ \to \RR$, 
the excess risk can be decomposed as
\begin{align*}
    \cE(f\circ \wh g) = & \EE\l[ L(Y, (f\circ \wh g)(X) ) - L(Y, f^*(Z))\r]   \\
    = ~ & \EE\l[ \ell_f(Y, X )\r]    +  \EE\l[ L(Y, (\fh\circ \wh g)(X) ) - L(Y, f^*(Z))\r] \\
    \le ~ & \EE\l[ \ell_f(Y, X )\r]    +  C_U  \EE\l[(\fh \circ \wh g)(X) - f^*(Z) \r]^2  &&\text{by \cref{assu_convex_and_smooth}}\\
    \le ~ & \EE\l[ \ell_f(Y, X )\r]    +  2 C_U \|\fh\|_K^2 \EE\Delta_{\wh g} +2 C_U \|\fh-f^*\|_\rho^2   &&\text{by \eqref{bd_irre_error} with $\theta =1$}.  
\end{align*}
The above error decomposition, with $\wh f$ in place of $f$, reveals that the  magnitude of the excess risk of $\wh f\circ \wh g$  is jointly determined by the estimation error $ \EE[ \ell_{\wh f}(Y, X )]$, the kernel-related latent error $\EE\Delta_{\wh g}$
and the approximation error $\|\fh-f^*\|_\rho^2 $. 

\subsubsection{Proof of \texorpdfstring{\cref{thm_risk_lip}}{\texttwoinferior}}
\begin{proof}[Proof of \cref{thm_risk_lip}]
 The proof largely follows the same arguments used in the proof of \cref{thm_risk}. 
Here we only mention some key changes. 
 
    First, a similar result for general loss functions to  \cref{used_lem_sec_A} follows from the convexity of $L(y,\cdot)$ by \cref{assu_lip_loss}.  Second, by using \cref{cor_bd_h_f_lip} stated and proved in the next section, for any $\eta\in (0,1)$,  with probability  at least $1-\eta$, the following holds uniformly over $f\in \cF_b$,
\[
\EE[\ell_f(Y,X)] ~ \le ~ 2\EE_n[\ell_f(Y,X)] + 2\lambda \|\fh\|_K^2.
\]
  This result replaces \eqref{lb_En_ell} in the proof of \cref{thm_risk}. 
For Case (1) considered in the proof of \cref{thm_risk}, the $\wh \alpha$ defined therein exists due to the convexity of $L(y,\cdot)$. So the proof follows analogously. For Case (2), the chain of inequalities in \eqref{eq_E_bd_wt_f} can be replaced by 
\begin{align*}
    & -\EE[\ell _{f_{\wt \alpha}}(Y,X)] \\
     &= - \EE\l[ L(Y,(f_{\wt \alpha} \circ \wh g)(X)) - L(Y,(\fh \circ \wh g)(X))\r] \\
      &= - \EE\l[ L(Y,(f_{\wt \alpha} \circ \wh g)(X)) - L(Y,f^*(Z))\r] + \EE\l[ L(Y,(\fh \circ \wh g)(X))-L(Y,f^*(Z) ) \r] \\
       &\le - C_L \EE\l[(f_{\wt \alpha} \circ \wh g)(X)-f^*(Z) \r]^2+   C_U \EE\l[(\fh \circ \wh g)(X)-f^*(Z) \r]^2 && \hspace{-0.58in}\text{by \cref{assu_convex_and_smooth}} \\
        &\le  C_U \EE\l[(\fh \circ \wh g)(X)-f^*(Z) \r]^2\\
&\le 2 C_U  \|\fh\|_K^2 ~ \EE\Delta_{\wh g} +{ 2C_U \|\fh-f^*\|_\rho^2} &&\text{by \eqref{bd_irre_error}. }
\end{align*}
Reasoning as in the proof of \cref{thm_risk} completes the proof. 
\end{proof}

\subsubsection{Uniform upper bounds of \texorpdfstring{$\EE[\ell_f(Y,X)]$}{\texttwoinferior} over \texorpdfstring{$f\in \cF_b$}{\texttwoinferior} }
The following lemma, which is an adaptation  of \cref{lem_bd_h_f} for general loss functions,   bounds from above $\EE[\ell_f(Y,X)]$ by its empirical counterpart 
$\EE_n[\ell_f(Y,X)]$, the fixed point $\delta_x$ of  $\psi_x(\cdot)$ in \eqref{def_psi_g_b_popu}, the kernel-related latent error $\EE \Delta_{\wh g}$  and the approximation error $\|\fh - f^*\|_\rho^2$, uniformly over $f\in \cF_b$. Recall  the definitions of $\xi_f$  from \eqref{def_g_f} and  $\Xi_f$  from \eqref{def_class_g_b}.

\begin{lemma}\label{lem_bd_h_f_lip}
Grant Assumptions  \ref{ass_f_H},  \ref{ass_bd_K}, \ref{assu_lip_loss} and \ref{assu_convex_and_smooth}. 
 Fix any $\eta\in(0,1)$. With probability at least $1-\eta$, for any $f\in \cF_b$, one has 
\begin{align*}
&\EE[ \ell_f (Y,X)] ~ \le  ~ 2\EE_n[\ell_f(Y,X)]+
c_1 C_L \max\left\{ C_\ell^2 C_L^{-2}, 1 \right\} ~  \delta_x   \\  
& \hspace{0.9 in}+ c_2 \l( C^2_\ell C_L^{-1} + C_\ell \kappa \|\fh\|_K\r)  \frac{\log (1/\eta)}{n}+ 
  4C_U \|\fh \|_K^2  \EE\Delta_{\wh g}  + 4 C_U \|\fh - f^*\|_\rho^2.
\end{align*} 
Here $c_1$ and $c_2$ are some absolute positive constants.
\end{lemma}
\begin{proof}
Define the scalar $V := 2C^2_\ell C_L^{-1}$ and the  function class 
\[
    \cF_\ell : = \Big\{\ell_f: f\in \cF_b\Big\}.
\]
Further define the functional $T:\cF_\ell \to \RR^+$ as  
\begin{align}\label{functional_Tf_lip}
T(\ell_f):= V \l( \EE\left[\ell_f(Y, X) \right]+   4 C_U  \l( \|\fh \|_K^2  \EE\Delta_{\wh g}  +  \|\fh - f^*\|_\rho^2\r)\r)
\end{align} 
as well as the local Rademacher complexity of $\cF_\ell$ as  
\begin{align*}
    \localcomp\Big(\big\{\ell_f \in \cF_\ell: T(\ell_f)\le \delta\big\}\Big). 
\end{align*}
The remaining proof is completed by splitting into $4$ steps. \\

\noindent{\it Step 1: Verification of the boundedness of $\cF_\ell$.}   
Using  some results in  Step (1) of the proof of \cref{lem_bd_h_f} yields
for every $f \in \cF_b$ and $(y,x)\in \RR \times \cX$, we have 
\begin{align}\label{la_eq22_lip}\nonumber
|\ell_f (y, x)|
&=\left| L(y,(f\circ \wh g)(x)) - L(y,(\fh \circ \wh g)(x)) \right|\\\nonumber
 &\le C_\ell \Bigl|(f\circ \wh g)(x) - (\fh \circ \wh g)(x)\Bigr|  &&\text{by \cref{assu_lip_loss}}
\\
 &\le 3C_\ell \kappa\|\fh\|_K 
 \end{align}  
so that the function class $\cF_\ell$ is uniformly bounded within $[-3C_\ell \kappa\|\fh\|_K, 3C_\ell \kappa\|\fh\|_K]$.\\

\noindent{\it Step 2: Bounding the variance of $\ell_f (Y,X)$.} Note that 
\begin{equation} \label{eq_lip_1}
    \begin{aligned}
 &  \Var\left(\ell_f (Y,X) \right)   \\
&\le  C_\ell^2 ~  \EE\l[ (f\circ \wh g)(X) -(\fh \circ \wh g)(X) \r]^2 
&&\text{by \cref{assu_lip_loss}}\\
&\le 2  C_\ell^2  \l( \EE\l[ (f\circ \wh g)(X) -f^*(Z)  \r]^2 +  \EE\l[ f^*(Z) - (\fh \circ \wh g)(X) \r]^2  \r)
\\
&\le   2  C_\ell^2 \EE\l[ (f\circ \wh g)(X) -f^*(Z)  \r]^2 +  
4 C_\ell^2 \|\fh\|_K ^2 \EE\Delta_{\wh g} + 4 C_\ell^2 \|\fh -f^*\|_\rho^2 &&\text{by \eqref{bd_irre_error}}. 
\end{aligned}
\end{equation}
Applying \cref{assu_convex_and_smooth} leads to 
\begin{equation} \label{eq_lip_2}
    \begin{aligned}
  & \EE\l[ (f\circ \wh g)(X) -f^*(Z)  \r]^2 
\\
&\le C_L^{-1}  ~ \EE[L(Y,(f\circ \wh g)(X)) - L(Y, f^*(Z) )  ] \\
&= C_L^{-1} \l( \EE[\ell_f(Y,X)   ]  + \EE[ L(Y,(\fh \circ \wh g)(X)) - L(Y,f^*(Z))  ]\r)\\
&= C_L^{-1} \l( \EE[\ell_f(Y,X)  ]  +  C_U \EE\l[ f^*(Z) - (\fh \circ \wh g)(X) \r]^2  \r)\\
&= C_L^{-1} \l( \EE[\ell_f(Y,X)  ]  + 2 C_U \|\fh\|_K ^2 \EE\Delta_{\wh g} +2C_U \|\fh-f^*\|_\rho^2   \r)  &&\text{by \eqref{bd_irre_error}}. 
\end{aligned}
\end{equation}
Together with \eqref{eq_lip_1} yields
\begin{align*}
   \Var\left(\ell_f (Y,X) \right)   \le   V \left( \EE\left[ \ell_f (Y,X)\right] + 4 V  C_U   \l( \|\fh \|_K^2  \EE\Delta_{\wh g}  +  \|\fh - f^*\|_\rho^2\r)  \right) \overset{\eqref{functional_Tf_lip}}{=} T(\ell_f).
\end{align*}
\\

\noindent{\it Step 3: Choosing a suitable sub-root function.} 
By inspecting the argument of proving \eqref{eq_lip_1} and \eqref{eq_lip_2}, we have that for any $f\in \cF_b$,
\begin{align*}
{V C_L \over 2} \EE[ \xi_f^2(X)] =& {V C_L \over 2} \EE\l[ (f\circ \wh g)(X) -(\fh \circ \wh g)(X) \r]^2 
\\
\le& V    \left(\EE[\ell_f(Y,X) ]  + 4 C_U\|\fh \|_K^2  \EE\Delta_{\wh g} +4C_U \|\fh-f^*\|_\rho^2 \right) \overset{\eqref{functional_Tf}}{=}  T(\ell_f)  
\end{align*}
so that for any $\delta \ge 0$,
\begin{align}\label{lem4_eq0_lip}
\localcomp\left(\big\{\ell_f \in \cF_\ell: T(\ell_f)\le \delta\big\}\right) \le  \localcomp\left(\big\{ \ell_f\in \cF_\ell:  V C_L  \EE[\xi^2_f (X)] \le 2 \delta\big\}\right). 
\end{align}
By \cref{assu_lip_loss}, 
for any $f_1,f_2\in \cF_{b}$ and $(y,x)\in  \RR\times \cX$,  we have
\begin{align*}
   \left| \ell_{f_1}(y,x)- \ell_{f_1}(y,x)\right|
= ~ &  \left| L(y,(f_1 \circ \wh g)(x)) - L(y,(f_2 \circ \wh g)(x)) \right| 
\\
\le ~  &  C_\ell ~  
\bigl |\xi_{f_1}(x)-\xi_{f_2}(x) \bigr|
\end{align*}
which implies that 
$\ell_f(Y_i,X_i)$ is $(C_\ell)$-Lipschitz in $\xi_{f}(X_i)$ for all $ i\in [n]$.
Then, by repeating the arguments of proving \eqref{lem4_eq1},  and together with \eqref{lem4_eq0_lip}, we conclude
\begin{align*}
   \localcomp\left(\big\{\ell_f \in \cF_\ell: T(\ell_f)\le \delta\big\}\right) 
 \le ~ &    C_\ell ~ \localcomp\left(\big\{ \xi_f\in \Xi_b:  V C_L  \EE[\xi^2_f (X)] \le 2 \delta\big\}\right) \\
 = ~ &   C_\ell ~   \psi_x\left(\frac{2\delta}{ V C_L} \right)
\end{align*}
so that  by choosing $\psi(\delta)= V  C_\ell   ~   \psi_x\left(\frac{2\delta}{ V C_L} \right)$ (Lemma \ref{lem_subroot_comp} ensures it is sub-root in $\delta$), 
for every $\delta\ge \delta_\star$, 
\[
\psi(\delta) \ge   V  ~  \localcomp\left(\big\{\ell_f \in \cF_\ell: T(\ell_f)\le \delta\big\}\right) . 
\]
\\

\noindent{\it Step 4: Bounding the fixed point of the sub-root function $\psi(\cdot)$.} 
Write the fixed point of $\psi(\cdot)$ as  $\delta_\star$. 
    By using \cref{lem_subroot_comp}, $\delta_\star$  
 can be bounded as
\[
\delta_\star \le \max\left\{4 C_\ell^2C_L^{-2}, 1\right\} \frac{VC_L}{2} \delta_x.  
\] 

Finally, the proof is completed by invoking Lemma \ref{supp_lem_3} with $-a= b=3C_\ell \kappa\|\fh\|_K,  B=V$, $L= 4 C_U  \l( \|\fh \|_K^2  \EE\Delta_{\wh g}  +  \|\fh - f^*\|_\rho^2\r) $ and $\theta = 2$.
\end{proof}

Combining Lemmas  \ref{lem_bd_td_delta}  and \ref{lem_bd_h_f_lip} readily yields the following corollary where we bound from above $\delta_x$ by $\delta_{n,\wh g}$, with the latter given in \eqref{def_delta_n_hatg}.

\begin{corollary} \label{cor_bd_h_f_lip}
  Grant Assumptions \ref{ass_f_H}--\ref{ass_mercer},  \ref{assu_lip_loss} and \ref{assu_convex_and_smooth}.    Fix any $\eta\in(0,1)$. With probability at least $1-\eta$, for any $f\in \cF_b$, one has 
    \begin{align*}
&\EE[ \ell_f (Z,X)] ~ \le  ~ 2\EE_n[\ell_f(Y,X)]+
c_1 C_L \max\left\{ C_\ell^2 C_L^{-2} , 1\right\} ~  V_K\delta_{n,\wh g}  \\  
& \hspace{0.85 in}+ c_2 \l( C^2_\ell C_L^{-1} + C_\ell \kappa \|\fh\|_K\r)  \frac{\log (1/\eta)}{n}+ 
  4C_U \|\fh \|_K^2  \EE\Delta_{\wh g}  + 4 C_U \|\fh - f^*\|_\rho^2.
\end{align*} 
\end{corollary}

\section{Auxilliary Lemmas}\label{app_sec_auxi}
In this section, we summarize the supporting lemmas used in our proof. 

\subsection{Definition and properties of the sub-root functions}\label{app_sec_subroot}

\begin{definition}
    A function $\psi: [0,\i)\to  [0,\i)$ is a sub-root function if it is nonnegative, nondecreasing, and if $\delta\mapsto \psi(\delta)/\sqrt{\delta}$ is nonincreasing for $\delta>0$. 
\end{definition}
A useful property about the sub-root function  defined above 
is stated as follows. 

\begin{lemma}[Lemma 3.2 in \cite{bartlett2005local}]\label{lem_subroot_monotone}
If $\psi:[0, \infty) \to [0, \infty)$ is a nontrivial\footnote{Functions other than the constant function $\psi\equiv 0$. } 
sub-root function, 
then it is continuous on $[0, \infty)$ and the equation $\psi(\delta)=\delta$ has a unique positive solution $\delta_\star$. Moreover, for all $\delta>0, \delta\geq \psi(\delta)$ if and only if $\delta_\star \leq \delta$. Here, $\delta_\star$ is referred to as the fixed point. 
\end{lemma}

The following lemma lists some useful properties of sub-root functions and corresponding fixed points, whose proof can be found in \cite{duan2021risk}. 
\begin{lemma}\label{lem_subroot_comp}

If  $\psi:[0, \infty) \rightarrow[0, \infty)$ is  nontrivial sub-root function and $ \delta_\star$ is its positive fixed point, then
\begin{itemize}
    \item[(i)] For any $c>0, \widetilde{\psi}(\delta):=c \psi\left(c^{-1} \delta\right)$ is sub-root and its positive fixed point $\wt\delta_{\star}$ satisfies $\wt\delta_{\star}=c \delta_\star$.
    \item[(ii)] For any $C>0, \widetilde{\psi}(\delta):=C \psi(\delta)$ is sub-root and its positive fixed point $\wt\delta_{\star}$ satisfies $\wt\delta_{\star} \leq(C^2 \vee 1) \delta_\star$.
\end{itemize} 
\end{lemma}

\subsection{Local Rademacher complexity}\label{app_sec_auxi1}

The following lemma indicates that  
one can take the local Rademacher complexity as a sub-root function.
\begin{lemma}\label{subroot}
If the class $\mathcal{F}$ is star-shaped,\footnote{A function class $\mathcal{F}$ is said star-shaped if for any $f\in \mathcal{F}$ and $\alpha\in [0,1]$, the scaled function $\alpha f$ also belongs to  $\mathcal{F}$.} and $T: \mathcal{F} \rightarrow \mathbb{R}^{+}$ is a (possibly random) function that satisfies $T(\alpha f) \leq \alpha^2 T(f)$ for any $f \in \mathcal{F}$ and any $\alpha \in[0,1]$, then the (random) function $\psi$ defined for $\delta\ge  0$ by
$$
\psi(\delta)= \emlocalcomp \left(\{f\in \mathcal{F}: T(f) \leq \delta\}\right)
$$
is sub-root.
Furthermore, 
the (random) function $\psi'$ defined for $\delta\ge  0$ by
$$
\psi'(\delta)= \localcomp  \left(\{f\in \mathcal{F}: T(f) \leq \delta\}\right)
$$
is also sub-root. 
\end{lemma}
\begin{proof}
See the proof of Lemma 3.4 in \cite{bartlett2005local}.
\end{proof}
Recall that
$$
\cF_b:=\left\{f\in\cH_K:  
\ \|f-\fh\|_K\le 3\|\fh\|_K\right\}. 
$$
\begin{lemma}[Corollary of Lemma \ref{subroot}]\label{subroot_1}
The function
 \[
\psi(\delta)=
\localcomp\Big(\big\{x \mapsto (f\circ \wh g)(x)-(f_\cH\circ \wh g)(x): f\in \cF_b, ~ \EE\l[(f\circ \wh g)(X)-(f_\cH\circ \wh g)(X)\r]^2\le \delta\big\}\Big)
 \]
 is sub-root. 
\end{lemma}
\begin{proof}
By translating the function class, it suffices to show that 
\[
\psi(\delta)=\localcomp\Big(\big\{x\mapsto (f\circ \wh g)(x): f\in \cH_K, \|f\|_K\le 3\|\fh \|_K, \EE\l[(f\circ \wh g)(X)\r]^2 \le \delta\big\}\Big)
 \]
 is sub-root.  To  show this, we  take $T(f)=\EE\l[(f\circ \wh g)(X)\r]^2 $. We find  \[
 T(\alpha f)=\EE\l[( \alpha f \circ \wh g)(X)\r]^2= \alpha^2\EE\l[(f\circ \wh g)(X)\r]^2 = \alpha^2 T(f).
 \]
 Moreover, it is clear that the function class $\{x\mapsto (f\circ \wh g)(x): f\in \cH_K, \|f\|_K\le 3\|\fh\|_K\}$ is star-shaped. Invoking Lemma \ref{subroot}
completes the proof.
\end{proof}




The following lemma is a variant of Theorem 3.3 in \cite{bartlett2005local}, which is useful for proving uniform concentration and can be found in \cite{duan2021risk}. 

\begin{lemma}[Corollary of Theorem 3.3 in \cite{bartlett2005local}]\label{supp_lem_3}
Let $\mathcal{F}$ be a class of functions with ranges in $[a, b]$ and $U_1,...,U_n$ be the i.i.d. copies of some random variable $U$.
Assume that there are some functional $T: \mathcal{F} \rightarrow \mathbb{R}^{+}$ and some constants $B$ and $L$ such that for every $f \in \mathcal{F}$, 
\begin{align*}
\operatorname{Var}[f(U)] \leq T(f) \leq B(\EE f(U)+L). &&\text{(Variance condition)}
\end{align*}
Let $\psi$ be a sub-root function and let $\delta_\star$ be the fixed point of $\psi$. Assume that $\psi$ satisfies, for any $\delta \geq \delta_\star, \psi(\delta) \geq B \localcomp(\{f \in \mathcal{F}: T(f) \leq \delta\})$. Then,  for any $\theta>1$ and $\eta\in (0,1)$, with probability at least $1-\eta$,
$$
\EE f(U) \leq \frac{\theta}{\theta-1} \EE_n f(U)
+\frac{c_1 \theta}{B} \delta_\star+\big(c_2(b-a)+c_3 B \theta\big) \frac{\log (1 / \eta)}{n}+\frac{L}{\theta-1}, \quad \text { for any } f \in \mathcal{F} .
$$
Also, with probability at least $1-\eta$,
$$
\EE_n f(U) \leq \frac{\theta+1}{\theta} \EE f(U)+\frac{c_1 \theta}{B} \delta_\star+\big(c_2(b-a)+c_3 B \theta\big) \frac{\log (1 / \eta)}{n}+\frac{L}{\theta}, \quad \text { for any } f \in \mathcal{F}.
$$
Here, $c_1, c_2, c_3>0$ are some universal constants.  
\end{lemma}

The following lemma is contained in \cite{ledoux2013probability}, which allows us to utilize the symmetrization technique for the family of Lipschitz functions. 
\begin{lemma}[Ledoux–Talagrand contraction inequality] \label{supp_lem_5}
For any set $\mathcal{T}\subseteq \RR^d$,  let $\{\psi_j: \RR \rightarrow\RR,j=1,...,d\}$ be any family of $B$-Lipschitz functions.
Then, we have
\begin{align*}
\EE\left( \sup_{\theta\in \mathcal{T}}\sum_{j=1}^d\varepsilon_j\psi_j(\theta_j)\right)
\le  B \EE\left( \sup_{\theta\in \mathcal{T}}\sum_{j=1}^d \varepsilon_j\theta_j\right).
\end{align*}
\end{lemma}

\subsection{Auxiliary concentration inequalities}
This section provides some useful concentration or tail inequalities that are used in the proof of the Appendix.

The first lemma states a variant of the classic Bernstein's inequality for the sum of random variables from \cite{bernstein1946theory}.

\begin{lemma}
\label{bern_rv}
Let $\xi_1, \ldots, \xi_n$ be a sequence of independent and identically distributed random variables taking values in $\RR$ with zero mean. If there exist some $H, S>  0 $ such that $\max _{i\in [n]}\left|\xi_i\right| \leq H$ almost surely and ${\EE} \xi_i^2 \leq S$ for any $i \in [n]$, then for any $\eta\in (0,1)$,  the following holds with probability at least $1-\eta$,
$$
\left|\frac{1}{n} \sum_{i=1}^n \xi_i\right| \leq \frac{2 H }{3 n} \log \frac{2}{\eta}+\sqrt{\frac{2 S }{n}\log \frac{2}{\eta}}.
$$
\end{lemma}

\begin{proof}
    See, for instance, Proposition 2.14 in \cite{wainwright2019high} for the proof of  Bernstein's inequality. The form of Bernstein's inequality in \cref{bern_rv} 
then follows after some simple algebra. 
\end{proof}

The following concentration inequality is used for bounding the sum of random variables taking values in real separable Hilbert space, which is based on the result of  Theorem 3.3.4 in  \cite{yurinsky2006sums}. See also \cite{caponnetto2007optimal,rudi2015less,rudi2017generalization,lin2020convergences}.
\begin{lemma}\label{lem_concen}
Let $\cF$  
be a real separable Hilbert space equipped with norm $\|\cdot\|_v$ and $\xi \in \cF$ be a random element. Assume that there exist some constants $H,S>0$ such that
\begin{align}\label{lem1_eq1}
    \EE\B [ \big \|\xi-\EE\xi  \big \|_{v}^k  \B ] \leq \frac{1}{2} k ! S H^{k-2}, \qquad \ \text{for any} \ k \in \mathbb{N} \ \text{and} \ k\geq 2.
\end{align}
 Let $\xi_1, \ldots, \xi_n \in \cF$ be i.i.d. copies of $\xi$. 
Then for any  $\eta \in (0,1)$, the following holds with probability at least $1-\eta$,
\begin{align*}
    \l \|\frac{1}{n} \sum_{i=1}^n \xi_i- \EE\xi \r\|_{v} \leq \frac{2H}{n}\log \frac{2}{\eta}+ \sqrt{ \frac{2 S}{{n}}\log \frac{2}{\eta} },
\end{align*}
Particularly,  \eqref{lem1_eq1} is satisfied if 
\begin{align*}
  \|\xi\|_v \leq  \frac{H}{2} \quad \text{almost surely} \qquad \text{and}\qquad \EE \b[\|\xi\|^2_{v}\b]\leq S.
\end{align*}
\end{lemma}


The following is Bernstein's inequality for the sum of random self-adjoint operators, which was proved by \cite{minsker2017some}. See also   \cite{rudi2015less,rudi2017generalization,lin2020convergences}.

\begin{lemma}\label{lemd3}
Let $\cF$ be a separable Hilbert space and let $\xi_1, \ldots, \xi_n$ be a sequence of independent and identically distributed self-adjoint random operators on $\cF$ endowed with the operator norm $\|\cdot\|_\op$. 
Assume that there exists $\EE \xi_i=0$ and $\|\xi_i\|_\op \leq H$ almost surely for some $H>0$ for any $i \in [n]$. 
Let $S$ be a positive operator such that $\EE[\xi_i^2] \leq S$. Then for any  $\eta \in(0,1)$, the following  holds with probability at least $1-\eta$,  
\begin{align*}
\l\|\frac{1}{n} \sum_{i=1}^n \xi_i\r\|_\op  \leq \frac{2 H \beta}{3 n}+ \sqrt{\frac{2\|S\|_\op \beta}{n}},
\end{align*}
where 
$\beta=\log \l(\frac{2 \tr (S)}{\|S\|_\op \eta}\r)$. 
\end{lemma}

The following lemma states a tail inequality for quadratic forms of a sub-Gaussian random vector, which comes from Lemma 30 in \cite{hsu2014random}. 

\begin{lemma} \label{tail_eq_qua_gaussian}
Let $\xi$ be a random vector taking values in $\mathbb{R}^n$ such that for some $c \geq 0$, 
$$
\mathbb{E}[\exp (\langle u, \xi\rangle)] \leq \exp \left(c\|u\|_2^2 / 2\right), \quad \forall u \in \mathbb{R}^n. 
$$
For any  symmetric positive semidefinite matrices $A \succeq 0$, and any $t>0$,
$$
\PP\left[\xi^{\top} A \xi>c\left(\tr(A)+2 \sqrt{\tr\left(A^2\right) t}+2\|A\|_\op t\right)\right] \leq \mathrm{e}^{-t}. 
$$
\end{lemma}

\end{document}